\documentclass[11pt,reqno]{amsart}

\usepackage{amsmath,amsthm,amscd,amssymb,amsfonts, amsbsy}
\usepackage{latexsym}
\usepackage{txfonts}
\usepackage{exscale}
\usepackage{mathrsfs}
\usepackage{esint} 
\usepackage{enumitem} 
\usepackage{color}
\usepackage[perpage]{footmisc}

\usepackage{relsize}
\usepackage[normalem]{ulem}

\usepackage{graphicx} 
\def\Yint#1{\mathchoice
    {\YYint\displaystyle\textstyle{#1}}%
    {\YYint\textstyle\scriptstyle{#1}}%
    {\YYint\scriptstyle\scriptscriptstyle{#1}}%
    {\YYint\scriptscriptstyle\scriptscriptstyle{#1}}%
      \!\iint}
\def\YYint#1#2#3{{\setbox0=\hbox{$#1{#2#3}{\iint}$}
    \vcenter{\hbox{$#2#3$}}\kern-.51\wd0}}
\def\longdash{{-}\mkern-3.5mu{-}} 
\def\tiltlongdash{\rotatebox[origin=c]{15}{$\longdash$}}
\def\fiint{\Yint\tiltlongdash}

\usepackage[
  hmarginratio={1:1},     
  vmarginratio={1:1},     
  textwidth=400pt,			  
	textheight=580pt,        
  heightrounded,          
]{geometry}

\usepackage[utf8]{inputenc}

\usepackage[
	colorlinks,
	citecolor=blue,
	linktoc=all,
	pagebackref,
	hypertexnames=false
]{hyperref}

\usepackage{appendix}

\definecolor{br}{rgb}{1, 0.4,0}

\newcommand{\res}{\hbox{ {\vrule height .24cm}{\leaders\hrule\hskip.1cm} } }

\allowdisplaybreaks

\numberwithin{equation}{section}

\theoremstyle{plain}
\newtheorem{theorem}[equation]{Theorem}
\newtheorem{prop}[equation]{Proposition}
\newtheorem{corollary}[equation]{Corollary}
\newtheorem{lemma}[equation]{Lemma}

\theoremstyle{definition}
\newtheorem{defn}[equation]{Definition}

\theoremstyle{remark}
\newtheorem{remark}[equation]{Remark}

\numberwithin{equation}{section}

\newcommand{\RR}{{\mathbb{R}}}
\newcommand{\NN}{{\mathbb{N}}}

\newcommand{\CC}{{\mathscr{C}}}
\newcommand{\HH}{\mathfrak{H}}

\newcommand{\E}{\mathcal{E}}

\newcommand{\sH}{~d\mathcal{H}^{d-1}}

\renewcommand{\emptyset}{\mbox{\textup{\O}}}

\DeclareMathOperator{\divg}{div}
\DeclareMathOperator{\spt}{spt}

\DeclareMathOperator{\tr}{tr}
\DeclareMathOperator{\dist}{dist}

\DeclareMathOperator{\Id}{Id}
\DeclareMathOperator{\I}{I}

\DeclareMathOperator{\spn}{span}

\newcommand{\pO}{\partial\Omega}
\newcommand{\pD}{\partial D}
\newcommand{\pR}{\partial \RR^d_+}

\newcommand{\St}{\mathcal{S}}
\newcommand{\Ct}{\mathcal{C}}
\newcommand{\Nt}{\mathcal{N}}

\reversemarginpar

\begin{document}

\allowdisplaybreaks

\title[Boundary unique continuation on Dini domains]{Boundary unique continuation on $C^1$-Dini domains \\ and the size of the singular set}

\author{Carlos Kenig}
\address{Carlos Kenig
\\ 
Department of Mathematics
\\
University of Chicago
\\
Chicago, IL 60637, USA}
\email{cek@math.uchicago.edu}

\author{Zihui Zhao}
\address{Zihui Zhao
\\ 
Department of Mathematics
\\
University of Chicago
\\
Chicago, IL 60637, USA}
\email{zhaozh@uchicago.edu}

\thanks{The first author was supported in part by NSF grant DMS-1800082, and the second author was partially supported by NSF grant DMS-1902756.}

\begin{abstract}
	Let $u$ be a harmonic function in a $C^1$-Dini domain $D$ such that $u$ vanishes on a boundary surface ball $\pD \cap B_{5R}(0)$. We consider an effective version of its singular set (up to boundary) $\St(u):=\{X\in \overline{D}: u(X) = |\nabla u(X)| = 0\} $ and give an estimate of its $(d-2)$-dimensional Minkowski content, which only depends on the upper bound of some modified frequency function of $u$ centered at $0$. Such results are already known in the interior and at the boundary of convex domains, when the standard frequency function is monotone at every point. The novelty of our work on Dini domains is how to compensate for the lack of such monotone quantities at boundary as well as interior points.
\end{abstract}

\maketitle

\tableofcontents

\section{Introduction}



In this paper, we study the size of the singular set when $u$ is a harmonic function in the Dini domain $D\subset \RR^d$, such that $u$ vanishes on an open set of the boundary. (Dini domains are $C^1$ domains given locally by graphs of functions whose gradient has a Dini modulus of continuity, see Definition \ref{def:Dini}.) More precisely, suppose $u = 0$ on the surface ball $B_{5R}(0) \cap \pD $ with $0\in \pD$, we show that inside a smaller ball $B_{\frac{3}{20} R}(0)$, the singular set $\mathcal{S}(u) := \{X \in \overline{D} : u(X) = |\nabla u(X)| = 0 \}$ is $(d-2)$-rectifiable, and moreover its $(d-2)$-dimensional Minkowski content is bounded by a constant depending on the upper bound of $N_0(4R)$, the frequency function of $u$ centered at $0$.

\begin{theorem}\label{thm:main}
	Let $R, \Lambda>0$ be fixed. There exists $r_c>0$ such that for any $(u, D) \in \HH(R, \Lambda)$ (see Definition \ref{def:function}) and any $r_0 \in (0, r_c)$, the effective singular set $\widetilde{\Ct}_{r_0}(u) \cap \Nt(u)$ (see the definitions in Section \ref{sec:prelim}) has the $(d-2)$-dimensional Minkowski bound
	\[ \mathcal{M}^{d-2, *}\left(\widetilde{\Ct}_{r_0}(u) \cap \mathcal{N}(u) \cap B_{\frac{R}{10}}(0) \right) \leq C(d, R, \Lambda). \]
	In particular the singular set satisfies
	\[ \mathcal{M}^{d-2, *}\left(\mathcal{S}(u) \cap B_{\frac{R}{10}}(0) \right) \leq C(d, R, \Lambda), \]
	and $\mathcal{S}(u) \cap B_{\frac{R}{10}}(0)$ is $(d-2)$-rectifiable.
\end{theorem}

Let us first give some historical background for the result in Theorem \ref{thm:main}. It seems that such results originate in problems in optimization and control theory, as is described in \cite{Wec, SW}, where the authors proved that when $D$ is a smooth domain $\mathcal{S}(u)\cap \pD$ has zero $(d-1)$-dimensional Hausdorff measure. The authors of \cite{SW} were inspired by the following problems:
\begin{quote}
	Consider a harmonic function $u$ in $D$, with $0\leq u\leq 1$ in $D$. Fix two points $X_0$ and $X_1$ in $D$. We wish to minimize $u(X_1)$ among all such $u$ with $u(X_0) = \frac12$.
\end{quote}
and
\begin{quote}
	\uline{Conjecture (\cite{SW}): The Bang-Bang Property}. There exists a unique minimizer $\tilde{u}$ to the above problem, and moreover $\tilde{u}$ is ``Bang-Bang'', namely $\tilde{u}|_{\pD} = \chi_F$ for some $F\subset \pD$.
\end{quote}
%
%
It can be proven (see \cite{SW, KenLN}) that in order to establish the conjecture in a domain $D$, it suffices to show that, with $u$ as in Theorem \ref{thm:main}, the set $\mathcal{S}(u) \cap \pD \cap B_{\frac{R}{10}}(0)$ has zero $(d-1)$-dimensional Hausdorff measure, as was proved for smooth domains in \cite{Wec, SW}. 

Another problem that arises in control theory, this time in the study of exact boundary controllability, is the following:
\begin{quote}
	Suppose that $u$ is the (Dirichlet) Laplacian eigenfunction in $D$ with eigenvalue $\lambda>0$ (that is to say $-\Delta v =\lambda v$ in $D$ and $v = 0$ on $\pD$). Assume that the normal derivative of $v$ is zero on $E\subset \pD$, with $\mathcal{H}^{d-1}(E) >0$. Is $v$ identically zero?
\end{quote}
By considering the harmonic lifting $u(x,t) = e^{-\sqrt{\lambda} t} v(x)$ on $D\times \RR$ of $v$, the above problem reduces again to the fact that for a non-zero $u$, the set $\mathcal{S}(u) \cap \left( \pD \times \RR \right)$ has zero surface measure on $\pD \times \RR$. So this problem again has a positive answer in smooth domains by \cite{Wec, SW}. The result was extended to $C^{1,1}$ domains by Lin in \cite[Theorem 2.3]{Lin}, using a reflection argument and the interior unique continuation result of \cite{AKS}. In fact, Lin was able to go further, using Federer's dimension reduction argument, and show that $\mathcal{S}(u)  \cap \pD$ has Hausdorff dimension $(d-2)$.

Another historical source for problems of this kind is a classical question of L. Bers. Consider a harmonic function $u$ in the upper half-plane $\RR^d_+$, which is in $C^1(\overline{\RR^d_+})$. Assume that there exists $E\subset \RR^{d-1} = \partial \RR^d_+$ such that $u = 0$ and $\partial_d u = 0$ on $E$. Is it true that if $u$ is not identically zero, the Lebesgue measure of $E$ must be zero? It is not hard to see that for planar domains this is indeed the case because the logarithm of the modulus of the gradient of a harmonic function is sub-harmonic. However for $d>2$, Wolff \cite{Wol} showed this to be false, and Bourgain and Wolff \cite{BW} showed this to be false even under the assumption that $u\in C^{1,\alpha}(\overline{\RR^d_+})$ for some $\alpha>0$.
Given this failure, it then became of interest to study the ``intermediate'' case where, in the notation of Theorem \ref{thm:main}, $u=0$ on an \textit{open} set $B_{5R}(0) \cap \pD$ of the boundary and attempt to show that $\mathcal{S}(u)  \cap \pD \cap B_{\frac{R}{10}}(0)$ has zero $(d-1)$-dimensional Hausdorff measure unless $u\equiv 0$. As mentioned earlier, this was proved in \cite{Wec, SW} in smooth domains and in \cite{Lin} for $C^{1,1}$ domains, where it was also proved that it was a $(d-2)$-dimensional set.

The next progress in these problems (always assuming that $u$ vanishes on a surface ball) was due to Adolfsson, Escauriaza and Kenig \cite{AEK} (see also Kenig-Wang \cite{KW} for an alternative proof) who proved that for convex domains $\mathcal{S}(u) \cap \pD$ has zero $(d-1)$-dimensional Hausdorff measure. This was then followed by works of Adolfsson-Escauriaza \cite{AE} and Kukavica-Nystr\"om \cite{KN}, who proved (using different methods) the result for Dini domains. In fact, Adolfsson and Escauriaza \cite{AE} proved that for Dini domains, the set $\mathcal{S}(u) \cap \pD$ has Hausdorff dimension $(d-2)$, using the Federer dimension reduction argument as in Lin \cite{Lin}. Recently, after many years without progress, Tolsa \cite{Tol} proved that for all $C^1$ domains (and even Lipschitz domains with small constant) and non-identically-zero harmonic functions $u$, the set $\mathcal{S}(u) \cap \pD$ has zero $(d-1)$-dimensional Hausdorff measure. The proof of this result was influenced by very important new methods developed by Logunov and Malinnikova in \cite{Logu, Logl, LM}. 

In the interior, using the expansion of harmonic functions by homogeneous harmonic polynomials, Han \cite{Han} proved the $(d-2)$-rectifiability of $\St(u)$; moreover Naber and Valtorta \cite{NVCS} estimated the $(d-2)$-dimensional Minkowski content of an effective version of $\St(u)$ and $\Ct(u)$.
These two results also hold for elliptic equations with Lipschitz coefficients, by a perturbation of the harmonic case.
The same argument does not seem to apply easily to the boundary. 
Instead, inspired by the quantitative stratification method developed by Naber and Valtorta (see \cite{NVRR, NVAH}) to study 
regularity results in geometric variational problems, McCurdy \cite{Mc} took up the study of $\mathcal{S}(u)$ and proved the analogue of Theorem \ref{thm:main} for convex domains, thus establishing the $(d-2)$-rectifiability of $\mathcal{S}(u)$ and a bound on its $(d-2)$-dimensional Minkowski content. Recall that it had been established before using Federer's dimension reduction argument that $\mathcal{S}(u) \cap \pD$ has Hausdorff dimension at most $(d-2)$ (see \cite{Lin, AE} for $C^{1,1}$ domains and Dini domains, respectively), but it was not even known that $\mathcal{H}^{d-2}(\mathcal{S}(u) \cap \pD) <+\infty$. So it was a big improvement to prove $(d-2)$-rectifiability and find an upper bound on the Minkowski content, not just for the singular set restricted to the boundary, but also as the singular set approaches the boundary from the interior. For the difference between the Minkowski content and Hausdorff measure, see Definition \ref{def:Minkowski}.

Our approach is also based on the techniques introduced by Naber and Valtorta in \cite{NVRR}. These techniques have been found useful in many problems in geometric measure theory and geometric analysis, as long as symmetry in the problem can be somewhat quantified by a monotone quantity (\cite{codim4, NVV, FS, EE} to name just a few examples). For the convenience of the reader, we highlight the main differences between our argument with that of Naber and Valtorta as well as other prior work using the method.
\begin{itemize}[topsep=0pt]
	\item Unlike the interior case (see \cite{NVCS}) or the boundary case for convex domains (see \cite{Mc}), the standard Almgren's frequency function is not monotone for boundary points in Dini domains. Instead fixing each boundary point $X\in \pD$, we use the transformation map $\Psi_X: \Omega_X \to D$ introduced in \cite{AE} so that the domain becomes \textit{almost convex} at that point. This way we get a monotone quantity for $X$, which measures how close the harmonic function $u$ is to be homogeneous near $X$.
	\item Not only do we study the singular set on the boundary, we also want to study how the singular set approaches the boundary from the interior. To bring interior points (which are very close to the boundary) in the picture, we also study Almgren's frequency function $N(X, r)$ for interior points, especially when the scale $r$ is big so that $B_r(X)$ intersects the boundary. The standard frequency function $N(X, r)$ is monotone increasing up to some critical scale $r_{cs}(X)$, which depends on $\dist(X, \pD)$. Outside of this monotonic interval, we replace $N(X, r)$ by the frequency function (of the same scale) centered at a boundary point closest to $X$, and justify that we only need to pay a small price for this modification. This is mainly treated in Section \ref{sec:interior}.
	\item For different points $X, X' \in \overline{D}$, the degrees of homogeneity may be different. (Especially recall that the monotone quantities for boundary points are defined in different domains $\Omega_{X}$ and $\Omega_{X'}$.) This creates a difficulty when we need to quantitatively measure how far $u$ is from being homogeneous (see Section \ref{sec:L2approx}) and when we study off-spinal points in the dimension reduction argument (see Proposition \ref{prop:tn}). So we need to study the variation of the monotone quantity (in small scales) as $X$ changes. This is dealt with mainly in Section \ref{sec:spvar}. (We remark that various degrees of homogeneity also appear in the earlier work of \cite{DMSV, FS}.)
	\item As discussed above, for both interior and boundary points we have found some alternatives in place of Almgren's frequency function to measure how close $u$ is from being homogeneous near that point. For a boundary point $X$ we use a monotone quantity defined in the reduced almost-convex domain $\Omega_X$; for an interior point $X$ we use the frequency function centered at a boundary point closest to $X$. We need to show the errors accrued by these alternatives do not accumulate when we count across possibly an infinite number of dyadic scales, see Theorem \ref{thm:L2appx}.
\end{itemize}

After introducing preliminary notation and definitions in Section \ref{sec:prelim}, we will define the modified frequency function and prove its monotonicity for boundary points in Sections \ref{sec:monotonicity} and \ref{sec:reduce}, and for interior points in Section \ref{sec:interior}. More precisely, for boundary points we first use the transformation map to change the Laplacian operator in a Dini domain $D$ to a divergence-form elliptic operator in an \textit{almost convex} domain $\Omega$ (see Section \ref{sec:reduce}). (Here almost convex is with respect to the elliptic operator and a single boundary point, see \eqref{def:convex}.) Then we use \textit{boundary convexity} to define the frequency function for solutions to a non-constant elliptic operator and prove its monotonicity as well as quantitative rigidity (see Section \ref{sec:monotonicity}). For interior points we establish the monotonicity and quantitative rigidity of the standard Almgren's frequency function up to some scale. We also state a closeness result between the frequency function centered at $X\in D$ and the modified frequency function centered at $\tilde{X}$, which is a boundary point closest to $X$.
In Section \ref{sec:blowup}, we first analyze the frequency function and use quantitative rigidity to characterize the tangent functions as homogeneous harmonic polynomials in the upper half-space. We also use the uniform bound of the frequency function to prove compactness for rescalings of the harmonic function.
Section \ref{sec:qsym} proves quantitative versions of the following two principles for harmonic functions: if the frequency function centered at a point $X$ is constant, then the harmonic function is homogeneous with respect to $X$; if the frequency functions centered at two distinct points $X, X'$ are both constant, then the harmonic function is invariant along the line connecting $X$ and $X'$, and it is homogeneous with respect to any point on the line (with the same degree of homogeneity). These two principles are used extensively throughout the paper.
In Section \ref{sec:spvar} we estimate the spatial variation of the (modified) frequency function for both boundary and interior points.
In the beginning of Section \ref{sec:L2approx}, we state two Reifenberg theorems, which are generalizations of the classical Reifenberg theorem and which connect the estimate of the $L^2$ $\beta$-number with packing estimates and rectifiability. The reason why Reifenberg-type theorem is relevant is that we can relate the $(d-2)$-dimensional $\beta$-number to the change of the frequency functions centered at effective singular points, 
as proven in Theorem \ref{thm:L2appx}.
Finally in Section \ref{sec:covering}, we cover the effective singular set $\Nt(u) \cap \widetilde{\Ct}_r(u)$ inductively, with the stopping criterion being that every point in the ball has a definite frequency drop. By Proposition \ref{prop:tn}, in a fixed ball all such points with small frequency drop  are either clustered to a small set (\textit{small} means of dimension $\leq (d-3)$), or they must lie in a tubular neighborhood of a $(d-2)$-dimensional affine plane. In the second case, we control their volume using the Reifenberg-type theorem.

\section{Preliminaries}\label{sec:prelim}
\begin{defn}\label{def:Dini}
	Let $\theta: [0,+\infty) \to (0, +\infty)$ be a nondecreasing function verifying
	\begin{equation}\label{cond:Dini}
		\int_0^* \frac{\theta(r)}{r} < \infty. 
	\end{equation} 
	A connected domain $D$ in $\RR^d$ is a \textit{Dini domain} with parameter $\theta$ if for each point $X_0$ on the boundary of $D$ there is a coordinate system $X=(x,x_d), x\in \RR^{d-1}, x_d\in \RR$ such that with respect to this coordinate system $X_0=(0,0)$, and there are a ball $B$ centered at $X_0$ and a Lipschitz function $\varphi: \RR^{d-1} \to \RR$ verifying the following
	\begin{enumerate}
		\item $\|\nabla \varphi\|_{L^\infty(\RR^{d-1})} \leq C_0$ for some $C_0>0$;
		\item $|\nabla \varphi(x)- \nabla \varphi(y)| \leq \theta(|x-y|)$ for all $x,y \in \RR^{d-1}$;
		\item $D\cap B=\{(x,x_d)\in B: x_d > \varphi(x) \}$.
	\end{enumerate}
\end{defn}
\begin{remark}
	By shrinking the ball $B$ if necessary, we may modify the coordinate system so that $\nabla \varphi(0)=0$.	
\end{remark}

\begin{defn}\label{def:function}
	Let $R, \Lambda>0$ be finite. For any domain $D$ and any function $u$, we say $(u, D) \in \mathfrak{H}(R, \Lambda)$ if
	\begin{itemize}
		\item $D$ is a Dini domain in $\RR^d$ with parameter $\theta$, and it satisfies $\pD \ni 0$, $D$ is graphical inside the ball $B_{5R}(0)$ and
			\begin{equation}\label{cond:R}
				\theta(8R)< \frac{1}{72}, \quad \int_0^{16R} \frac{\theta(s)}{s} ~ds \leq 1 ;
			\end{equation} 
		\item $u$ is a non-trivial harmonic function in $D\cap B_{5R}(0)$,
		\item $u = 0$ on $\pD \cap B_{5R}(0)$,
		\item the frequency function for $v:= u\circ \Psi_0$ satisfies $N_0(4R)\leq \Lambda<+\infty$ (see Section \ref{sec:reduce} for the definitions of $\Psi_0$ and the frequency function).
	\end{itemize}  
\end{defn}
\begin{remark}\label{rmk:deptheta}
	In the above definition we indicate how $(u,D)$ depends on $R$, but we do not explicitly indicate its dependence on the Dini parameter $\theta$. This is an absuse of notation.	
	 In fact, throughout the paper we allow $\theta$ to vary as long as it is bounded from above by a given Dini function $\theta_0$.
	 We use the quantifier $R$, which is determined by the Dini parameter $\theta$, in the notation above because $R$ will appear explicitly in the statements of lemmas and theorems.
%
\end{remark}

For any compact set $K$ and $\tau>0$, we use
\[ B_\tau(K) := \{X: \dist(X, K) < \tau \} \]
to denote the $\tau$-neighborhood of $K$.
\begin{defn}[Minkowski content]\label{def:Minkowski}
	Let $A$ be a bounded subset of $\RR^d$. The $s$-dimensional upper and lower Minkowski contents of $A$ are defined as
	\[ \mathcal{M}^{s,*}(A) = \limsup_{r \to 0+} (2r)^{s-d} \left|B_r(A) \right|, \]
	\[ \mathcal{M}^{s}_*(A) = \liminf_{r\to 0+} (2r)^{s-d} \left| B_r(A) \right|. \]
	Moreover the upper and lower Minkowski dimensions of $A$ are defined as
	\[ \overline{\dim}_{\mathcal{M}} ~A = \inf\{s>0: \mathcal{M}^{s,*}(A) = 0 \} = \sup\{s>0: \mathcal{M}^{s,*}(A)>0 \}, \]
	\[ \underline{\dim}_{\mathcal{M}} ~A = \inf\{s>0: \mathcal{M}^{s}_*(A) = 0 \} = \sup\{s>0: \mathcal{M}^{s}_*(A)>0 \}. \]
\end{defn}
\begin{remark}
	It is not hard to see that $\mathcal{H}^s(A) \lesssim \mathcal{M}^s_*(A) \leq \mathcal{M}^{s, *}(A)$, where strict inequalities are possible. In other words, bounding the Minkowki content is stronger than bounding the Hausdorff measure of the same dimension. In fact let $\mathbb{Q}^d\subset \RR^d$ be the $d$-dimensional rational lattice, its Hausdorff dimension is $0$ and yet its Minkowski dimension is $d$. See \cite[Chapter 5]{Mat} for more about Hausdorff measure, Minkowski content and packing measure.
\end{remark}

For a harmonic function $u$ in $D \cap B_{5R}(0)$ with vanishing boundary data and $D$ being a Dini domain, we have that $u$ is continuously differentiable up to boundary (see \cite{DEK}).
Let 
\begin{equation}\label{def:Ntu}
	\mathcal{N}(u):= \{ X\in \overline{D} \cap B_{5R}(0): u(X) = 0 \} 
\end{equation} 
be the nodal set of $u$. In particular by the above definition $\pD \cap B_{5R}(0) \subset \mathcal{N}(u)$. We also define the critical set of $u$ by
\[ \mathcal{C}(u) := \{X\in \overline{D} \cap B_{5R}(0): \nabla u(X) = 0\}, \]
and define the singular set of $u$ by
\begin{equation}\label{def:Stu}
	\mathcal{S}(u): = \mathcal{N}(u) \cap \mathcal{C}(u) = \left\{X\in \overline{D} \cap B_{5R}(0): u(X)= |\nabla u(X)|= 0\right\}.
\end{equation} 

\begin{defn}\label{def:Tpr}
Let $u$ be an arbitrary non-trivial $L^2$ function. We use the following notation to denote its rescalings centered at $X\in \overline{D}$ with scale $r>0$:
\[ T_{X,r}u(Y) := \frac{u(X+r Y)-u(X)}{\left( \frac{1}{r^d} \mathlarger{\iint}_{B_{r}(X)\cap D } \left|u-u(X) \right|^2 ~dZ \right)^{\frac12}}, \qquad T_{r}u(Y) := \frac{u(rY) - u(0)}{\left(\frac{1}{r^d} \mathlarger{\iint}_{B_{r}(0)\cap D } \left|u-u(0)\right|^2 ~dZ \right)^{\frac12}}. \]
%
\end{defn}
\begin{remark}
	Clearly the above rescaling satisfies 
\[ \iint_{B_1(0) \cap \frac{D-X}{r}} |T_{X,r} u|^2 ~dY = 1. \]
\end{remark}
\begin{remark}
	We also remark that for any $X \in \mathcal{N}(u)$, the above rescaling satisfies
	\[ T_{X,r}u(Y) = \frac{u(X+r Y)}{\left( \frac{1}{r^d} \mathlarger{\iint}_{B_{r}(X)\cap D } u^2 ~dZ \right)^{\frac12}} = \frac{u(X+r Y)}{\left( \frac{1}{r^d} \mathlarger{\iint}_{B_{r}(X) } u^2 ~dZ \right)^{\frac12}}. \]
	Since $u$ vanishes on the boundary $\pD \cap B_{5R}(0)$, it is convenient to simply extend $u$ by zero outside of $D\cap B_{5R}(0)$. This is what we did in the second equality.
\end{remark}

In this paper we work with a quantitative version of the critical set, defined as follows. Let $\beta, \alpha_0>0$ be two small constants to be determined later\footnote{see Proposition \ref{prop:tn} as well as Lemma \ref{lm:sym}. Heuristically as $r\to 0+$, $T_{X,r} u$ converges to a homogeneous harmonic polynomial $P_N$ of degree $N\in \mathbb{N}$, and it satisfies $\iint_{B_1(0)} |P_N|^2 ~dZ = 1$. In particular it implies that $\iint_{B_1(0)}|\nabla P_N|^2 ~dZ = 2N+d$. In the special case when $N=1$ and $P_N$ is linear, $|\nabla P_N|$ equals a dimensional constant, denoted by $\alpha_d$. Here $\alpha_0$ is chosen to be strictly smaller than $\alpha_d$. On the other hand when the degree $N>1$, by homogeneity $|\nabla P_N(r\omega)| = O(r^{N-1}) $ grows polynomially in the radial direction. Since $\iint_{B_1(0)}|\nabla P_N|^2 ~dZ = 2N+d$ and it has a uniform upper bound when $N\leq C(\Lambda)$, the polynomial growth of $|\nabla P_N|$ implies that we can choose $\beta<1$ so that for any degree $N\leq C(\Lambda)$, $\sup_{B_\beta(0)} |\nabla P_N|$ is also strictly smaller than $\alpha_d$. (It's certainly not the case that $ |\nabla P_N| < \alpha_d$ on all of $B_1(0)$.)} (whose values depend on $d, \Lambda$). For any $r>0$ we set
\begin{align}
	\Ct_r(u) & := \left\{X\in \overline{D}: \inf_{B_{\beta}(0)} |\nabla T_{X, r} u| \leq \alpha_0  \right\} \nonumber \\
	& = \left\{ X \in \overline{D}: \inf_{B_{\beta r}(X)} r^2 |\nabla u|^2 \leq \alpha_0^2 \cdot \frac{1}{r^d} \iint_{B_r(X)} |u-u(X)|^2 ~dZ \right\}. \label{def:Cru}
\end{align}
Clearly $\Ct(u) \subset \Ct_r(u)$ for every $r>0$, and $X\in \overline{D} \setminus \Ct_r(u)$ means that $|\nabla u|$ has a uniform lower bound centered at $X$ at scale $r$. Heuristically, near $X$ one of the following two scenarios happen. Either $u$ blows up to a homogeneous harmonic polynomial of degree $N \geq 2$, in which case $X\in \Ct(u)$ and $\nabla T_{X,r}u(0) = 0$ for any $r>0$. Or $u$ blows up to a homogeneous harmonic polynomial of degree $N=1$, which is just a linear function. Since the gradient of a linear function with unit $L^2$-norm has constant modulus, by choosing $\alpha_0$ sufficiently small $X\in \Ct_r(u)\setminus \Ct(u)$ just means that even though $u$ blows up to a linear function, but at the scale $r$, $u$ is a fixed distance away from linear functions.\footnote{This can be made rigorous in the interior case. By \cite[Theorem 3.1]{Han} the harmonic function has an expansion of the form $u(X+Y) - u(X) = P_N(Y) + \Psi(Y)$, where $P_N$ is a homogeneous harmonic polynomial of degree $N \in \NN$ and $\Psi(Y) = O(|Y|^{N+\epsilon})$; moreover $\|\nabla T_{X,r} u - \nabla P_N\|_{L^p(B_1(0))} = O(r^{\epsilon + \frac{d}{p} - 1})\to 0$ for any $p\in (1, d]$. }
Moreover, we define
\begin{equation}\label{def:Ctu}
	\widetilde{\Ct}_r(u) : = \left\{X \in \overline{D}: \inf_{B_{\beta}(0)} |\nabla T_{X, s} u| \leq \alpha_0 \text{ for any }r\leq s \leq r_c \right\} = \bigcap_{r\leq s\leq r_c} \Ct_s(u)
\end{equation}
to make sure $\widetilde{\Ct}_r(u)$ is monotone with respect to $r$, i.e.
\begin{equation}\label{incl:Ct}
	\Ct(u) \subset \widetilde{\Ct}_{r_1}(u) \subset \widetilde{\Ct}_{r_2}(u), \quad \text{ for any } 0< r_1 \leq r_2.
\end{equation}
The value of $r_c$ is determined (depending on $d, R, \Lambda$) as in the end of Section \ref{sec:covering}.

\begin{lemma}\label{lm:spine}
	Let $h$ be a harmonic function in $\RR^d$. Suppose $h$ is $N_1$-homogeneous with respect to the origin, and it is also $N_2$-homogeneous with respect to $X\in \RR^d \setminus\{0\}$\footnote{For any $X\in \RR^d$, we say $h$ is $N$-homogeneous with respect to $X$ if $h-h(X)$ is $N$-homogeneous, namely $h(X+\lambda Z) - h(X) = \lambda^{N}( h(X+Z) - h(X))$ for any $Z\in \RR^d$ and $\lambda \in \RR_+$}. Then $N_1 = N_2 \in \mathbb{N}$, and 
	\begin{itemize}
		\item either $h$ is linear, i.e. $N_1 = N_2 = 1$;
		\item or $h$ is invariant in the $X$-direction, i.e.
	\[ h(tX + Y) = h(Y), \quad \text{ for any } Y\in \RR^d \text{ and } t\in \RR. \]
	\end{itemize}
\end{lemma}
\begin{proof}
	Let $t\in [0,1]$. By the assumption, we have
	\[ h\left( (1-t)X \right) - h(0) = (1-t)^{N_1} \left[ h(X) - h(0) \right], \]
	and
	\[ h\left( (1-t) X \right) - h(X) = t^{N_2} \left[ h(0) - h(X) \right]. \]
	Combining these two equalities we get
	\[ h(0) + (1-t)^{N_1} \left[ h(X) - h(0) \right] = h(X) + t^{N_2} \left[ h(0) - h(X) \right], \]
	or equivalently
	\[ \left[ (1-t)^{N_1} + t^{N_2} \right] \cdot \left[ h(X) - h(0) \right] = h(X) - h(0). \]
	Since $h$ is a harmonic function, the degrees of homogeneity $N_1, N_2$ must be positive integers. Therefore either $N_1 = N_2 = 1$, or $h(X) - h(0) = 0$.
	
	Assume we are in the second case where $h(X) = h(0)$. In particular we have $h(tX) = h(0)$ for every $t\in \RR$.	
	By the assumption $h$ is $N_1$-homogeneous with respect to the origin, so
	\begin{equation}\label{tmp:hom0}
		\langle \nabla h(Z), Z \rangle = N_1 \left[ h(Z) - h(0) \right]. 
	\end{equation} 
	Similarly we also have
	\begin{equation}\label{tmp:hom1}
		\langle \nabla h(Z), Z-X \rangle = N_2 \left[ h(Z) - h(X) \right]. 
	\end{equation} 
	Subtracting \eqref{tmp:hom1} from \eqref{tmp:hom0}, we get
	\begin{equation}\label{tmp:hom2}
		\langle \nabla h(Z), X \rangle = (N_1 - N_2) \left[ h(Z) - h(0) \right].
	\end{equation}
	
	Let $r>0$ be fixed and we define 
	\[ g(t) = \log N(tX, r) = \log \frac{r D(tX, r)}{H(tX, r)}, \]
	where
	\[ D(tX, r) = \iint_{B_r(tX)} |\nabla h|^2 ~dZ, \quad H(tX, r) = \int_{\partial B_r(tX)} \left| h - h(tX) \right|^2 \sH. \]
	We want to show that $g'(t) = 0$.
	To that end, we compute
	\begin{align*}
		\frac{d}{dt} D(tX, r) = \frac{d}{dt} \iint_{B_r(0)} |\nabla h(tX+Y)|^2 ~dY & = 2 \iint_{B_r(tX)} \langle \nabla h(Z), \nabla^2 h(Z)(X) \rangle ~dZ \\
		& = 2\iint_{B_r(tX)} \left\langle \nabla h, \nabla \langle \nabla h, X \rangle \right\rangle  ~dZ \\
		& = 2 \iint_{B_r(tX)} \divg\left(\langle \nabla h, X \rangle \nabla h \right) ~dZ \\
		& = 2 \int_{\partial B_r(tX)} \langle \nabla h, X \rangle \frac{\partial h}{\partial n} ~d\mathcal{H}^{d-1}.
	\end{align*}
	And
	\begin{align*}
		\frac{d}{dt} H(tX, r) & = \frac{d}{dt} \int_{\partial B_r(0)} \left| h(tX+Y) - h(tX) \right|^2 ~d\mathcal{H}^{d-1}(Y) \\
		& = 2\int_{\partial B_r(tX)} \langle \nabla h, X \rangle \left( h-h(tX) \right) ~d\mathcal{H}^{d-1} \\
		& \qquad \qquad \qquad - 2\langle \nabla h(tX), X \rangle \cdot \int_{\partial B_r(tX)} \left( h - h(tX) \right) \sH \\
		& = 2\int_{\partial B_r(tX)} \langle \nabla h, X \rangle \left( h-h(tX) \right) ~d\mathcal{H}^{d-1},
	\end{align*}
	where we use the mean value property in the last equality.
	Combined with \eqref{tmp:hom2}, $h(tX) = h(0)$ and the equality
	\[ \iint_{B_r(tX)} |\nabla h|^2 ~dZ = \int_{\partial B_r(tX)} \left( h-h(tX) \right) \frac{\partial h}{\partial n} ~ d\mathcal{H}^{d-1}, \]
	we obtain
	\begin{align}
		g'(t) & = \frac{\frac{d}{dt} D(tX, r)}{D(tX, r)} - \frac{\frac{d}{dt} H(tX, r)}{H(tX, r)} \nonumber \\
		& = \frac{ 2 \int_{\partial B_r(tX)} \langle \nabla h, X \rangle \frac{\partial h}{\partial n} ~d\mathcal{H}^{d-1}}{\int_{\partial B_r(tX)} \left( h - h(0) \right) \frac{\partial h}{\partial n} ~d\mathcal{H}^{d-1}} - \frac{2\int_{\partial B_r(tX)} \langle \nabla h, X \rangle \left( h - h(0) \right) ~d\mathcal{H}^{d-1}}{\int_{\partial B_r(tX)} \left|h - h(0) \right|^2 ~d\mathcal{H}^{d-1}} \label{eq:spvar_ls} \\
		& = 2(N_1-N_2) - 2(N_1 - N_2) \nonumber \\
		& = 0. \label{eq:spvar_ls0}
	\end{align}
	Therefore the map $t\mapsto N(tX, r)$ is a constant. In particular
	\[ N_2 = N(X, r) = N(tX, r) = N(0, r) = N_1, \]
	and $h$ is $N_1$-homogeneous with respect to the vertex $tX$, for any $t\in \RR$.
	Moreover, by \eqref{tmp:hom2} we have that for every $Y\in \RR^d$ and $t\in \RR$,
	\[ h(tX+Y) - h(Y) = \int_0^t \frac{d}{ds} h(sX+Y) ~ds = \int_0^t \langle \nabla h(sX+Y), X \rangle ~ds = 0  \]
	
\end{proof}

The above theorem also has an analogue in the upper half-space:
\begin{lemma}\label{lm:spinebd}
	Let $h$ be a harmonic function in $\RR^d_+$ and $h=0$ on $\partial \RR^d_+$. Suppose $h$ is $N_1$-homogeneous with respect to the origin, and it is also $N_2$-homogeneous with respect to the vertex $X\in \partial \RR^d_+ \setminus\{0\}$. Then $h$ is invariant in the $X$-direction, i.e.
	\[ h(tX + Y) = h(Y) \quad \text{ for any } Y\in \RR^d \text{ and } t\in \RR. \]
	Moreover $N_1 = N_2 \in \mathbb{N}$.
\end{lemma}
\begin{proof}
	Using $h(0) = h(X) = 0$, the proof of $N_1 = N_2$ is similar to that of Lemma \ref{lm:spine}. The only difference is that when we integrate by parts, we have an extra term on the boundary $B_r(tX) \cap \partial \RR^d_+$. However, since $h$ vanishes on the boundary and $X \in \partial \RR^d_+$, we have that $\langle \nabla h, X \rangle = 0$. To prove the invariance in the $X$-direction, it suffices to study the first alternative when $h$ is linear. Suppose $h$ is linear and the invariant hyperplane is denoted by $V$. Either $V = \partial \RR^d_+$, then $h$ is invariant in the $X$-direction since $X\in V$; or $V$ intersects $\partial \RR^d_+$ transversally, in which case $h\equiv 0$ trivially and there is nothing to prove.

\end{proof}

\section{Monotonicity formula for elliptic operators}\label{sec:monotonicity}

Let $\Omega$ be a Lipschitz domain with $0\in \pO$, and let $A$ be a symmetric elliptic matrix defined on $\RR^d$ satisfying the following assumptions (c.f. \cite[Theorem 1.1]{AE}):

\begin{itemize}
	\item $A(0)=\Id$ and 
		\begin{equation}\label{def:convex}
			\langle A(X)X, n(X) \rangle \geq 0, \quad \text{ for all } X \in \Delta_1 = B_1 \cap \pO,
		\end{equation}
 		where $n(X)$ denotes the unit outer normal to $\Omega$ at $X$;
	\item There is a non-decreasing function $\theta: [0, + \infty ) \to [0,+ \infty )$ such that
		\begin{equation}\label{eq:gradA}
			|A(X)-A(0)| \leq \theta(|X|), \quad |\nabla A(X)| \leq \theta(|X|)/|X|
		\end{equation}
		and
		\begin{equation}\label{eq:Dini}
			\int_0^* \frac{\theta(r)}{r} dr < +\infty.
		\end{equation}
\end{itemize}
Let $v$ be a non-trivial solution to $Lv = \divg(A(X)\nabla v)$ satisfying $ v= 0$ on the surface ball $\Delta_1:= B_1(0) \cap \pO$. The assumption \eqref{def:convex} means that the domain $\Omega$ is $L$-convex with respect to the boundary point $0\in \pO$.
Inspired by Almgren's frequency function for harmonic functions, we want to define the frequency function for $v$ centered at $0$, and prove that it is almost monotone.
\bigskip

We first define a metric tensor $\bar g = \bar{g}_{ij}(X) ~ dX_i \otimes dX_j$ by setting 
\[ (\bar{g}_{ij}(X)) := (\det A(X))^{\frac{1}{d-2}} A(X)^{-1}, \quad \text{ where the indices } i, j \text{ are in } \{1, \cdots, d\}; \] and its inverse metric is denoted by $(\bar{g}^{ij}) = (\bar{g}_{ij})^{-1}$. We are interested in a neighborhood of the origin, so we define
\[ r(X)^2 := \bar{g}_{ij}(0) (X-0)_i (X-0)_j \quad \text{and } \quad\eta(X) := \bar{g}^{kl}(X) \frac{\partial r}{\partial x_k} \frac{\partial r}{\partial x_l}. \]
Now we define a new metric tensor $g = g_{ij}(X)~ dX_i \otimes dX_j$ by setting
\[ g_{ij}(X) := \eta(X) \bar{g}_{ij}(X) = (\det A(X))^{\frac{1}{d-2}} \eta(X) \left( A(X)^{-1} \right)_{ij}. \]
The advantage of modifying the metric $\bar{g}$ by $\eta(X)$ is that under the new metric, the geodesic distance from $0$ to any $X$ is the same as $r(X)$; moreover, under polar coordinates $(r,\omega)$ the metric $g$ is simply
\begin{equation}
	g(r,\omega) = dr \otimes dr + r^2 b_{kl}(r, \omega) ~d\omega^i \otimes d\omega^j, \quad \text{where } i,j \in \{1, \cdots, d-1\}.
\end{equation}
We refer interested readers to \cite[Section 3]{AKS}.
We denote the inverse metric of $g$ as $(g^{ij}(X)) = (g_{ij}(X))^{-1}$, and write $|g(X)| = |\det g_{ij}(X)|$. Let $\tilde{\eta}(X) = \left( \det A(X) \right)^{\frac{1}{d-2}} \eta(X)$, and simple computations show
\[ \tilde{\eta}(X) = \dfrac{\langle A(X)X, X \rangle}{|X|^2}, \quad g(X) = \tilde{\eta}(X) A(X)^{-1}. \]
By ellipticity and the assumption \eqref{eq:gradA}, we know that for all $X$ in a small neighborhood of the origin we have
\begin{equation}\label{eq:bdul}
	C_1 \leq \eta(X), ~ |b(r,\omega)| \leq C_2
\end{equation}
\begin{equation}\label{eq:bdgrad}
	|\nabla \eta(X)|, ~ |\nabla g_{ij}(X)|, ~ |\nabla b_{kl}(X)| \lesssim |\nabla A(X)| \leq \theta(|X|)/|X|
\end{equation}
where the constants depend on $n$ and the ellipticity of the matrix $A$.

Under the metric $g$, we can write the divergence-form elliptic operator $Lv = \divg(A(X) \nabla v)$ as
\begin{equation}\label{eq:eqing}
	L_g v = \divg_g(\mu(X) \nabla_g v), \quad \text{ where } \mu(X) = \eta(X)^{-\frac{d-2}{2}}.
\end{equation}
Here we denote by $\nabla_g v$ and $\divg_g \vec{Y}$ the intrinsic gradient of a function $v$ and the intrinsic divergence of a vector field $\vec{Y}$ in the metric $g$, i.e. 
\[ \nabla_g v = g^{ij} \frac{ \partial v}{\partial x_i} \frac{\partial}{\partial x_j}, \quad \divg_g \vec{Y} = \frac{1}{\sqrt{|g|}} \divg\left(\sqrt{|g|} \vec{Y} \right) \]

For a solution $v$ to \eqref{eq:eqing}, we define 
\begin{equation}\label{def:freqv0}
	D_g(r) = \iint_{B_r \cap \Omega} \mu |\nabla_g v|_g^2 dV_g, \quad H_g(r) = \int_{\partial B_r \cap \Omega} \mu v^2 dV_{\partial B_r}.
\end{equation}
Here $B_r$ denotes the geodesic ball centered at $0$ in the metric $g$. But by the above discussion it is the same as the Euclidean ball centered at $0$ of the same radius.
Now we define the frequency function of $u$ as
\begin{equation}\label{def:freqv}
	N_g(r) = \frac{r D_g(r)}{H_g(r)}. 
\end{equation} 
For simplicity we often omit the dependence on the metric $g$ and just write $D(r), H(r)$ and $N(r)$.
We will show the following:

\begin{prop}\label{prop:monotonicity}
	There exists some $r_0<1$, such that the frequency function satisfies
	\begin{equation}\label{eq:derNr}
		N'(r) = O\left( \frac{\theta(r)}{r} \right) N(r) + R_h(r) + R_b(r),
	\end{equation}
	for all $r<r_0$. Here
	\[ R_h(r) = \frac{2r}{H^2(r)} \left[ \left( \int_{\partial B_r \cap\Omega} \mu v^2 ~ dV_{\partial B_r} \right) \cdot \left( \int_{\partial B_r \cap \Omega} \mu \left( v_\rho \right)^2 ~dV_{\partial B_r} \right) - \left( \int_{\partial B_r \cap\Omega} \mu v v_\rho ~ dV_{\partial B_r} \right)^2 \right] \geq 0, \]
	\[ R_b(r) = \frac{1}{H(r)} \int_{\Delta_r} \frac{(\partial_n v)^2}{ \tilde{\eta}(X)} \langle A(X) X, n_\Omega(X) \rangle \langle A(X)n_\Omega(X), n_\Omega(X) \rangle d\mathcal{H}^{d-1} \geq 0, \]
	where $v_\rho = \langle \nabla_g v, X/|X|\rangle$ denotes the radial differentiation.\footnote{Let $\vec{Y}$ be an arbitrary vector, then its projection onto the radial direction can be written as $\left\langle \vec{Y}, \frac{ \frac12 \nabla_g |X|^2}{|X|} \right\rangle_g = \left\langle \vec{Y}, \frac{X}{|X|} \right \rangle $.} (We use $\langle \cdot, \cdot \rangle$ to denote Euclidean inner product, and use $\langle \cdot, \cdot \rangle_g$ to denote inner product with the metric $g$.) 
	
	In particular, there exists a constant $C>0$ depending on $n$ and the ellipticity of the matrix $A$, such that the modified frequency function
	\begin{equation}\label{eq:Nrmn}
		\widetilde{N}(r):= N(r)\exp\left( C\int_0^r \frac{\theta(s)}{s} ds \right) \text{ is monotone increasing with respect to } r.
	\end{equation}
\end{prop}

\begin{remark}
	\begin{itemize}
		\item Written in Euclidean metric, we have
	\begin{equation}\label{eq:DrHrinEm}
		D(r) = \iint_{B_r \cap \Omega} \langle A(X) \nabla v, \nabla v \rangle dX, \quad H(r) = \int_{\partial B_r \cap \Omega} \tilde{\eta}(X) v^2(X) d\mathcal{H}^{d-1}(X), 
	\end{equation} 
	with $\tilde{\eta}(X) = \langle A(X) X, X \rangle/|X|^2$.	
		\item For convenience, we can also write $R_h(r)$ as
			\begin{equation}\label{eq:Rhr}
				R_h(r) = \frac{2r}{H(r)} \int_{\partial B_r \cap\Omega} \mu \left| v_\rho - N(r)  ~\frac{v}{r} \right|^2 ~dV_{\partial B_r}. 
			\end{equation} 
	\end{itemize}
\end{remark}

\begin{proof}
	
We first compute $H'(r)$. Let $|b(r,\omega)| = |\det \left( b_{kl}(r,\omega) \right) |$, then
\[ dV_g(X) = \sqrt{|g(X)|} ~dX, \quad dV_{\partial B_r}(\omega) = r^{d-1} \sqrt{|b(r,\omega)|} ~d\mathcal{H}^{d-1}(\omega). \]
By extending $u$ by zero in $\Omega^c$, we can rewrite $H(r)$ as
\[ H(r) = \int_{\partial B_r} \mu v^2 ~ dV_{\partial B_r} = r^{d-1} \int_{\partial B_1} \mu(r,\omega) \sqrt{|b(r,\omega)|} v^2(r,\omega) d\mathcal{H}^{d-1}(\omega).  \]
Taking the derivative in $r$, we get
\[ H'(r) = \frac{d-1}{r} H(r) + 2\int_{\partial B_r} \mu v  v_\rho dV_{\partial B_r} + \int_{\partial B_r} \frac{1}{\sqrt{|b|}} \partial_{\rho}\left( \mu \sqrt{|b|} \right) v^2 dV_{\partial B_r}. \]
By \eqref{eq:bdul} and \eqref{eq:bdgrad}, we can bound the last term and get
\begin{equation}\label{eq:derHr}
	H'(r) = \left( \frac{d-1}{r} + O\left( \frac{\theta(r)}{r} \right)  \right) H(r) + 2 \int_{\partial B_r} \mu v  v_\rho ~dV_{\partial B_r}.
\end{equation} 
Moreover, since $\divg_g(\mu\nabla_g v) = 0$ we have
\[ \iint_{B_r \cap\Omega} \divg_g(\mu \nabla_g v^2)~ dV_g = 2 \iint_{B_r \cap\Omega} \mu |\nabla_g v|_g^2 ~dV_g. \]
On the other hand by the divergence theorem and the assumption $v\equiv 0$ on $\Delta_r$, we also have
\[ \iint_{B_r \cap\Omega} \divg_g(\mu \nabla_g v^2)~ dV_g = 2\int_{\partial B_r \cap\Omega} \mu v v_\rho ~ dV_{\partial B_r}. \]
Therefore we have
\begin{equation}\label{eq:DrderHr}
	D(r) = \iint_{B_r \cap \Omega} \mu |\nabla_g v|_g^2 ~dV_g = \int_{\partial B_r \cap\Omega} \mu v v_\rho ~dV_{\partial B_r}.
\end{equation}
\bigskip

Now we compute $D'(r)$. Since $D(r)$ is a solid integral in $B_r$, we easily get the formula
\[ D'(r) = \int_{\partial B_r \cap\Omega} \mu|\nabla_g v|_g^2 ~ dV_{\partial B_r}. \] 
We need the following \textbf{Rellich identity}: if the function $v$ satisfies $\divg_g(\mu \nabla_g v) = 0$, then
\begin{align}
	\divg_g \left(\mu|\nabla_g v|^2_g \cdot \frac{1}{2} \nabla_g |X|^2 \right)  = & ~ 2\divg_g \left( \left\langle \frac12 \nabla_g|X|^2, \nabla_g v \right\rangle_g \cdot \mu \nabla_g v \right) + \divg_g \left(\frac12 \nabla_g |X|^2 \right) \cdot \mu |\nabla_g v|^2_g \nonumber \\
	& - 2\mu |\nabla_g v|_g^2 + O(\theta(|X|)) \cdot \mu|\nabla_g v|^2_g.\label{eq:Rellich}
\end{align}
For comparison we remind the readers of the standard Rellich identity
\[ \divg(|\nabla v|^2 X) = 2\divg(\langle X, \nabla v \rangle \nabla v) + \divg(X)|\nabla v|^2 - 2|\nabla v|^2. \]
(Notice that in Euclidean metric, the position vector $X$ is pointing to the radial direction; but under the metric $g$, the radial direction at the point $X$ is computed by $\frac12 \nabla_g |X|^2$.)
We claim the identity \eqref{eq:Rellich} holds, and show why it helps in the computation of $D'(r)$. We integrate the above identity in the domain $B_r \cap \Omega$, and see what each term becomes.

By the divergence theorem,
\begin{align*}
	I:= \iint_{B_r \cap \Omega} \divg_g \left(\mu|\nabla_g v|^2_g \cdot \frac{1}{2} \nabla_g |X|^2 \right) dV_g = \int_{\partial (B_r \cap \Omega)} \sqrt{|g|} ~ \mu|\nabla_g v|^2_g \left\langle  \frac{1}{2} \nabla_g |X|^2, n \right\rangle d\mathcal{H}^{d-1},
\end{align*}
\begin{align*}
	II & := \iint_{B_r \cap \Omega} \divg_g \left( \left\langle \frac12 \nabla_g|X|^2, \nabla_g v \right\rangle_g \cdot \mu \nabla_g v \right) dV_g \\
	&= \int_{\partial (B_r \cap \Omega)} \sqrt{|g|}  \left\langle \frac12 \nabla_g|X|^2, \nabla_g v \right\rangle_g \left\langle  \mu \nabla_g v, n \right\rangle d\mathcal{H}^{d-1}.
\end{align*}
On $\partial B_r \cap \Omega$, the unit outer normal is $n(X) = X/|X|$ and
\[ \left\langle \frac12 \nabla_g |X|^2, \frac{X}{|X|} \right\rangle = \frac{1}{\tilde{\eta}(X)} \frac{\langle A(X) X, X \rangle}{|X|} = |X| = r,  \]
\[ \left\langle \frac12 \nabla_g |X|^2, \nabla_g v \right\rangle_g = \frac{1}{\tilde{\eta}(X)} \left\langle A(X) \nabla v, X \right\rangle = |X| ~v_\rho = r ~v_\rho, \quad \left\langle \nabla_g v, \frac{X}{|X|} \right\rangle = v_\rho. \]
On $\Delta_r: = B_r \cap \pO$, the unit outer normal vector is that of the domain $\Omega$, denoted by $n_\Omega(X)$. And 
\[ \left\langle  \frac{1}{2} \nabla_g |X|^2, n \right\rangle = \frac{ \langle A(X)X, n_\Omega(X) \rangle }{\tilde{\eta}(X)} \geq 0, \]
\[ \sqrt{|g|} ~ \mu |\nabla_g v|_g^2 = \langle A(X)\nabla v, \nabla v\rangle = \left( \partial_n v \right)^2 \left\langle A(X) n_\Omega(X), n_\Omega(X) \right\rangle \geq 0, \]
\[ \langle \nabla_g v, n_\Omega \rangle = \frac{1}{\tilde{\eta}(X)} \langle A(X) \nabla v, n_\Omega(X) \rangle = \frac{\partial_n v}{\tilde{\eta}(X)} \langle A(X) n_\Omega(X), n_\Omega(X) \rangle, \]
\[ \left\langle \frac12 \nabla_g|X|^2, \nabla_g v \right\rangle_g = \frac{1}{\tilde{\eta}(X)} \left\langle A(X) \nabla v, X \right\rangle = \frac{\partial_n v}{\tilde{\eta}(X)} \langle A(X) X, n_\Omega(X) \rangle. \]
Here we have used $v\equiv 0$ on $\Delta_r$, and thus $\nabla v$ only has normal component and $\nabla v = \partial_n v \cdot n_\Omega$.
Combined we conclude
\begin{equation}\label{eq:Rlk1}
	I = r\int_{\partial B_r \cap\Omega} \mu |\nabla_g v|_g^2 d V_{\partial B_r} + \int_{\Delta_r} \frac{(\partial_n v)^2}{ \tilde{\eta}(X)} \langle A(X) X, n_\Omega(X) \rangle \langle A(X)n_\Omega(X), n_\Omega(X) \rangle d\mathcal{H}^{d-1}, 
\end{equation} 
\begin{equation}\label{eq:Rlk2}
	II = r \int_{\partial B_r \cap \Omega} \mu \left(v_\rho \right)^2 ~dV_{\partial B_r} + + \int_{\Delta_r} \frac{(\partial_n v)^2}{ \tilde{\eta}(X)} \langle A(X) X, n_\Omega(X) \rangle \langle A(X)n_\Omega(X), n_\Omega(X) \rangle d\mathcal{H}^{d-1}. 
\end{equation} 

On the other hand,
\begin{align}
	\divg_g\left( \frac12 \nabla_g |X|^2 \right) & = \frac{1}{\sqrt{|g|}} \divg \left(  \frac{\sqrt{|g|}}{\tilde{\eta}(X)} A(X)X \right) \nonumber \\
	& = \frac{1}{\tilde{\eta}(X)} \divg\left( A(X) X \right) + \frac{1}{\sqrt{|g|}} \left\langle \nabla \left( \frac{\sqrt{|g|}}{\tilde{\eta}(X)} \right), A(X) X \right\rangle \nonumber \\
	& = \frac{\tr A(X)}{\tilde{\eta}(X)} + O\left(\theta(|X|) \right),\label{tmp:div3}
\end{align}
where in the last equality we have used the bounds \eqref{eq:bdul} and \eqref{eq:bdgrad}.
It follows from $|A(X)-\Id| \leq \theta(|X|)$ that
\[ \tr A(X) = d + O(\theta(|X|)), \]
and
\[ \tilde{\eta}(X) = \frac{\langle A(X)X, X \rangle}{|X|^2} = 1 + \frac{\langle ( A(X)-\Id) X, X \rangle}{|X|^2} = 1+O(\theta(|X|)). \]
Plugging into \eqref{tmp:div3} we get
\begin{equation}\label{eq:Rlk3}
	\divg_g\left( \frac12 \nabla_g |X|^2 \right) = d+O(\theta(|X|)). 
\end{equation} 

Combining \eqref{eq:Rlk1}, \eqref{eq:Rlk2}, \eqref{eq:Rlk3} and the Rellich identity \eqref{eq:Rellich}, we obtain
\begin{align*}
	r\int_{\partial B_r \cap\Omega} \mu|\nabla_g v|_g^2 ~dV_{\partial B_r}  = & 2r \int_{\partial B_r \cap\Omega} \mu \left(v_\rho \right)^2 ~ dV_{\partial B_r} + \left( d-2+O(\theta(r)) \right) \iint_{B_r \cap\Omega} \mu|\nabla_g v|_g^2 ~ dV_{g} \\
	& + \int_{\Delta_r} \frac{(\partial_n v)^2}{ \tilde{\eta}(X)} \langle A(X) X, n_\Omega(X) \rangle \langle A(X)n_\Omega(X), n_\Omega(X) \rangle d\mathcal{H}^{d-1}.
\end{align*} 
In other words, we get the desired quantity $D'(r)$:
\begin{align}
	D'(r) = & \frac{d-2+O(\theta(r))}{r} D(r) + 2 \int_{\partial B_r \cap\Omega} \mu \left(v_\rho \right)^2 ~ dV_{\partial B_r} \nonumber \\
	& + \frac{1}{r} \int_{\Delta_r} \frac{(\partial_n v)^2}{ \tilde{\eta}(X)} \langle A(X) X, n_\Omega(X) \rangle \langle A(X)n_\Omega(X), n_\Omega(X) \rangle d\mathcal{H}^{d-1}. \label{eq:derDr}
\end{align} 

Combining \eqref{eq:derHr}, \eqref{eq:derDr} and \eqref{eq:DrderHr}, we get
\begin{align*}
	\frac{N'(r)}{N(r)} & = \frac{1}{r} + \frac{D'(r)}{D(r)} - \frac{H'(r)}{H(r)} \\
	& = O\left(\frac{\theta(r)}{r} \right) + 2 \dfrac{\left( \int_{\partial B_r \cap\Omega} \mu v^2 ~ dV_{\partial B_r} \right) \cdot \left( \int_{\partial B_r \cap \Omega} \mu \left( v_\rho \right)^2 ~dV_{\partial B_r} \right) - \left( \int_{\partial B_r \cap\Omega} \mu v v_\rho ~ dV_{\partial B_r} \right)^2}{D(r)H(r)} \\
	& \qquad +\frac{1}{rD(r)} \int_{\Delta_r} \frac{(\partial_n v)^2}{ \tilde{\eta}(X)} \langle A(X) X, n_\Omega(X) \rangle \langle A(X)n_\Omega(X), n_\Omega(X) \rangle d\mathcal{H}^{d-1}.
\end{align*}	
Notice that the second term is non-negative by the Cauchy-Scharz inequality, and it equals zero if and only if 
\[ v_\rho = \lambda_r v \text{ on } \partial B_r \cap\Omega. \]
The last terms is also non-negative by the assumption on the matrix $A(X)$. And since $\partial_n v \not\equiv 0$, it equals zero if and only if $\langle A(X) X, n_\Omega(X) \rangle \equiv 0$ on $\Delta_r$.
To sum up we get
\[ \frac{N'(r)}{N(r)} \geq -C ~\frac{\theta(r)}{r}, \]
and hence 
\[ \widetilde{N}(r)= N(r) \exp\left( C\int_0^r \frac{\theta(s)}{s} ds \right) \text{ is monotone increasing with respect to } r. \]
This finishes the proof of the proposition.
\bigskip

Now we return to \textit{the proof of the Rellich identity \eqref{eq:Rellich}}. We will use two abstract formulae:
\begin{equation}\label{eq:absdiv}
	\divg_g (a\vec{Y}) = a \divg_g(\vec{Y}) + \langle \nabla_g a, \vec{Y} \rangle_g = a \divg_g(\vec{Y}) + \langle \nabla a, \vec{Y} \rangle,
\end{equation}
\begin{equation}\label{eq:absgrad}
	\langle \nabla_g \langle \vec{A}, \vec{B} \rangle_g, \vec{C} \rangle_g = \langle \nabla \langle \vec{A}, \vec{B} \rangle_g, \vec{C} \rangle = \frac{\partial A_j}{\partial x_l} g_{jk} B_k C_l + A_j g_{jk} \frac{\partial B_k}{\partial x_l} C_l + A_j \frac{\partial g_{jk}}{\partial x_l} B_k C_l.
\end{equation} 
By \eqref{eq:absdiv} the left hand side of \eqref{eq:Rellich} becomes
\begin{align*}
	\divg_g\left( \mu|\nabla_g v|_g^2 \cdot \frac12 \nabla_g |X|^2 \right) = \divg_g \left(\frac12 \nabla_g |X|^2 \right) \cdot \mu |\nabla_g v|^2_g + \left\langle \nabla_g \langle  \mu \nabla_g v, \nabla_g v \rangle_g, \frac12 \nabla_g|X|^2 \right\rangle_g.
\end{align*}
The first term on the right hand side also appears in the right hand side of \eqref{eq:Rellich}, so it suffices to compute the second term. By \eqref{eq:absgrad} we compute
\begin{align*}
	&\left\langle \nabla_g \langle  \mu \nabla_g v, \nabla_g v \rangle_g, \frac12 \nabla_g|X|^2 \right\rangle_g = \frac{\partial \left( \mu \nabla_g v \right)_j}{\partial x_l}  g_{jk} \left(\nabla_g v \right)_k \left( \frac12 \nabla_g |X|^2 \right)_l \\
	& \qquad + \qquad \mu \left( \nabla_g v \right)_j g_{jk} \frac{\partial \left(\nabla_g v \right)_k}{\partial x_l}  \left( \frac12 \nabla_g|X|^2 \right)_l + \mu \left(\nabla_g v\right)_j \frac{\partial g_{jk}}{\partial x_l} \left( \nabla_g v \right)_k \left( \frac12 \nabla_g|X|^2 \right)_l \\
	& = 2\mu \frac{\partial (\nabla_g v)_j}{\partial x_l} g_{jk} (\nabla_g v)_k \left( \frac12 \nabla_g |X|^2 \right)_l + \frac{\partial \mu}{\partial x_l} |\nabla_g v|_g^2 \left( \frac12 \nabla_g |X|^2 \right)_l + \mu (\nabla_g v)_j \frac{\partial g_{jk}}{\partial x_l} (\nabla_g v)_k \left( \frac12 \nabla_g |X|^2 \right)_l \\
	& =: E_0 + E_1 + E_2.
\end{align*}
In the second equality we use the symmetry of the matrix $g$, which follows from the symmetry of the elliptic matrix $A$.
On the other hand, by \eqref{eq:absdiv} and $\divg_g(\mu \nabla_g v) = 0$ we have
\begin{align*}
	 & \divg_g \left( \left\langle \frac12 \nabla_g|X|^2, \nabla_g v \right\rangle_g \cdot \mu \nabla_g v \right) = \left\langle \nabla_g \left\langle \frac12 \nabla_g|X|^2, \nabla_g v \right\rangle_g, \mu\nabla_g u\right\rangle_g \\
	 & \qquad =\frac{\partial \left( \frac12 \nabla_g|X|^2 \right)_j}{\partial x_l} g_{jk} (\nabla_g v)_k \mu(\nabla_g v)_l + \left( \frac12 \nabla_g|X|^2 \right)_j g_{jk} \frac{\partial(\nabla_g v)_k}{\partial x_l} \mu(\nabla_g v)_l \\
	&\qquad \qquad\qquad + \left( \frac12 \nabla_g|X|^2 \right)_j \frac{\partial g_{jk}}{\partial x_l} (\nabla_g v)_k \mu (\nabla_g v)_l \\
	& \qquad =: E_3 + E'_0 + E_4.
\end{align*}
By the symmetry of $g$ it is not hard to see $E_0 = 2\cdot E'_0$. Therefore
\begin{align*}
	& \divg_g\left( \mu|\nabla_g v|_g^2 \cdot \frac12 \nabla_g |X|^2 \right) - 2\divg_g \left( \left\langle \frac12 \nabla_g|X|^2, \nabla_g u \right\rangle_g \cdot \mu \nabla_g v \right) - \divg_g \left(\frac12 \nabla_g |X|^2 \right) \cdot \mu |\nabla_g v|^2_g \\
	& \qquad \qquad = E_1 + E_2 - 2\cdot E_3 - 2\cdot E_4.
\end{align*}
By \eqref{eq:bdul} and \eqref{eq:bdgrad} we can easily get
\[ E_1, E_2, E_4 = O\left( \theta(|X|) \right) \cdot \mu|\nabla_g v|_g^2. \]
On the other hand, the product rule, \eqref{eq:bdul}, \eqref{eq:bdgrad} and the symmetry of $g$ yield
\begin{align*}
	E_3 & = O(\theta(|X|)) \cdot \mu|\nabla_g v|_g^2 + \mu (\nabla_g v)_k \delta_{kl} (\nabla_g v)_l \\
	& = O(\theta(|X|)) \cdot \mu|\nabla_g v|_g^2 + \mu \frac{\langle A \nabla v, A \nabla v \rangle}{\tilde{\eta}^2} \\
	& = \mu |\nabla_g v|_g^2 + O(\theta(|X|)) \cdot \mu|\nabla_g v|_g^2,
\end{align*}
where we use the assumption \eqref{eq:gradA} in the last equality. This finishes the proof of the Rellich identity.
\end{proof}

To get an idea of the quantitative rigidity, we look at the sharp case.
\begin{corollary}[Rigidity of the monotonicity formula]\label{cor:rigidity}
	Suppose $A(X) \equiv \Id$ and $\theta(r) = 0$. Then $N(r)$ is a constant (denoted by $N$) if and only if
	\begin{enumerate}
		\item $v(r,\omega) = r^N v(\omega)$ is a homogeneous harmonic function of degree $N$ in $\Omega$;
		\item $\Omega$ is a cone with vertex at the origin.
	\end{enumerate} 
\end{corollary} 
\begin{remark}
	So far we do not know if $N$ is a positive integer or not, so it is perhaps a misnomer to call it $N$. But in Subsection \ref{subsec:blowupofv} we will show that if the domain $\Omega$ is $C^1$, it blows up to an upper half-space and thus the homogeneity $N$ of a harmonic function must be a positive integer.	
\end{remark}

\begin{proof}
	By \eqref{eq:derNr} and since $\theta(r) = 0$, $N(r)$ is a constant if and only if both $R_h(r)$ and $R_b(r)$ vanish for all $r$. By \eqref{eq:Rhr} $R_h(r)$ equals zero if and only if 
			\[ v_\rho - N(r) \frac{v}{r} \equiv 0 \text{ on } \partial B_r \cap \Omega.  \]
			When $A(X) = \Id$, we have $v_\rho = \langle \nabla_g v, \frac{X}{|X|} \rangle = \langle \nabla v, \frac{X}{|X|} \rangle$. Hence it follows that $v(r,\omega) = r^N v(\omega)$ is a homogeneous harmonic function of degree $N$.
			Recall that we assume $v$ is a non-trivial function, hence $R_b(r)$ equals zero if and only if 
			\[ \langle A(X)X, n_\Omega(X) \rangle \equiv 0, \]
			that is to say, $\Omega$ is a cone (not necessarily convex) with vertex at the origin.
\end{proof}

\begin{corollary}[The doubling property of $H(r)$]\label{cor:Hrbd}
	The quantity 
	\[ \widetilde{H}(r):= \frac{H(r)}{r^{d-1}} \]
	satisfies
	\begin{equation}\label{eq:Hrcmpl}
		\left( \frac{r_2}{r_1} \right)^{2\widetilde{N}(r_1) \exp \left(-C \int_0^{r_2} \frac{\theta(s)}{s} ~ds \right)} \exp\left(-C\int_{r_1}^{r_2} \frac{\theta(s)}{s} ~ds \right) \leq \frac{\widetilde{H}(r_2)}{\widetilde{H}(r_1)} 
	\end{equation}
	\begin{equation}\label{eq:Hrcmpu}
		 \frac{\widetilde{H}(r_2)}{\widetilde{H}(r_1)} \leq \left( \frac{r_2}{r_1} \right)^{2\widetilde{N}(r_2) \exp \left( -C\int_0^{r_1} \frac{\theta(s)}{s} ~ds \right)} \exp \left( C\int_{r_1}^{r_2} \frac{\theta(s)}{s} ~ds \right)
	\end{equation}
	for any two radii $r_1, r_2$ sufficiently small satisfying $0<r_1<r_2$.
	
\end{corollary}
\begin{proof}
	Combining \eqref{eq:derHr} and \eqref{eq:DrderHr} we get
	\begin{equation}\label{eq:Hrder}
		\frac{H'(r)}{H(r)} = \frac{d-1}{r} + O\left( \frac{\theta(r)}{r} \right) + \frac{2D(r)}{H(r)} = \frac{d-1}{r} + O\left( \frac{\theta(r)}{r} \right) + \frac{2N(r)}{r}.
	\end{equation} 	
	Therefore there exists a constant $C>0$ such that
	\[ \frac{d}{dr} \log \left( \frac{H(r)}{r^{d-1}} \exp\left( C\int_0^r \frac{\theta(s)}{s} ~ds \right) \right) \geq \frac{2N(r)}{r} = \frac{2\widetilde{N}(r)}{r} \exp\left( -C\int_0^r \frac{\theta(s)}{s} ~ds \right). \]
	Integrating the above differential inequality on the interval $[r_1, r_2]$ and using the monotonicity of $\widetilde{N}(r)$, we obtain
	\begin{align*}
		\log \widetilde{H}(r_2) - \log \widetilde{H}(r_1) & \geq \int_{r_1}^{r_2} \frac{2\widetilde{N}(r)}{r} \exp\left( -C\int_0^r \frac{\theta(s)}{s} ~ds \right) ~dr \\
		& \geq 2\widetilde{N}(r_1) \exp\left(-C\int_0^{r_2} \frac{\theta(s)}{s} ~ds \right) \left( \log r_2 - \log r_1 \right).
	\end{align*} 
	The desired inequality \eqref{eq:Hrcmpl} follows from taking the exponential of the above. The proof of \eqref{eq:Hrcmpu} is similar, except that we use the upper bound $\widetilde{N}(r) \leq \widetilde{N}(r_2)$.
\end{proof}

\section{Reduction from a Dini domain to an almost convex domain}\label{sec:reduce}

In this section, we recall the reduction in \cite{AE} of the Laplacian operator in a Dini domain $D$ to a divergence-form elliptic operator in a Lipschitz domain $\Omega$, so that the elliptic operator satisfies the assumptions in the beginning of Section \ref{sec:monotonicity}. We focus on how the reduction map changes relevant quantities of the solution, and in particular how the frequency function discussed in Section \ref{sec:monotonicity} controls the \textit{symmetry} of the original harmonic function, modulo a Dini-type error term.

Let $D$ be a Dini domain (see Definition \ref{def:Dini}) with parameter $\theta$. We set
\[ \tilde{\theta}(r) = \frac{1}{\log^2 2} \int_r^{2r} \frac1t \int_t^{2t} \frac{\theta(s)}{s} ds ~dt, \quad \text{and } \quad \alpha(r) = 3\frac{d}{dr} \left(r\tilde{\theta}(r) \right). \]
Simple computations show that 
\[ \theta(r) \leq \tilde{\theta}(r) \leq \theta(4r), \quad 3\theta(r) \leq \alpha(r) \leq 13\theta(4r). \]
Let $u$ be a non-trivial solution to $\Delta u=0$ in the domain $D$. Assume that $0 \in \partial D$, $D$ is graphical inside the ball $B_{5R}(0)$ and $\theta(8R)<1/13$ for some $R>0$.\footnote{Throughout the paper we will often require \textit{the scale to be sufficiently small}. Unless otherwise specified it always means that the radius is less or equal to $R$ satisfying the assumption here.} Suppose $u$ satisfies $u=0$ on $ B_{5R}(0) \cap \pD$. We consider the following map
\begin{equation}\label{def:Psi}
	\Psi: X=(x,x_d) \in \RR^{d-1} \times \RR \mapsto (x, x_d + 3|X|\tilde{\theta}(|X|)) \in \RR^d. 
\end{equation} 
Let $X_0=(x_0,\varphi(x_0)) \in B_R(0) \cap \pD $ be arbitrary, and we define $\Psi_{X_0}(X):= X_0 + \Psi(X)$. Then
\begin{equation}\label{eq:DPsi}
	D\Psi_{X_0}(X) = D\Psi(X) = \begin{pmatrix}
	 1 & & & \\	 
	  & \ddots & & \\
	  & & 1 & \\
	\frac{\alpha(|X|)}{|X|} x_1 & \cdots &\frac{\alpha(|X|)}{|X|} x_{n-1} & 1+ \frac{\alpha(|X|)}{|X|} x_d 
\end{pmatrix} = \Id + O(\alpha(|X|)). 
\end{equation} 
In particular
\begin{equation}\label{eq:detDPsi}
	1-\alpha(|X|) \leq \det D\Psi(X) = 1+ \frac{\alpha(|X|)}{|X|} x_d \leq 1+\alpha(|X|), 
\end{equation} 
and hence $\Psi$ is invertible in $B_{2R}(0)$.
Set $v(X) = u(\Psi_{X_0}(X))$ and
\[ \Omega_{X_0} = \left\{X=(x,x_d) \in \RR^d: x_d > \varphi(x+x_0) - \varphi(x_0)-3|X|\tilde{\theta}(|X|) \right\}. \]
We have that $u$ is a solution to $\Delta u = 0$ in $D\cap B_{2R}(X_0)$ if and only if $v$ is a solution to $\divg(A(X)\nabla v)=0$ in $\Omega_{X_0} \cap B_{2R}(0)$, where
\begin{equation}\label{def:A}
	A(X) = \det D\Psi(X) \cdot \left(D\Psi(X) \right)^{-1} \left(\left(D\Psi(X) \right)^{-1}\right)^t 
\end{equation} 
is a symmetric, elliptic matrix whose ellipticity constants are $1\pm O(\alpha(|X|))$. Moreover, following the same computation as in \cite{AE} we know that $A(X)$ satisfies the desired assumption \eqref{eq:gradA} (with a constant multiple of $\theta(4r)$ in place of $\theta(r)$) and that
\[ \langle A(X)X, n_{\Omega_{X_0}}(X) \rangle \gtrsim |X|\theta(|X|) \geq 0. \]

To summarize Section \ref{sec:monotonicity} and the discussion above, we have proven the following. For any $X_0 \in B_R(0) \cap \pD $, after the transformation map $\Psi_{X_0}$ the function $v_{X_0}: = u \circ \Psi_{X_0}$ satisfies the divergence-form elliptic operator with coefficient matrix $A(X)$, and $A(X)$ satisfies the assumptions in the beginning of Section \ref{sec:monotonicity}. Therefore the frequency function for $v_{X_0}$ centered at $0=\Psi_{X_0}^{-1}(X_0)$ is almost monotone. We set the notation
\begin{equation}\label{def:freqDini}
	N_{X_0}: = \lim_{r\to 0+} \widetilde{N}(v_{X_0}, r) \quad \text{ and } \quad N_{X_0}(r): = \widetilde{N}(v_{X_0},r)
\end{equation} 
to emphasize the dependence of the modified frequency function $\widetilde{N}$ on the corresponding base point $X_0\in B_R(0) \cap \pD$. 

Recall that we studied the divergence-form elliptic operator in Section \ref{sec:monotonicity} using an intrinsic metric $g$. We will show that
\begin{equation}\label{eq:uvhom}
	\left| \langle \nabla_g v(Y), Y-0 \rangle - \langle \nabla u(Z), Z-X_0 \rangle \right| \lesssim |Y| \theta(4|Y|) |\nabla v(Y)|.
\end{equation}
where $Z=\Psi_{X_0}(Y) \in D$ is the image of $Y \in \Omega_{X_0}$. This estimate will be used in Section \ref{sec:L2approx}.
By definitions
\begin{equation}\label{eq:etabd}
	\tilde{\eta}(Y) = \frac{\langle A(Y)Y, Y \rangle}{|Y|^2} = \det D\Psi(Y) \frac{\left|\left(\left(D\Psi(Y) \right)^{-1}\right)^t Y \right|^2  }{|Y|^2} = \det D\Psi(Y) \cdot \left( 1+O(\alpha(|Y|)) \right), 
\end{equation} 
\begin{equation}\label{eq:YPsiY}
	|Y-\Psi(Y) | \leq 3|Y|\tilde{\theta}(|Y|), 
\end{equation} 
\[ \nabla v(Y) = \left( D\Psi(Y) \right)^t \nabla u(Z) = \left(\Id + O(\alpha(|Y|)) \right) \nabla u(Z) \]
Finally,
\begin{align*}
	\nabla_g v(Y) & = \frac{1}{\tilde{\eta}(Y)} A(Y)\nabla v(Y) \\
	& = (1+O(\alpha(|Y|))) \cdot \left(D\Psi(Y) \right)^{-1} \left(\left(D\Psi(Y) \right)^{-1}\right)^t \cdot \left(D\Psi(Y) \right)^t \nabla u(Z) \\
	& = (1+O(\alpha(|Y|))) \cdot \left(D\Psi(Y) \right)^{-1} \nabla u(Z) \\
	& = (1+O(\alpha(|Y|))) \cdot (\Id + O(\alpha(|Y|))) \nabla u(Z)
\end{align*}
and thus
\begin{align*}
	\langle \nabla_g v(Y),Y \rangle & =  \left\langle (1+O(\alpha(|Y|))) \cdot (\Id + O(\alpha(|Y|))) \nabla u(Z), \Psi(Y) + O(|Y|\tilde{\theta}(|Y|)) \right\rangle \\
	& = \langle \nabla u(Z), Z-X_0 \rangle + O\left( |Y| \theta(4|Y|)  \right) \cdot \left|\nabla u(Z) \right| \\
	& = \langle \nabla u(Z), Z-X_0 \rangle + O\left( |Y| \theta(4|Y|)  \right) \cdot \left|\nabla v(Y) \right|.
\end{align*}

\section{Blow up analysis and compactness for boundary points}\label{sec:blowup}


The main goal of this section is to show that the frequency function centered at $X_0 \in \pD$ carries a lot of information about the local behavior of the harmonic function $u$ near $X$. We also remark that the same analysis can also be carried out for interior points, after we establish the monotonicity formula of the frequency function in Proposition \ref{prop:intmonotonicity}. 

We first prove the following lemma, which essentially says the frequency function of scale $r$ takes comparable values in an $r$-neighborhood. In particular, we get a upper bound on the frequency function.
\begin{lemma}\label{lm:freqbd}
	There exist constants $C_2, C_3>0$ such that the following holds.
	Let $X_0, X'_0 \in B_{R}(0) \cap \pD $ be arbitrary. Suppose they satisfy $|X_0 - X'_0| \leq \frac{r}{20}$ for some $r$ sufficiently small (or to be precise, $r$ small enough so that $r\leq 3R$ and $\theta(4r)<1/30$). Then
	\[ N_{X_0}\left( \frac{r}{4} \right) \leq C_2 +  C_3 N_{X'_0} \left(\frac{4r}{3} \right), \]
	as long as the right hand side is finite.
	
	In particular, suppose the frequency function at the origin satisfies $N_0(4R)\leq \Lambda <+\infty$, then 
	\begin{equation}\label{eq:freqbdR}
		N_{X_0}\left( \frac34 R \right) \leq C_2 + C_3\Lambda,\quad \text{ for any } X_0 \in B_{\frac{3}{20}R}(0) \cap \pD. 
	\end{equation} 
\end{lemma}
\begin{proof}
	We claim that
	\begin{equation}\label{eq:uvL2}
		\frac12 \iint_{B_{\frac{s}{2}}(X_0)} u^2 ~dX \leq \iint_{B_s\cap \Omega_{X_0}} \mu v_{X_0}^2 ~dV_g \leq 2 \iint_{B_{2s}(X_0) \cap D} u^2 ~dX, 
	\end{equation} 
for all $X_0 \in B_R(0) \cap \pD$ and all $s>0$ small enough.
	By \eqref{eq:etabd}, a change of variables and \eqref{eq:detDPsi},
	\begin{align}
		\iint_{B_s\cap \Omega_{X_0}} \mu v_{X_0}^2 ~dV_g = \iint_{B_s\cap \Omega_{X_0}} \tilde{\eta} v_{X_0}^2 ~dY & = \left( 1+ O\left( \theta(4s) \right) \right) \iint_{B_s \cap \Omega_{X_0}} u^2(\Psi_{X_0}(Y)) ~dY \nonumber \\
		& = \left( 1+ O\left( \theta(4s) \right) \right) \iint_{\Psi_{X_0}(B_s) \cap D} u^2(X) \left(1+O\left(\theta(4s) \right) \right)^{-1}  ~dX. \label{tmp:doublinguv}
	\end{align}
	Notice that for any $\rho>0$ sufficiently small, the transformation $\Psi$ maps the sphere $\partial B_\rho$ bijectively to the sphere $\partial B_\rho + 3\rho \tilde{\theta}(\rho)e_d$, which is the sphere $\partial B_\rho$ shifted upwards by $3\rho\tilde{\theta}(\rho)$ in the positive $e_d$-direction. Hence
	\[ \Psi_{X_0}(B_s) = X_0 + \Psi(B_s) \subset B_{s+3s\tilde{\theta}(s)}(X_0) \subset B_{\frac{11}{10}s}(X_0) \]
	as long as $s$ is sufficiently small (so that $\theta(4s)<1/30$).
	On the other hand, recall that $\Psi$ is invertible in $B_{2R}(0)$. That is to say for any $X'\in B_{\frac{9}{10}s}$ we can find some $X$ in the neighborhood of the origin such that $\Psi(X)=X'$. Moreover, such $X$ satisfies
	\[ |X| \leq \|\left( D\Psi \right)^{-1}\|_\infty |X'| \leq \left(1+\alpha\left( \frac{9}{10}s \right) \right) \frac{9}{10}s < s. \]	
	That is to say $B_{\frac{9}{10}s} \subset \Psi(B_s)$, and hence 
	\[ B_{\frac{9}{10}s}(X_0) \subset X_0 + \Psi(B_s) = \Psi_{X_0}(B_s). \] 
	Plugging back into \eqref{tmp:doublinguv}, we conclude that
	\begin{equation}\label{eq:doublinguv}
		\frac12 \iint_{A_{\frac{11}{20} r ,\frac{9}{10}r }(X_0) \cap D} u^2 ~dX \leq \iint_{A_{\frac{r}{2},r} \cap \Omega_{X_0}} \mu v_{X_0}^2 ~dV_g \leq 2\iint_{A_{\frac{9}{20}r, \frac{11}{10}r }(X_0) \cap D} u^2 ~dX,
	\end{equation}
	where we use $A_{r_1, r_2}(X_0)$ with $0<r_1<r_2$ to denote the annulus $B_{r_2}(X_0) \setminus B_{r_1}(X_0)$.
	
	Let $\rho>0$ be sufficiently small, and let $a>1$ be a constant to be chosen later.
	By Corollary \ref{cor:Hrbd}, we have that
	\begin{align}
		\frac{H_{X_0}(a\rho)}{H_{X_0}(\rho)} = a^{n-1} \frac{\widetilde{H}_{X_0}(a\rho)}{\widetilde{H}_{X_0}(\rho)} & \leq a^{n-1} a^{2 N_{X_0}(a\rho)\exp \left(-C\int_0^\rho \frac{\theta(4s)}{s} ~ds \right)} \exp \left(C\int_{\rho}^{a\rho} \frac{\theta(4s)}{s} ~ds \right) \nonumber \\
		& \leq C_1 a^{n-1} a^{2 N_{X_0}(a\rho)},\label{eq:doublingsolidHru}
	\end{align}
	where we write $C_1 = \exp\left(C\int_0^{4R} \frac{\theta(4s)}{s} ~ds \right)<+\infty$. Similarly, we also have the lower bound
	\begin{align}
		\frac{H_{X_0}(a\rho)}{H_{X_0}(\rho)} \geq a^{n-1} a^{2N_{X_0}(\rho)\exp \left(-C\int_0^{a\rho} \frac{\theta(4s)}{s} ~ds \right)} \exp \left(-C\int_{\rho}^{2\rho} \frac{\theta(4s)}{s} ~ds \right) \geq C_1^{-1} a^{n-1} a^{2N_{X_0}(\rho)/C_1}.\label{eq:doublingsolidHrl}
	\end{align}
	
	Using \eqref{eq:doublinguv} twice, we get
	\begin{align}
		\frac{ \mathlarger{\iint}_{A_{\frac{r}{2},r} \cap \Omega_{X_0}} \mu v^2_{X_0} ~dV_g}{ \mathlarger{\iint}_{A_{\frac{r}{4},\frac{r}{2}} \cap \Omega_{X_0}} \mu v^2_{X_0} ~dV_g} 
		\leq 4 \frac{ \mathlarger{\iint}_{A_{\frac{9}{20}r, \frac{11}{10}r }(X_0) \cap D} u^2 ~dX}{ \mathlarger{\iint}_{A_{\frac{11}{40}r,\frac{9}{20}r}(X_0) \cap D} u^2 ~dX} 
		& \leq 4 \frac{\mathlarger{\iint}_{A_{\frac{2}{5}r, \frac{23}{20} r}(X'_0) \cap D} u^2 ~dX}{ \mathlarger{\iint}_{A_{\frac{13}{40}r, \frac{2}{5}r}(X'_0) \cap D} u^2 ~dX} \nonumber \\
		& \leq 16  \frac{ \mathlarger{\iint}_{A_{\frac{9}{25}r,\frac{253}{200}r} \cap \Omega_{X'_0}} \mu v^2_{X'_0} ~dV_g}{ \mathlarger{\iint}_{A_{\frac{143}{400}r, \frac{9}{25}r} \cap \Omega_{X'_0}} \mu v^2_{X'_0} ~dV_g}.\label{tmp:uvdiffbase}
	\end{align}
	Using the lower bound in \eqref{eq:doublingsolidHrl} with $a=2$ and the monotonicity of $N_{X_0}(\cdot)$, we have
	\begin{align}
		\frac{ \mathlarger{\iint}_{A_{\frac{r}{2},r} \cap \Omega_{X_0}} \mu v^2_{X_0} ~dV_g}{ \mathlarger{\iint}_{A_{\frac{r}{4},\frac{r}{2}} \cap \Omega_{X_0}} \mu v^2_{X_0} ~dV_g}
		= \frac{\mathlarger{\int}_{\frac{r}{2}}^{r} H_{X_0}(\rho) ~d\rho }{\mathlarger{\int}_{\frac{r}{4}}^{\frac{r}{2}} H_{X_0}(\rho) ~d\rho }
		= \frac{2\mathlarger{\int}_{\frac{r}{4}}^{\frac{r}{2}} H_{X_0}(2\rho) ~d\rho }{\mathlarger{\int}_{\frac{r}{4}}^{\frac{r}{2}} H_{X_0}(\rho) ~d\rho }
		\geq C_1^{-1} 2^{n} 2^{2N_{X_0}(\frac{r}{4})/C_1 }.\label{eq:uvdiffbasel}
	\end{align}
	Similarly, we want to use the upper bound in \eqref{eq:doublingsolidHru} to bound the right hand side of \eqref{tmp:uvdiffbase}. However, since the integral region of the numerator is not exactly a rescaling of the integral region of the denominator, we need to first break down the integral on top. To do that we denote $\ell:= \frac{9}{25}r-\frac{143}{400}r = \frac{1}{400}r$, 
	\[ r_1 = \frac{9}{25}r \quad \text{ and } \quad a_1 = a:=\frac{r_1}{\frac{143}{400}r} = \frac{144}{143}, \]
	and inductively
	\[ r_j = r_1 + a \ell \cdot(j-1) \quad \text{ and } \quad a_j = \frac{r_j}{\frac{143}{400}r} = \frac{144}{143} + \frac{144}{143^2}(j-1) \]
	for $j=1, \cdots, 361$ (since $r_{361}$ is the last one to be smaller than $\frac{253}{200}r$). For each $j$ we can use the upper bound in \eqref{eq:doublingsolidHru} with coefficient $a_j$ to bound
	\begin{align*}
		\frac{ \mathlarger{\iint}_{A_{r_j,r_{j+1}} \cap \Omega_{X'_0}} \mu v^2_{X'_0} ~dV_g}{ \mathlarger{\iint}_{A_{\frac{143}{400}r, \frac{9}{25}r} \cap \Omega_{X'_0}} \mu v^2_{X'_0} ~dV_g} = \frac{\mathlarger{\int}_{r_j}^{r_{j+1}} H_{X'_0}(\rho) ~d\rho }{\mathlarger{\int}_{\frac{143}{400}r}^{\frac{9}{25}r} H_{X'_0}(\rho) ~d\rho }
		\leq C_1 a_j^{n} a_j^{2N_{X'_0}(r_{j+1})} \leq C_1 4^n 4^{2N_{X'_0}(\frac43 r ) }.
	\end{align*}
	 Finally
	\begin{align}
		\frac{ \mathlarger{\iint}_{A_{\frac{9}{25}r,\frac{253}{200}r} \cap \Omega_{X'_0}} \mu v^2_{X'_0} ~dV_g}{ \mathlarger{\iint}_{A_{\frac{143}{400}r, \frac{9}{25}r} \cap \Omega_{X'_0}} \mu v^2_{X'_0} ~dV_g} 
		\leq \sum_{j=1}^{361} \frac{ \mathlarger{\iint}_{A_{r_j,r_{j+1}} \cap \Omega_{X'_0}} \mu v^2_{X'_0} ~dV_g}{ \mathlarger{\iint}_{A_{\frac{143}{400}r, \frac{9}{25}r} \cap \Omega_{X'_0}} \mu v^2_{X'_0} ~dV_g} \leq 361 C_1 4^n 4^{2N_{X'_0}(\frac43 r ) }. \label{eq:uvdiffbaseu}
	\end{align}
	Combining \eqref{tmp:uvdiffbase}, \eqref{eq:uvdiffbasel} and \eqref{eq:uvdiffbaseu}, we conclude
	\[ N_{X_0}\left(\frac{r}{4} \right) \leq C_2 + C_3 N_{X'_0}\left(\frac{4r}{3} \right), \]
	where the constants $C_2, C_3$ depends on $n$ and the value of $C_1$. 	
\end{proof}

Following the proof of Lemma \ref{lm:freqbd} and the bound \eqref{eq:freqbdR} of the frequency function, we get the following lemma which will become useful later.
\begin{lemma}\label{lm:vthinlayer}
	Assume $N_0(4R)\leq \Lambda < +\infty$.
	For any $X_0 \in \pD \cap B_{R}(0)$, we have
	\begin{equation}\label{eq:vthinlayer}
		\frac{\mathlarger{\iint}_{B_{\tau \rho}(\partial B_\rho) \cap \Omega_{X_0}} \mu v_{X_0}^2 ~dV_g }{\mathlarger{\iint}_{B_\rho \cap \Omega_{X_0}} \mu v_{X_0}^2 ~dV_g } \leq C_1 \left[ 1 - \left( \frac{1-\tau}{1+\tau} \right)^{d+2N_{X_0}((1+\tau)\rho)} \right]. 
	\end{equation} 
	In particular if $|X_0| \leq \frac{3}{20} R$ and $\rho< \frac34 R$, 
	\[\text{as } \tau \to 0, \text{ the ratio } \frac{\mathlarger{\iint}_{B_{\tau \rho}(\partial B_\rho) \cap \Omega_{X_0}} \mu v_{X_0}^2 ~dV_g }{\mathlarger{\iint}_{B_\rho \cap \Omega_{X_0}} \mu v_{X_0}^2 ~dV_g } \text{ converges to } 0 \text{ uniformly for } X_0 \text{ and } \rho. \]
\end{lemma}
\begin{proof}
	Using polar coordinates, we have
	\begin{equation}\label{tmp:thinlayer1}
		\frac{\mathlarger{\iint}_{B_{\tau \rho}(\partial B_\rho) \cap \Omega_{X_0}} \mu v_{X_0}^2 ~dV_g }{\mathlarger{\iint}_{B_\rho \cap \Omega_{X_0}} \mu v_{X_0}^2 ~dV_g } \leq \frac{\mathlarger{\int}_{(1-\tau)\rho}^{(1+\tau)\rho} H_{X_0}(s) ~ds}{\mathlarger{\int}_0^{(1-\tau)\rho} H_{X_0}(s) ~ds}.
	\end{equation}
	Let $a=\frac{(1+\tau)\rho}{(1-\tau)\rho} = \frac{1+\tau}{1-\tau}>1$, then 
	\begin{align}
		\frac{\mathlarger{\int}_0^{(1-\tau)\rho} H_{X_0}(s) ~ds}{\mathlarger{\int}_{(1-\tau)\rho}^{(1+\tau)\rho} H_{X_0}(s) ~ds} = \sum_{j=0}^\infty \frac{\mathlarger{\int}_{a^{-(j+1)} (1-\tau)\rho}^{a^{-j} (1-\tau)\rho} H_{X_0}(s) ~ds}{\mathlarger{\int}_{(1-\tau)\rho}^{(1+\tau)\rho} H_{X_0}(s) ~ds} & \geq \sum_{j=0}^\infty \frac{1}{C_1 a^{j\left[ d+2N_{X_0}((1+\tau)\rho \right]}} \nonumber \\
		& = \frac{1}{C_1} \frac{1}{1-\left( \frac{1-\tau}{1+\tau} \right)^{d+2N_{X_0}((1+\tau)\rho)} }. \label{tmp:thinlayer2}
	\end{align}
	Combining \eqref{tmp:thinlayer1} and \eqref{tmp:thinlayer2} we get the desired inequality \eqref{eq:vthinlayer}. 
	
	By the assumption and \eqref{eq:freqbdR}, we have that the exponent
	\[ d+ 2N_{X_0}((1+\tau)\rho) \leq C(d, \Lambda) <+\infty \]
	for $\tau$ sufficiently small. Therefore the ratio tends to zero as $\tau \to 0$.
\end{proof}


In the rest of this section, we first study what the monotonicity formula in Proposition \ref{prop:monotonicity} tells us about the tangent function of $v$. Then we switch gears to the harmonic function $u$ in the Dini domain and study its tangent function, by looking at the transformation maps between $u$ and $v$.

\subsection{Blow up analysis in $\Omega_{X_0}$ using the monotonicity formula}\label{subsec:blowupofv}
Let $X_0 \in B_R(0) \cap \pD$ be fixed. Since $v_{X_0}$ has boundary value zero, we may extend it by zero across $\Omega_{X_0}$. For any $r>0$ we define
\begin{equation}\label{def:vr}
	v_r(Y) = T_r v_{X_0} = \frac{v_{X_0}(rY)}{\left(\frac{1}{r^d} \mathlarger{\iint}_{ B_{r}} v_{X_0}^2 ~dX \right)^{\frac12} }, \quad \text{ for } Y \in \frac{\Omega_{X_0}}{r} \cap B_5. 
\end{equation} 
The normalizing factor in the denominator guarantees that 
\begin{equation}\label{eq:vrL2norm}
	1\leq \iint_{\frac{\Omega_{X_0}}{r} \cap B_5} |v_r|^2 ~dY= \frac{\iint_{B_{5r}} v_{X_0}^2 ~dX}{\iint_{B_r} v_{X_0}^2 ~dX} \leq C_1 5^{d+2N_{X_0}(5r)} \leq C(d, \Lambda) <+\infty, 
\end{equation} 
by \eqref{eq:Hrcmpu} and \eqref{eq:freqbdR}, as long as $|X_0| \leq \frac{3}{20}R$ and $r\leq \frac{3}{20}R$.
In particular $\|v_r\|_{L^2}$ is uniformly bounded. Heuristically if the function $v$ can be written as the sum of homogeneous functions, then as $r\to 0+$ we have $v_r$ approaches the leading order homogeneous function (modulo a normalization factor).
We first look at how the above rescaling by $r$ affects the frequency function.
Notice that not only is $v_r$ defined in the domain $\frac{\Omega}{r}$, which is different from the domain $\Omega$ of $v$, it also satisfies a different divergence-form elliptic operator:
\[ \divg(A_r(Y) \nabla v_r) = 0 \text{ in } \frac{\Omega}{r}, \quad \text{ where the matrix } A_r(Y):= A(rY). \]
So when we write the frequency function for $v_r$ at the scale $s>0$, to be more rigorous we should write $N(\Omega/r, A_r, v_r, s)$. But for the sake of simplicity we will just write $N(v_r, s)$ and keep in mind that the definition depends on the domain and elliptic operator (and thus the corresponding intrinsic metric) for the function $v_r$ . By \eqref{eq:DrHrinEm} and a change of variables, we get
\begin{equation}\label{eq:freqscaleinv}
	N(v_r, s) = \frac{s D(v_r, s)}{H(v_r, s)} = \frac{\frac{s}{r^{d-2}} D(v_{X_0}, rs) }{\frac{1}{r^{d-1}} H(v_{X_0}, rs) } = \frac{rs D(v,rs)}{H(v,rs)} = N(v_{X_0},rs). 
\end{equation} 
Fix the scale $s>0$, we get
\[ N(v_{X_0},rs) = \widetilde{N}(v_{X_0},rs) \exp\left(-C\int_0^{rs} \frac{\theta(4\tau)}{\tau} ~d\tau \right) \rightarrow \lim_{\rho \to 0+} \widetilde{N}(v_{X_0},\rho) =: N_{X_0}, \quad \text{ as } r\to 0. \]
Hence
\begin{equation}\label{eq:freqarescaling}
	N(v_r, s) \rightarrow N_{X_0} \quad \text{ as } r \to 0.
\end{equation} 

On the other hand, we have the following lemma:
\begin{lemma}\label{lm:vrunifgrad}
	Assume that $N_0(4R)\leq \Lambda +\infty$. Then there exists a universal constant $C'>0$ (depending on the values of $d, C_1$ and $\Lambda$), such that for any $X_0 \in \pD$ satisfying $|X_0|< \frac{3}{20}R$ and any $r<\frac{1}{36} R$, we have
	\[ |\nabla v_r(Y)| \leq C', \quad \text{ for all } Y \in \frac{\Omega_{X_0}}{r} \cap B_5. \]	
\end{lemma} 
\begin{proof}
	By \eqref{eq:DPsi} and the chain rule,
\begin{equation}\label{tmp:vr1}
	|\nabla v_r(Y)| = \frac{\left|\nabla \left( u(X_0 + \Psi(rY)) \right) \right|}{ \left( \frac{1}{r^d} \mathlarger{\iint}_{ B_{r} \cap \Omega_{X_0}} v_{X_0}^2 ~dX \right)^{\frac12} } \lesssim \frac{r |\nabla u(Z)| }{\left( \frac{1}{r^d} \mathlarger{\iint}_{ B_{r} \cap \Omega_{X_0}} v_{X_0}^2 ~dX \right)^{\frac12} },
\end{equation}
where $Z=X_0 + \Psi(rY)\in B_{\frac{16}{3} r}(X_0)$. (The elliptic variant of \cite[Theorem 1.4.10]{CK} shows that $\nabla u \in C\left(\overline{D} \cap B_{5R}(0)\right)$, with a modulus of continuity depending on the Dini parameter. In fact, the parabolic theorems in \cite{CK} imply the elliptic ones by extending the elliptic solution independent of the $t$-variable.) 
We then use the gradient estimate in \cite[Theorem 1.4.3]{CK}. (In \cite{CK} the authors work with a Dini-type elliptic operator in the upper half-space, and we work with the Laplacian in a Dini domain. These two settings are related by a simple change of variables $(x,x_d )\in D \mapsto (x, x_d - \varphi(x)) \in \RR^d_+$, and the Laplacian in $D$ becomes a divergence-form elliptic operator in $\RR^d_+$ with coefficient matrix $\begin{pmatrix}
	\Id & -\nabla \varphi \\
	(-\nabla \varphi)^t & 1+|\nabla \varphi|^2
\end{pmatrix} $.) After a rescaling, \cite[Theorem 1.4.3]{CK} implies 
\begin{equation}\label{tmp:vr4}
	\sup_{B(X_0, \frac{16}{3} r) \cap D} |\nabla u| \lesssim \frac{1}{r} \left( \fiint_{B(X_0, \frac{17}{3} r) \cap D} u^2 ~dX \right)^{\frac12} \lesssim  \frac{1}{r} \left( \fiint_{B_{6r} \cap \Omega_{X_0} } \mu v_{X_0}^2 ~dV_g \right)^{\frac12}
\end{equation} 
By the doubling property \eqref{eq:Hrcmpu} and \eqref{eq:freqbdR},
\begin{equation}\label{tmp:vr6}
	\frac{\mathlarger{\iint}_{B_{6r} \cap \Omega_{X_0} } \mu v_{X_0}^2 ~dV_g }{\mathlarger{\iint}_{B_{r} \cap \Omega_{X_0} } \mu v_{X_0}^2 ~dV_g } \leq C_1 6^{d+2N_{X_0}(6r)} \leq C(C_1, \Lambda ) < +\infty
\end{equation}
if $|X_0|< \frac{3}{20} R$ and $r< \frac{R}{8}$.
Finally combining \eqref{tmp:vr1}, \eqref{tmp:vr4}, \eqref{eq:etabd} and \eqref{tmp:vr6} we conclude that
\[ |\nabla v_r(Y)| \leq C \left( \frac{ \mathlarger{\iint}_{B_{6r} \cap \Omega_{X_0} } \mu v_{X_0}^2 ~dV_g }{ \mathlarger{\iint}_{B_{r} \cap \Omega_{X_0} } v_{X_0}^2 ~dX } \right)^{\frac12} \leq C'<+\infty, \quad \text{ for all } |Y|\leq 5. \]
\end{proof}

Let $r_i$ be a sequence such that $r_i \to 0+$. Since
\[ \Omega_{X_0} = \{(x,x_d): x_d> \varphi(x+x_0)-\varphi(x_0) - 3|X| \tilde{\theta}(|X|) \}, \]
with $X_0 = (x_0, \varphi(x_0))$, it is not hard to see that inside $B_5$,
\[ \frac{\Omega_{X_0}}{r_i} \text{ converges graphically to the half-space } \Omega_\infty :=\{(y,y_d): y_d > \nabla \varphi(x_0) \cdot y \}. \]
Throughout the paper, we say a sequence of domains $D_j:=\{(x,x_d): x_d > \psi_i(x)\} $ \emph{converges graphically} to a domain $D_\infty$, if $\psi_i$ converges locally uniformly to a function $\psi_\infty$ and $D_\infty$ is the domain above the graph of $\psi_\infty$, i.e. $D_\infty = \{(x,x_d): x_d > \psi_\infty(x) \}$.

Since we also have $v_{r_i}(0)=0$, it follows from the Arzela-Ascoli Theorem that modulo passing to a subsequence
\[ v_{r_i} \text{ converges uniformly to a function } v_\infty \text{ in } B_5, \]
and $v_\infty = 0$ on $B_5 \setminus \Omega_\infty$. Note that $v_\infty \not\equiv 0$ because of \eqref{eq:vrL2norm}. Moreover we also have $\nabla v_{r_i} \rightharpoonup \nabla v_\infty$ in $L^2$ and thus $v_\infty$ is a harmonic function in $\Omega_\infty$.
In fact, we also have that $\nabla v_{r_i}$ converges to $\nabla v_\infty$ strongly in $L^2$. Firstly let $K$ be a compact set in $\Omega_\infty \cap B_5$. Then $K$ is also compactly contained in $\frac{\Omega_{X_0}}{r_i}$ for all $i$ sufficiently large. In $K$, the function $v_{r_i} - v_\infty$ satisfies the differential equation
\[ \divg \left(A_{r_i}(Y) \nabla (v_{r_i} - v_\infty) \right) = \divg \left((\Id - A_{r_i}(Y)) \nabla v_\infty \right), \]
where $A_{r_i}(Y) = A(r_i Y)$.
We apply the interior gradient estimate in \cite[Corollary 1.2.22]{CK}, but this time with non-trivial right hand side (see the general form of the differential equation in \cite[(1.2.1)]{CK}). And we get
\[ \nabla v_{r_i} - \nabla v_\infty \text{ converges to } 0 \text{ uniformly in } K. \]
On the other hand, let $\tau>0 $ and we set $B_\tau(\partial\Omega_\infty) := \{X: \dist(X, \partial\Omega_\infty)<\tau\}$. Since $\nabla v_{r_i}$ is uniformly bounded, we may choose $\tau$ small enough so that
\[ \iint_{B_\tau(\partial\Omega_\infty) \cap B_5} |\nabla v_{r_i}|^2 ~dX < \epsilon. \]
Then for an arbitrary $s\in (0,5)$, we have
\begin{align*}
	\limsup_{r_i\to 0+} \iint_{B_s \cap \frac{\Omega_{X_0}}{r_i}} |\nabla v_{r_i}|^2 ~dX \leq \lim_{r_i\to 0+} \iint_{B_s \setminus B_\tau(\Omega_\infty) \cap \frac{\Omega_{X_0}}{r_i} } |\nabla v_{r_i}|^2 ~dX + \epsilon \leq \iint_{B_s \cap \Omega_\infty} |\nabla v_\infty|^2 ~dX + \epsilon.
\end{align*}
Combined with Fatou's Lemma we conclude that
\[ \iint_{B_s \cap \frac{\Omega_{X_0}}{r_i}} |\nabla v_{r_i}|^2 ~dX \to \iint_{B_s \cap \Omega_\infty} |\nabla v_\infty|^2 ~dX. \]
In particular, it implies that
\[ N(v_{r_i}, s) \to N(v_\infty, s) \text{ as } r_i \to 0. \]
Combined with \eqref{eq:freqarescaling}, we obtain
\[ N(v_\infty, s) \equiv N_{X_0} \quad \text{ for all } 0<s<5. \]
By the rigidity of the monotonicity formula in Corollary \ref{cor:rigidity}, this implies that the harmonic function $v_\infty(r,\omega) = r^{N_{X_0}} v(\omega)$ is homogeneous of degree $N_{X_0}$. Hence $\lambda:= N_{X_0} (N_{X_0} + d-2)$ is an eigenvalue for the Laplace-Beltrami operator of the half-sphere $\Omega_\infty \cap \mathbb{S}^{n-1}$. We notice that any odd extension of a Dirichlet eigenfunction on the half-sphere is an eigenfunction on the entire sphere, which vanishes on the equator. Additionally, recall that eigenfunctions on the entire sphere are restrictions of harmonic polynomials. Therefore we have proven:
\begin{lemma}
	$N_{X_0}: = \lim_{r\to 0+} \widetilde{N}(v_{X_0}, r)$ must take positive integer values.
\end{lemma}

\subsection{Compactness}\label{sec:compactness}

In this subsection, we will make sense of the following observation: if we blow up at a point in $p \in B_R(0) \cap \pD$, the configuration at the limit is the same as the blow-up limit for the corresponding reduced domain $\Omega_{p}$. In fact the observation also holds true for pseudo blow-ups, i.e. we zoom in not at a fixed point, but at a sequence of points contained in a compact set. To be more precise,

\begin{prop}\label{prop:cpt}
	Let $R, \Lambda>0$ be fixed, and $(u_i, D_i) \in \HH(R, \Lambda)$.
	Let $p_i$ be a sequence of points in $B_{\frac{3}{20}R}(0) \cap \pD_i$ and $r_i$ be a sequence that converges to zero. 
	Modulo passing to a subsequence, 
	\begin{enumerate}
		\item \label{item:CVgeometry} the sequence of domains $\frac{D_i-p_i}{r_i}$ converges graphically to $D_\infty$, where $D_\infty = \{(y,y_d): y_d > \nabla \varphi(x_\infty) \cdot y \}$ is an upper half-space and $p_\infty = (x_\infty, \varphi(x_\infty))$ is a limit point of the sequence $\{p_i\}$;
		\item inside $B_5(0)$ the sequence $T_{p_i, r_i} u_i$ converges uniformly and in $W^{1,2}$ to a function $u_\infty$, which satisfies $\Delta u_\infty = 0$ in $D_\infty \cap B_5(0)$ and $u_\infty=0$ in $B_5(0) \setminus D_\infty$.
	\end{enumerate} 
	
	Moreover, consider the functions $v_{p_i} = u_i \circ \Psi_{p_i}$ in domains $\Omega_{p_i}$ respectively.
	We have that $\frac{\Omega_{p_i}}{r_i}$ also converges graphically to $D_\infty$, and inside $B_5$ the sequence $T_{r_i} v_{p_i}$ also converges uniformly and in $W^{1,2}$ to $u_\infty$.
\end{prop} 
\begin{remark}\label{rmk:cpt}
	In fact, in the proof we will see that to get compactness, it suffices to assume $r_i$'s are bounded (in that case the limiting domain $D_\infty$ is no longer the upper half-space). We need $r_i \to 0$ to get that $T_{p_i, r_i} u_i$ and $T_{r_i} v_{p_i}$ converge to the same limiting function, and in particular to get that $T_{r_i} v_{p_i}$ converges to a harmonic function. (Note that $T_{r_i} v_{p_i}$ solves an elliptic equation with coefficient matrix $A_{r_i}(Y)$, which does not converge to the identity matrix if $r_i \not\to 0$.)
\end{remark}

\begin{proof}
By the assumption $D_i\cap B_{5R}(0)$ is above the graph of a $C^1$-Dini function $\varphi_i$, which satisfies $\varphi_i(0)=0$, $\nabla\varphi_i(0)=0$ (since $\pD_i$ is tangent to $\RR^{d-1}\times\{0\}$ at the origin) and
\begin{equation}\label{eq:Dinivarphii}
	|\nabla \varphi_i(x) - \nabla \varphi_i(y)| \leq \theta(|x-y|). 
\end{equation} 
By Arzela-Ascoli Theorem $\varphi_i$ converges in $C^1$ to a function $\varphi: \RR^{d-1} \to \RR$, which also satisifes \eqref{eq:Dinivarphii}. We write $p_i=(x_i, \varphi_i(x_i))$. Since each $x_i$ is contained in $B^{d-1}_{\frac{3}{20}R}(0)$, modulo taking a subsequence $x_i$ converges to some $x_\infty \in \RR^{d-1}$. Hence $p_i = (x_i, \varphi_i(x_i))$ converges to some $p_\infty = (x_\infty, \varphi(x_\infty))$ satisfying $|p_\infty| \leq \frac{3}{20} R$. We let $D$ denote the domain above the graph of $\varphi$. Clearly it is also a Dini domain with parameter $\theta$, $\pD \ni 0, p_\infty$ and $\pD$ is tangent to $\RR^{d-1} \times\{0\}$ at the origin.

 It follows that $\frac{D_i-p_i}{r_i}$ is locally above the graph of a function $\tilde{\varphi}_i: \RR^{d-1} \to \RR$, where
\[ \tilde{\varphi}_i(y) = \frac{1}{r_i} \left[ \varphi(x_i+r_i y) - \varphi(x_i) \right], \quad \text{ for } |y|\leq 5. \]
By the properties of $\varphi$ in \eqref{eq:Dinivarphii}, we have
\begin{align*}
	\left| \tilde{\varphi}_i(y) - \langle \nabla\varphi_i(x_\infty) , y \rangle \right| 
	& \leq \left| \tilde{\varphi}_i(y) - \langle \nabla \varphi_i(x_i) , y \rangle \right| + \left| \langle \nabla \varphi_i(x_i), y \rangle - \langle \nabla\varphi_i(x_\infty), y \rangle \right| \\
	& \leq \theta(r_i|y|)|y| + \theta(|x_i-x_\infty|) |y| \\
	& \leq 5\left( \theta(5r_i) + \theta(|x_i-x_\infty|) \right).
\end{align*}
The last term tends to zero uniformly in $|y| \leq 5$. Therefore $\tilde{\varphi}_i(y)$ converges uniformly to a linear function $\varphi_\infty(y) = \langle \nabla \varphi(x_\infty), y \rangle$.
Moreover, inside $B_5(0)$ the domain $\frac{D_i-p_i}{r_i}$ converges graphically to the domain 
\[ D_\infty := \{(y,y_n): y_n > \varphi_\infty(y) = \langle \nabla \varphi (x_\infty), y \rangle \}, \]
which is the region above a hyperplane $\{(y,y_n): y_n = \langle \nabla \varphi(x_\infty), y \rangle\}$. This finishes the proof of \eqref{item:CVgeometry}.

For simplicity of notation we write $\tilde{u}_i = T_{p_i, r_i} u_i$ and $\tilde{v}_i = T_{r_i} v_{p_i}$.
Notice that the uniform gradient bound proved in Lemma \ref{lm:vrunifgrad} is independent of the base point $X_0$, or as in the current notation, the base points $p_i$'s. So we can use similar argument as in the previous subsection to show that $\tilde{u}_i$ converges uniformly and in $W^{1,2}$ to a function $u_\infty$ in $B_5(0)$, such that $u_\infty = 0$ in $B_5(0) \setminus \Omega_\infty$ and $\Delta u_\infty = 0$ in $\Omega_\infty \cap B_5(0)$.

On the other hand, as in Section \ref{sec:reduce} setting
\[ v_{p_i}(X) = u_i(\Psi_{p_i}(X)) = u_i \left(p_i + \Psi(X) \right), \quad \text{ for } X\in \Omega_{p_i}, \]
we know it satisfies $\divg(A(X)\nabla v_i)=0$ where the matrix $A(X)$ is defined as in \eqref{def:A}. Notice that the matrix $A(X)$ is independent of the base point $p_i$. 
Recall that for each $p_i=(x_i, \varphi(x_i))$, the domain $\Omega_{p_i}$ is above the graph of an implicit function $\phi_i: \RR^{d-1} \to \RR$ defined as
\[ \phi_i(x): = \varphi_i(x_i + x)-\varphi_i(x_i) - 3|X| \tilde{\theta}(|X|), \quad \text{ where } X=(x,\phi_i(x)). \]
Therefore $\frac{\Omega_{p_i}}{r_i}$ is locally above the graph of the implicit function
\[ \tilde{\phi}_i(y) : = \frac{1}{r_i} \left[ \varphi_i(x_i+r_i y) - \varphi_i(x_i) - 3r_i|Y| \tilde{\theta}(r_i|Y|) \right], \quad \text{ where } Y= (y,\tilde{\phi}_i(y)). \]
On the compact set $|y|\leq 5$, the function $\tilde{\phi}_i$ converges uniformly to a linear function $\phi_\infty(y) = \langle \nabla \varphi(x_\infty), y \rangle $, which is the same as $\varphi_\infty(y)$. Hence inside $B_5$,
\[ \frac{\Omega_{p_i}}{r_i} \text{ also converges graphically to the domain } D_\infty. \]
Besides, for any test function $\psi$ compactly supported in $D_\infty \cap B_5$, using a change of variables several times we know
\begin{align}
	& \iint_{\frac{\Omega_{p_i}}{r_i}} \left\langle A(r_i Y) \nabla \tilde{v}_i(Y), \nabla \psi(Y) \right\rangle ~dY \nonumber \\
	& = \frac{1}{\left(\frac{1}{r_i^d} \mathlarger{\iint}_{B_{r_i} \cap \Omega_{p_i}} v_{p_i}^2 ~dX \right)^{\frac12}} \frac{1}{r_i^{d-2}} \iint_{\Omega_{p_i}} \langle A(X) \nabla v_{p_i}(X), \nabla \psi_i(X) \rangle ~dX \qquad \psi_i(X): = \psi\left( \frac{X}{r_i} \right) \nonumber \\
	& = \frac{\left( \mathlarger{\iint}_{B_{r_i}(p_i) \cap D_i } u_i^2 ~dX \right)^{\frac12}}{\left( \mathlarger{\iint}_{B_{r_i} \cap \Omega_{p_i}} v_{p_i}^2 ~dX \right)^{\frac12}} \iint_{\frac{D_i-p_i}{r_i}} \langle \nabla \tilde{u}_i(Z), \nabla \tilde{\psi}_i(Z) \rangle ~dZ, \qquad \tilde{\psi}_i(Z) = \psi\left( \frac{\Psi^{-1}(r_i Z) }{r_i} \right). \label{eq:uvblowup}
\end{align}
By the definition of $\Psi$, we know that
\[ B\left(0, \frac{s}{1+3\tilde{\theta}(s)} \right) \subset \Psi^{-1}(B_s(0)) \subset B\left(0, \frac{s}{1-3\tilde{\theta}(s)} \right). \]
Hence
\begin{align*}
	\frac{\mathlarger{\iint}_{\left( \Psi^{-1}(B_{r_i}) \Delta B_{r_i} \right) \cap \Omega_{p_i} } \mu v_{p_i}^2 ~dV_g }{\mathlarger{\iint}_{B_{r_i} \cap \Omega_{p_i} } \mu v_{p_i}^2 ~dV_g } \leq \frac{\mathlarger{\iint}_{B_{\tau_i r_i}(\partial B_{r_i}) \cap \Omega_{p_i} } \mu v_{p_i}^2 ~dV_g }{\mathlarger{\iint}_{B_{r_i} \cap \Omega_{p_i} } \mu v_{p_i}^2 ~dV_g },\quad \text{ where } \tau_i = \frac{3\tilde{\theta}(r_i)}{1-3\tilde{\theta}(r_i)}.
\end{align*}
By Lemma \ref{lm:vrunifgrad}, the right hand side tends to zero as $r_i \to 0$ (and thus $\tau_i \to 0$). Using this and the definitions of $\Omega_{p_i}$ and $v_{p_i}$, we conclude the factor
\begin{equation}\label{tmp:factor1}
	\frac{\left(\mathlarger{\iint}_{B_{r_i}(p_i) \cap D_i } u_i^2 ~dX \right)^{\frac12}}{\left(\mathlarger{\iint}_{B_{r_i} \cap \Omega_{p_i}} v_{p_i}^2 ~dX \right)^{\frac12}} \to 1 \text{ as } r_i \to 0. 
\end{equation} 
Let $Y= \frac{\Psi^{-1}(r_i Z) }{r_i}$, then by the definition \eqref{def:Psi} of the map $\Psi$ we have
\[ Z=(y,y_d - 3|Y| \tilde{\theta}(r_i|Y|)) = Y- 3|Y| \tilde{\theta}(r_i|Y|) e_d. \]
Since $\psi$ is a smooth function whose support is contained in on $B_5$, it follows that
\begin{align*}
	|\nabla \tilde{\psi}_i(Z) - \nabla \psi(Y)| \leq \|\partial_d \tilde{\psi}_i\|_\infty \left( 3\tilde{\theta}(r_i|Y|) + 3r_i|Y| \tilde{\theta}'(r_i|Y|) \right),
\end{align*}
\[ |\nabla \psi(Z) - \nabla \psi(Y) | \leq \|\nabla^2 \psi\|_\infty |Y-Z| \leq 3|Y| \|\nabla^2 \psi\|_\infty \tilde{\theta}(r_i|Y|). \]
Hence
\[ |\nabla \psi(Z) - \nabla \tilde{\psi}_i(Z)| \rightarrow 0 \text{ uniformly for } |Z| \leq 5.  \]
Therefore
\begin{align*}
	& \left| \iint_{\frac{D_i-p_i}{r_i}} \langle \nabla \tilde{u}_i(Z), \nabla \tilde{\psi}_i(Z) \rangle ~dZ - \iint_{D_\infty} \langle \nabla u_\infty(Z), \nabla \psi(Z) \rangle ~dZ \right| \\
	& \qquad \leq \left| \iint_{\frac{D_i-p_i}{r_i}} \langle \nabla \tilde{u}_i(Z), \nabla \tilde{\psi}_i(Z) \rangle ~dZ -\iint_{\frac{D_i-p_i}{r_i}} \langle \nabla \tilde{u}_i(Z), \nabla \psi(Z) \rangle ~dZ \right| \\
	& \qquad \qquad + \left|\iint_{\frac{D_i-p_i}{r_i}} \langle \nabla \tilde{u}_i(Z), \nabla \psi(Z) \rangle ~dZ  -  \iint_{D_\infty} \langle \nabla u_\infty(Z), \nabla \psi(Z) \rangle ~dZ \right|
\end{align*}
also converges to zero. Combined with \eqref{eq:uvblowup} and \eqref{tmp:factor1} we get
\[ \iint_{\frac{\Omega_{p_i}}{r_i}} \left\langle A(r_i Y) \nabla \tilde{v}_i, \nabla \psi(Y) \right\rangle ~dY  \rightarrow \iint_{D_\infty} \langle \nabla u_\infty(Z), \nabla \psi(Z) \rangle ~dZ. \]
Since $A(r_iY)$ converges to the identity map uniformly in $|Y|\leq 5$, it follows that $\tilde{v}_i$ also converges uniformly and in $W^{1,2}$ to the function $u_\infty$. In particular $\tilde{u}_i$ and $\tilde{v}_i$ converge to the same harmonic function $u_\infty$ in $D_\infty \cap B_5$.

\end{proof}

\section{Frequency function for interior points}\label{sec:interior}

\subsection{Monotonicity formula of the frequency function}
For any interior point $p$ in a $C^1$ domain $D$, we will show that the frequency function $N(p, r)$ is monotone increasing for radius $r$ in a finite interval, the length of which depends on $\dist(p, \pD)$.
 More precisely, for any $r>0$ let
\begin{equation}\label{def:DrHrinS}
	D(p, r) = \iint_{B_r(p) \cap D} |\nabla u|^2 ~dX \quad \text{ and } \quad H(p, r) = \int_{\partial B_r(p) \cap D} u^2 \sH; 
\end{equation} 
and let the frequency function centered at $p$ be defined as
\begin{equation}\label{def:NrinS}
	N(p, r) = \frac{rD(p, r)}{H(p, r)}. 
\end{equation} 
For simplicity of notation, when there is no confusion we often drop the dependence on $p$ and simply write $D(r), H(r)$ and $N(r)$.

\begin{prop}\label{prop:intmonotonicity}
	Let $D$ be a $C^1$ domain, in the sense that 
	\[ D\cap B_{5R}(0) = \{(x, x_d) \in \RR^d \times \RR: x_d > \varphi(x) \} \cap B_{5R}(0) \]
	where $\varphi$ is a $C^1$ function and $0\in \pD$.
	Let $u$ be a harmonic function in $D$ such that $u =0 $ on $\pD \cap B_{5R}(0)$. Then for any $p\in D\cap B_{2R}(0)$, the frequency function $N(p, r)$ satisfies
	\[ N'(r) = R_h(r) + R_b(r), \]
	where
	\[ R_h(r):= \frac{2r}{H(r)} \int_{\partial B_r(p) \cap D} \left|\partial_\rho u - \frac{N(r)}{r} u  \right|^2 \sH, \]
	\[ R_b(r):= \frac{1}{H(r)} \int_{B_r(p) \cap \pD} (\partial_n u)^2 \langle X-p, n_D(X) \rangle \sH. \]
	Here $\partial_\rho u$ denotes the radial derivative of $u$ with center $p$, and $\partial_n u$ denotes the normal derivative of $u$ pointing away from $D$.
	
	In particular $N(r)$ is monotone increasing with respect to $r$, as long as 
	\begin{equation}\label{cond:preintmono}
		r\theta(r) \leq \dist(p, \pD),
	\end{equation} 
	where $\theta(\cdot)$ is the modulus of continuity for the function $\nabla \varphi$. 
\end{prop}
\begin{proof}
	The proof is similar to the proof of Proposition \ref{prop:monotonicity} for the case of the Laplacian operator, so we will just sketch the proof emphasizing the differences. Using the assumptions that $\Delta u=0$ and $u$ vanishes on the boundary, we get
	\begin{equation}\label{tmp:derHSr}
		H'(r) = \frac{d-1}{r} H(r) + 2 \int_{\partial B_r(p) \cap D} u ~\partial_\rho u \sH = \frac{d-1}{r} H(r) + 2D(r),
	\end{equation}
	and 
	\[
		D'(r) = \frac{d-2}{r} D(r) + 2\int_{\partial B_r(p) \cap D} (\partial_\rho u)^2 \sH + \frac1r \int_{B_r(p) \cap \pD} (\partial_n u)^2 \langle X-p, n_D(X) \rangle \sH.
	\]
	Therefore
	\begin{align*}
		\frac{N'(r)}{N(r)} & = \frac1r + \frac{D'(r)}{D(r)} - \frac{H'(r)}{H(r)} \\
		& = \frac{2}{D(r)} \int_{\partial B_r(p) \cap D} \left|\partial_\rho u - \frac{N(r)}{r} u  \right|^2 \sH
		+ \frac{1}{rD(r)} \int_{B_r(p) \cap \pD} (\partial_n u)^2 \langle X-p, n_D(X) \rangle \sH.
	\end{align*}
	Or equivalently,
	\begin{align}
		N'(r) = \frac{2r}{H(r)} \int_{\partial B_r(p) \cap D} \left|\partial_\rho u - \frac{N(r)}{r} u  \right|^2 \sH + \frac{1}{H(r)} \int_{B_r(p) \cap \pD} (\partial_n u)^2 \langle X-p, n_D(X) \rangle \sH.\label{tmp:intNrder}
	\end{align} 
	
	We claim that if the condition \eqref{cond:preintmono} holds, then
	\[ \langle X-p, n_D(X)\rangle \geq 0, \quad \text{ for every } X\in B_r(p) \cap \pD. \]
	It then follows that $R_b(r) \geq 0$.
	In fact, we denote $p=(x_0, z_0)$, where $x_0 \in \RR^{d-1}$ and $z_0\in \RR$ satisfies $z_0 - \varphi(x_0) \geq \dist(p, \pD)$. Let $X=(x,\varphi(x))$ be an arbitrary boundary point in $B_r(p)$. then
	\begin{align*}
		\langle X-p, n_D(X)\rangle & = \left\langle (x-x_0, \varphi(x)-z_0), \frac{(\nabla \varphi(x), -1) }{\sqrt{1+|\nabla \varphi(x)|^2} } \right\rangle \\
		& = \frac{1}{{\sqrt{1+|\nabla \varphi(x)|^2} }} \left[ \langle \nabla \varphi(x), x-x_0 \rangle + z_0 - \varphi(x) \right] \\
		& = \frac{1}{{\sqrt{1+|\nabla \varphi(x)|^2} }} \left[ \left( z_0 - \varphi(x_0) \right) + \left(\varphi(x_0) - \varphi(x) - \langle \nabla \varphi(x), x_0 - x\rangle \right) \right] \\
		& \geq \frac{1}{{\sqrt{1+|\nabla \varphi(x)|^2} }} \left[ \dist(p,\pD) - \theta(r) r \right] \\
		& \geq 0.
	\end{align*}	
	Since we also have $R_h(r) \geq 0$, by \eqref{tmp:intNrder} we know $N(r)$ is monotone increasing with respect to $r$.
\end{proof}

As in the purely interior case, the limit of the frequency function $N(p,0):=\lim_{r\to 0+} N(p, r)$ measures the vanishing order of $u$ at $p$. 
By the blow up analysis similar to Subsection \ref{subsec:blowupofv}, when $p\in \mathcal{N}(u)$ we know that $N(p,0) \in \mathbb{N}$.
\footnote{If $u(p) \neq 0$, then the leading order of $u$ near $p$ is simply the non-trivial constant $u(p)$. In this case, we need to define the frequency function centered at $p$ using the harmonic function $u-u(p)$, so as to capture the leading order term of $u-u(p)$.}

Since $\tilde{\theta}(\cdot)>0$ is a nondecreasing and continuous function, the map $r\in [0,R] \mapsto r\tilde{\theta}(r) \in [0, R\tilde{\theta}(R)]$ is bijective. In particular it gives a decomposition of the interval $[0, R\tilde{\theta}(R)]$,
and for any real number $x\in [0, R\tilde{\theta}(R)]$ we can find $r\in [0, R]$ such that $x=r\tilde{\theta}(r)$.
Thus for any $p\in D\cap B_{2R}(0)$ sufficiently close to the boundary, we can always find a unique $r>0$ such that 
\begin{equation}\label{def:criticalval}
	 \dist(p, \pD) = r\tilde{\theta}(r). 
\end{equation} 
Thus we will refer to such $r$ as the \textit{critical scale} for $p$ and denote it by $r_{cs}(p)$, and say that $[0, r_{cs}(p)]$ is the \textit{monotonic interval} for $p$. To unify the notation we use the convention that $r_{cs}(p) = 0$ if $p\in \pD$.

\subsection{Doubling property and frequency function beyond the critical scale}

For interior points, similar to Lemma \ref{lm:freqbd} we also get a uniform bound (depending on $d, \Lambda$) on the frequency function.
\begin{lemma}\label{lm:freqbd_in}
	Assume that $N_0(4R) \leq \Lambda<+\infty$. Suppose $X\in D \cap B_{\frac{R}{10}}(0)$ is such that $\dist(X, \pD) \leq \frac{3}{5}R \tilde{\theta}(\frac35 R)$. Then for any $r\leq  \frac{3}{5}R$, we have
	\begin{equation}\label{eq:freqbdR_in}
		N(X, r) \leq C(\Lambda) N_{\tilde{X}} \left( \frac34 R \right) \leq C'(\Lambda),
	\end{equation} 
	where $\tilde{X} \in \pD$ satisfies $|X-\tilde{X}| = \dist(X, \pD)$.
%
%
\end{lemma}
\begin{remark}
	Notice that the above upper bound holds even outside of the monotonic interval $[0,r_{cs}(X)]$, as long as we have $r\leq \frac{3}{5}R$.
\end{remark}
\begin{proof}
	Since the frequency function $N(X, \cdot)$ is monotone increasing on the interval $[0, r_{cs}(X)]$, it suffices to prove \eqref{eq:freqbdR_in} for $r_{cs}(X) \leq r \leq \frac{3}{5} R$. By \eqref{cond:R}, we have
	\[ |X-\tilde{X}| = \dist(X, \pD) < \frac{r_{cs}(X)}{4} \leq \frac{r}{4}. \]
	By the sub-harmonicity of $u^2$ and the doubling property of $H(\tilde{X}, \cdot)$, we have
	\begin{align*}
		H(X,r) \gtrsim \frac{1}{r} \iint_{A_{\frac{r}{4}, r}(X)} u^2 ~dZ \geq \frac{1}{r} \iint_{A_{\frac{r}{2}, \frac34 r}(\tilde{X})} u^2 ~dZ \gtrsim \int_{\partial B_{\frac{r}{2}}(\tilde{X})} u^2 \sH \gtrsim_{\Lambda} H\left(\tilde{X}, \frac54 r \right).
	\end{align*}
	Therefore
	\begin{align*}
		N(X, r) = \frac{r \mathlarger{\iint}_{B_r(X)} |\nabla u|^2 ~dZ}{ H(X, r)} \lesssim \frac{r \mathlarger{\iint}_{B_{\frac54 r}(\tilde{X}) } |\nabla u|^2 ~dZ }{H(\tilde{X}, \frac54 r)} \leq N_{\tilde{X}}\left(\frac54 r \right) \lesssim N_{\tilde{X}}\left(\frac34 R \right).
	\end{align*}
%
%
%
\end{proof}

The above frequency bound yields the doubling property for the $L^2$ norm, similar to the boundary case in Corollary \ref{cor:Hrbd}.
\begin{corollary}\label{cor:doublingu_in}
	Assume that $N_0(4R) \leq \Lambda<+\infty$. Then for any $X\in D\cap B_{\frac{R}{10}}(0)$ such that $\dist(X, \pD) \leq \frac35 R \tilde{\theta}(\frac35 R)$,
	and any $\rho>0$, $a>1$ such that $a\rho \leq \frac35 R$, we have
	\begin{equation}\label{eq:doublingin}
		\frac{\mathlarger{\iint}_{B_{a\rho}(X)} u^2 ~dZ}{\mathlarger{\iint}_{B_{\rho}(X)} u^2 ~dZ} \leq a^{d+C(\Lambda)}.
	\end{equation}
\end{corollary}
\begin{proof}
Recall we have computed the derivative of $H(X,s)$ in \eqref{tmp:derHSr}, which can be reformulated into
\begin{align*}
	\left(\log \frac{H(X,s)}{s^{d-1}} \right)' = \frac{2N(X,s)}{s}.
\end{align*}
Integrating the above equality and by the upper bound of the frequency function, we get
\begin{equation}\label{eq:doublingnH}
	\frac{H(X,as)}{H(X,s)} = a^{d-1}\exp\left[ \int_s^{as} \frac{2N(X,\tau)}{\tau} ~d\tau \right] \leq a^{d-1+C(\Lambda)}. 
\end{equation} 
Therefore
\begin{align*}
	\frac{\mathlarger{\iint}_{B_{a\rho}(X)}u^2 ~dZ}{\mathlarger{\iint}_{B_{\rho}(X)} u^2 ~dZ} = \frac{a\mathlarger{\int}_0^{\rho} H(X,as) ~ds}{\mathlarger{\int}_0^{\rho} H(X,s) ~ds } \leq a^{d+C(\Lambda)}.
\end{align*}

\end{proof}

\begin{lemma}\label{lm:urunifgrad}
	Assume that $N_0(4R)\leq \Lambda +\infty$. For any $X \in D \cap B_{\frac{R}{10}}(0) \cap \mathcal{N}(u)$ such that $\dist(R,\pD)\leq \frac35 R \tilde{\theta}(\frac35 R)$, and any
	$r\leq \frac{3}{35} R$, we have
	\[ \left|\nabla T_{X, r} u(Y) \right| \leq C, \quad \text{ for all } Y \in \frac{D-X}{r} \cap B_5. \]
\end{lemma}
\begin{proof}
	Recalling the definition of $T_{X,r}u$ in Definition \ref{def:Tpr}, we have
	\[ |\nabla T_{X,r}u(Y)|^2 = \frac{\left( r|\nabla u(X+rY)| \right)^2}{\frac{1}{r^d} \mathlarger{\iint}_{B_r(X)} u^2 ~dZ }. \]
	We first consider the purely interior case when $6r < \dist(X, \pD)$.
	In this case, by the interior gradient estimate, $u(X)=0$ and the doubling property in Corollary \ref{cor:doublingu_in}, we have
	\begin{align*}
		\left|\nabla T_{X, r} u(Y) \right| \leq \frac{\sup_{B_{5r}(X)} \left( r|\nabla u| \right)^2}{\frac{1}{r^d} \mathlarger{\iint}_{B_r(X)} u^2 ~dZ } \lesssim \frac{ \mathlarger{\iint}_{B_{6r}(X)} u^2 ~dZ}{\mathlarger{\iint}_{B_r(X)} u^2 ~dZ } \leq 6^{d+C(\Lambda)}
	\end{align*}
	is uniformly bounded. Now assume $\dist(X, \pD) \leq 6r$.	
	Let $\tilde{X}\in \pD$ such that $|X-\tilde{X}| = \dist(X,\pD)$.
	We bound the numerator by the boundary gradient estimate of $u$:
	\begin{equation}\label{tmp:numbd}
		\left( r|\nabla u(X+rY)| \right)^2 \leq \sup_{B_{6r}(\tilde{X}) \cap D} \left( r|\nabla u|\right)^2 \lesssim \frac{1}{r^d} \iint_{B_{7r}(\tilde{X})} u^2 ~dZ.
	\end{equation}
	To bound the denominator, we use the doubling property \eqref{eq:doublingnH} and $|X-\tilde{X}| \leq 6r$, and get
	\begin{align}
		\frac{1}{r^d} \iint_{B_r(X)} u^2 ~dZ \gtrsim_{\Lambda} \frac{1}{r^d} \iint_{B_{7r}(X)} u^2 ~dZ \gtrsim \frac{1}{r^d} \iint_{B_r(\tilde{X})} u^2 ~dZ. \label{tmp:denbd}
	\end{align}
	Therefore combining \eqref{tmp:numbd} and \eqref{tmp:denbd}, it follows from Corollary \ref{cor:Hrbd} and the bound on the frequency function that
	\[ |\nabla T_{X,r} u|^2 \lesssim \frac{\mathlarger{\iint}_{B_{7r}(\tilde{X})} u^2 ~dZ }{\mathlarger{\iint}_{B_{r}(\tilde{X})} u^2 ~dZ } \leq 7^{d+2N_{\tilde{X}}(7r)} \leq C(\Lambda). \]
	
%
%
\end{proof}

For an interior point $X\in D$, since the frequency function $N(X, \cdot)$ is only monotone increasing in the interval $[0, r_{cs}(X)]$, for large radius we will replace $N(X, r)$ by the corresponding frequency function centered at a boundary point $\tilde{X} \in \pD$, which satisfies $|X-\tilde{X}| = \dist(X, \pD)$. The following lemma justifies this choice.

\begin{lemma}\label{lm:spvarin_far}
	Let $R, \Lambda>0$, $\rho \in (0, 1/6]$ and $\delta_{in}>0$ be fixed. There exists $r_{in} = r_{in}(\delta_{in}, \rho)>0$ such that the following holds for any $(u,D) \in \HH(R,\Lambda)$. Suppose $p\in D\cap B_{\frac{R}{10}}(0) \cap \mathcal{N}(u)$.
	Let $q\in \pD $ satisfy $|q-p| \leq 2 \dist(p,\pD)$. Then for any radius $\rho r_{cs}(p) \leq r \leq r_{in}$\footnote{Notice that if this condition on $r$ is not vacuous, i.e. if $\rho r_{cs}(p) \leq r_{in}$, then $\dist(p, \pD)\leq \frac{r_{in}}{\rho} \tilde{\theta}(\frac{r_{in}}{\rho})$, which is in particular bounded above by $\frac35 R \tilde{\theta} (\frac35 R)$.}
	we have
	\[ \left| N(p, r) - N_q( r) \right| \leq \delta_{in}. \]
%
\end{lemma}

\begin{proof}
	We argue by contradiction. To that end we assume there exist sequences $(u_i, D_i) \in \HH(R, \Lambda)$, $p_i \in D_i\cap B_{\frac{R}{10}}(0) \cap \mathcal{N}(u_i)$ with $\rho r_{cs}(p_i) \leq r_i \to 0$ and 
	\begin{equation}\label{assp:distr}
		d_{p_i} := \dist(p_i, \pD_i) = r_{cs}(p_i) \tilde{\theta}(r_{cs}(p_i)) \leq \frac{r_i}{\rho} \tilde{\theta}\left(\frac{r_i}{\rho} \right), 
	\end{equation}  
	and $q_i \in \pD_i$ with $|p_i - q_i| \leq 2d_{p_i}$, such that
	\begin{equation}\label{cl:freqbin}
		\left| N(p_i, r_i) - N_{q_i}(r_i) \right| > \delta_{in}>0. 
	\end{equation}
	
	By assumption 
	\[ D_i \cap B_{5R}(0) \subset \{(x, x_d) \in \RR^{d-1} \times \RR: x_d > \varphi_i(x)\} \]
	for some $C^1$ function $\varphi_i$ with Dini parameter $\theta$.
	Without loss of generality we assume $\varphi_i(0) = 0$ and $\nabla \varphi_i(0) = 0$. Hence 
	\[ |\nabla \varphi_i(x)| = |\nabla \varphi_i(x) - \nabla \varphi_i(0) | \leq \theta(|x|) \leq \theta(5R) \]
	is uniformly bounded, and by Arzela-Ascoli $\varphi_i$ converges uniformly to a function $\varphi$ which satisfies the same properties as $\varphi_i$.
	
	We denote $p_i = (x_i, z_i) \in \RR^{d-1} \times \RR$. By assumption we have $|p_i| < \frac{R}{10}$. Simple geometry shows that
	\begin{equation}\label{tmp:diffvertdist}
		d_{p_i} \leq z_i - \varphi_i(x_i) \leq d_{p_i} \sqrt{ 1+\left|\theta\left(R \right) \right|^2} + \theta(R) d_{p_i} \leq 2d_{p_i}.
	\end{equation} 
	Simple computations show that $\frac{D_i - p_i}{r_i} \cap B_5(0)$ corresponds to the region above the graph of the function $\psi_i: \RR^{d-1} \to \RR$, defined as
	\[ \psi_i(y) := \frac{1}{r_i} \left[ \varphi_i(x_i+r_i y) - \varphi_i(x_i) - (z_i - \varphi_i(x_i)) \right]. \]
	By \eqref{tmp:diffvertdist} and the assumption \eqref{assp:distr}, we have that $\psi_i(y)$, modulo passing to a subsequence, converges uniformly to a linear function $\varphi_\infty(y):= \langle \nabla \varphi(x_\infty), y \rangle$, where $x_\infty$ is a cluster point for $\{x_i\} \subset B_R(0)$. In other words, inside $B_5(0)$ the sequence of domains
	\[ \frac{D_i-p_i}{r_i} \text{ converges to the upper half space } D_\infty:= \{(y,y_d)\in \RR^{d-1} \times \RR: y_d > \varphi_\infty(y) \}.  \]
	
	By Lemma \ref{lm:urunifgrad}, $T_{p_i,r_i}u_i(0) = 0$
	and compactness, we get that the sequence 
	\[ T_{p_i, r_i} u_i(Y): = \frac{u_i(p_i + r_i Y) }{ \left( \frac{1}{r_i^d} \mathlarger{\iint}_{B_{r_i}(p_i)} u_i^2 ~dZ \right)^{\frac12} } \] 
	converges uniformly and in $W^{1,2}$ to a harmonic function $u_\infty$ in $D_\infty \cap B_5(0)$. 
	On the other hand, by Proposition \ref{prop:cpt} the sequence 
	\[ T_{q_i, r_i} u_i(Y) := \frac{u_i(q_i+r_i Y)}{\left(\frac{1}{r_i^d} \mathlarger{\iint}_{B_{r_i}(q_i)} u_i^2 ~dZ \right)^{\frac12} } 
	\]
	(and the sequence $T_{r_i} v_{q_i}$) also converges uniformly and in $W^{1,2}$ to a harmonic function $\tilde{u}_\infty$ in the same upper half space $D_\infty$.
	Moreover we claim there exists some $a_i \approx 1$ such that
	\begin{equation}\label{claim:changebin}
		|a_i T_{p_i,r_i}u_i(Y) - T_{q_i, r_i} u_i(Y)| \to 0 \text{ uniformly for } Y \in B_5(0).
	\end{equation}
	Assuming the claim is true, then $\tilde{u}_\infty = u_\infty$. (A priori $\tilde{u}_\infty$ is a constant multiple of $u_\infty$, and the constant must be $1$ since they both have unit $L^2$ norm on $B_1(0)$.)
	Hence
	\begin{align*}
		\left| N(u_i, p_i, r_i) - \widetilde{N}(v_{q_i}, r_i) \right| & = \left| N(T_{p_i, r_i} u_i, 0, 1)  - N(T_{r_i} v_{q_i}, 0, 1) \exp \left(C\int_0^{r_i} \frac{\theta(4s)}{s} ~ds \right) \right| \\
		& \to \left| N(u_\infty, 0, 1) - N(\tilde{u}_\infty, 0,1) \right| = 0,
	\end{align*}
	which contradicts the assumption \eqref{cl:freqbin}.
%
%
	
	\textit{Proof of the claim \eqref{claim:changebin}}.
	By the assumption \eqref{assp:distr} and $|q_i - p_i| \leq 2d_{p_i} \ll r_i$, we know
	\[ \frac{1}{r_i^d} \iint_{B_{r_i}(p_i)} u_i^2 ~dZ \approx \frac{1}{r_i^d} \iint_{B_{r_i}(q_i)} u_i^2 ~dZ. \]
	We denote
	\[ a_i =  \left(\frac{ \frac{1}{r_i^d} \mathlarger{\iint}_{B_{r_i}(p_i)} u_i^2 ~dZ}{\frac{1}{r_i^d} \mathlarger{\iint}_{B_{r_i}(q_i)} u_i^2 ~dZ } \right)^{\frac12} \approx 1. \]
	By the boundary gradient estimate in \cite[Theorem 1.4.3]{CK}, the doubling property \eqref{eq:Hrcmpu} at $q_i$, and the assumption \eqref{assp:distr}, we get
	\begin{align}
		\frac{\left| u_i(p_i+r_i Y) - u_i(q_i + r_i Y) \right|}{\left(\frac{1}{r_i^d} \mathlarger{\iint}_{B_{r_i}(q_i)} u_i^2 ~dZ \right)^{\frac12} } & \leq \frac{\left| p_i - q_i \right| }{\left(\frac{1}{r_i^d} \mathlarger{\iint}_{B_{r_i}(q_i)} u_i^2 ~dZ \right)^{\frac12} } \sup_{B_{\frac{11}{2} r_i}(q_i) \cap D_i } |\nabla u_i| \nonumber \\
		&\leq \frac{\left| p_i - q_i \right| }{\left(\frac{1}{r_i^d} \mathlarger{\iint}_{B_{r_i}(q_i)} u_i^2 ~dZ \right)^{\frac12} } \cdot \frac{1}{r_i} \left( \frac{1}{r_i^d} \mathlarger{\iint}_{B_{6r_i}(q_i)} u_i^2 ~dZ \right)^{\frac12} \nonumber \\
		& \lesssim \frac{|p_i-q_i|}{r_i} \left(\frac{ \frac{1}{r_i^d} \mathlarger{\iint}_{B_{6r_i}(q_i)} u_i^2 ~dZ  }{\frac{1}{r_i^d} \mathlarger{\iint}_{B_{r_i}(q_i)} u_i^2 ~dZ } \right)^{\frac12} \nonumber \\
		& \lesssim \frac{d_{p_i}}{r_i} \cdot 6^{N(q_i, 6r_i)} \nonumber \\
		& \lesssim \tilde{\theta}\left(\frac{r_i}{\rho} \right) \cdot 6^{C(\Lambda)} \to 0.\label{tmp:bint2}
	\end{align}
	This finishes the proof of \eqref{claim:changebin}.
\end{proof}

The above lemma means that for any interior point $p\in D$ close enough to the boundary, whenever $r$ is out of the monotonic interval for $p$, we can always replace the frequency function $N(p,r)$ by that of a nearby boundary point $q$, with a small price to pay. Therefore we define
\[
	N_p(r) := \left\{\begin{array}{ll}
		N(p, r), & r\leq r_{cs}(p) \\
		N_q(r) = \widetilde{N}(v_{q}, r), & r> r_{cs}(p)
	\end{array} \right.
\]
where $q\in \pD$ is such that $|q-p| = \dist(p, \pD)$. Hence $N_p(r)$ is monotone increasing in the disjoint intervals $[0, r_{cs}(p)]$ and $(r_{cs}(p), +\infty)$, with a possible jump of $\delta_{in}$ since
\[ 
 N_q(r_{cs}(p)) - \delta_{in} \leq N(p, r_{cs}(p)) \leq N_q(r_{cs}(p)) + \delta_{in} \leq N_q(r_{cs}(p)+) + \delta_{in}. \]

To summarize Sections \ref{sec:reduce} and \ref{sec:interior}, for any $(u,D)\in \mathfrak{H}(R,\Lambda)$ we have defined the frequency function for any point $X\in \overline{D}$ as follows:
\begin{equation}\label{def:Npr}
	N_X(r) =N_X^u(r) = \left\{\begin{array}{ll}
	 	\widetilde{N}(v_X, r), & \text{ if } X \in \pD \\
	 	N(X, r), & \text{ if } X\in D \text{ and } r\leq r_{cs}(X) \\
	 	N_{\tilde{X}}(r) = \widetilde{N}(v_{\tilde{X}}, r), & \text{ if } X\in D \text{ and } r>r_{cs}(X)
	\end{array} \right.
\end{equation}
where $\tilde{X}\in \pD$ is such that $|X-\tilde{X}| = \dist(X, \pD)$.


\section{Quantitative symmetry}\label{sec:qsym}
In this section, we use compactness to prove 
quantitative cone-splitting and dimension reduction.
Heuristically, the intuition for dimension reduction is the following simple observation.
Suppose $C$ is a cone in $\RR^d$ which is translation-invariant along a $k$-dimensional linear subspace $V$. In other words we may identify $V$ with $\{0\} \times \RR^k$ and identify $C$ with a cylindrical cone $C_0 \times \RR^k$, where $C_0 $ is a cone in $\RR^{d-k}$. Then any point $X$ away from the spine $\{0\} \times \RR^k$ is symmetric in $(k+1)$ directions (i.e. along the direction of the spine $V$ and the radial direction in $C_0$).
To prove quantitative cone-splitting and dimension reduction in our setting requires two non-trivial adaptations of the above observation. Firstly at each boundary point $X_0$ the frequency function is only monotone for $v_{X_0} = u\circ \Psi_{X_0}$, and we need to combine the information of $v_{X^j}$'s with different base points $X_j$'s to produce an invariant subspace. Secondly, along the radial direction the limit function is not constant, but grows polynomially. So we need a different approach to distinguish points on the spine and away from the spine.

\begin{defn}
	Let $Y_0, \cdots, Y_k$ be arbitrary points in $B_r(p) \subset \RR^k$. If for all $i=1, \cdots, k$, they satisfy
	\[ Y_i \notin B_{\tau r}\left( Y_0 + \spn\{Y_1-Y_0, \cdots, Y_{i-1}-Y_0\} \right), \]
	we say that these points $\tau r$-effectively span the $k$-dimensional affine subspace $V:= Y_0 + \spn\{Y_1-Y_0, \cdots, Y_k - Y_0\}$ in $B_r(p)$.
	
	Given a set $F\subset B_r(p)$, we say that $F$ $\tau r$-effectively spans a $k$-dimensional affine subspace if there exists a $(k+1)$-tuple $\{Y_0, \cdots, Y_k\} \subset F$ which $\tau r$-effectively spans a $k$-dimensional affine subspace.
\end{defn} 

\begin{remark}
 Unlike the notion of \textit{linear independence}, the above effective notion is preserved under limits.
\end{remark}


\begin{prop}\label{prop:tn}
	Let $R, \Lambda>0$ and $\delta_0, \rho, \tau \in (0,1)$ be fixed. There exist $\delta>0$, $\beta>0$ and $r_{tn} \leq r_{in}(\delta)$ (where $r_{in}(\delta)$ denotes the radius in Lemma \ref{lm:spvarin_far} if we take $\delta_{in} = \delta$ and $\rho= 1/6$)
	such that the following holds for any $(u, D) \in \mathfrak{H}(R, \Lambda)$ and $r\leq r_{tn}$.
	Suppose
	\begin{equation}\label{asmp:supNXr}
		\sup_{X\in B_{2r}(0) \cap \Nt(u)} N_X(r) \leq \tilde{\Lambda} \leq C(\Lambda).  
	\end{equation} 
	If the set 
	\[ F= \{X\in B_{2r}(0) \cap \Nt(u): N_X(\rho r) \geq \tilde{\Lambda} - \delta \} \]
	$2\tau r$-effectively spans a $(d-2)$-dimensional affine subspace $V$, then
	\begin{equation}\label{eq:delta0pinch}
		\text{for every }X\in B_{2\beta r}(V) \cap B_{2r}(0) \cap \Nt(u), \quad 
		N_X(\rho r) \geq \tilde{\Lambda} - \delta_0;
	\end{equation} 
	and
	\begin{equation}\label{eq:tn}
		\Ct_{r}(u) \cap B_{2r}(0) \cap \Nt(u) \subset B_{2\beta r}(V). 
	\end{equation} 
\end{prop} 

\begin{proof}
	\textit{Proof of \eqref{eq:delta0pinch}.} We remark that $B_{2\beta r}(V) \cap B_{2r}(0) \cap \Nt(u)$ can not be empty, since it contains at least $(d-1)$ points generating the subspace $V$. We also remark that 
	\begin{equation}\label{tmp:chgF}
		F\subset \left\{X\in B_{2r}(0) \cap \Nt(u): \left|N_X(r) - N_X(\rho r) \right| \leq \delta \right\}. 
	\end{equation}  
	Indeed it is clearly true if $X\in \pD$, since 
	\[ 0 \leq N_X(r) - N_X(\rho r) \leq \tilde{\Lambda} - \left( \tilde{\Lambda} - \delta \right) =\delta. \]
	When $X\in D$, \eqref{tmp:chgF} is true if $N_X(r)\geq N_X(\rho r)$. If not (which can only happen when $\rho r \leq r_{cs}(X) < r$), by the definition \eqref{def:Npr}, $r_{cs}(X) < r\leq r_{in}(\delta)$ and Lemma \ref{lm:spvarin_far} we have
	\[ N_X(r) - N_X(\rho r) = N_{\tilde{X}}(r) - N(X, \rho r) \geq N_{\tilde{X}}(r) - \left(N_{\tilde{X}}(\rho r) + \delta \right) \geq -\delta, \]
	where $\tilde{X} \in \pD$ is such that $|X-\tilde{X}| = \dist(X, \pD)$.
	By the assumption \eqref{asmp:supNXr} and the definition of $F$,
	\[ N_X(r) - N_X(\rho r) \leq \tilde{\Lambda} - (\tilde{\Lambda}-\delta) = \delta. \] Hence \eqref{tmp:chgF} is true.
	
	We argue by contradiction. That is, we assume there are $(u_i, D_i) \in \mathfrak{H}(R, \Lambda)$ and $\delta_i, r_i, \beta_i \to 0$ satisfying $r_i \leq r_{in}(\delta_i)$, such that
	\[ \sup_{X\in B_{2r_i}(0) \cap \Nt(u_i)} N^{u_i}_{X}(r_i) \leq \tilde{\Lambda}_i \leq C(\Lambda), \] 
	the set 
	\[ F_i = \{X\in B_{2r_i}(0) \cap \Nt(u_i): N^{u_i}_{X_i}(\rho r_i) \geq \tilde{\Lambda}_i - \delta_i  \} \]
	$2\tau r_i$-effectively spans a $(d-2)$-dimensional affine subspace, denoted by 
	\[ V_i = X_i^0 + \spn\{ X_i^1 - X_i^0, \cdots, X_i^{d-2} - X_i^0 \}, \quad \text{ with } X_i^j \in F_i \text{ for all } j\in\{0, \cdots, d-2\}; \]
	and yet there exists $X_i \in B_{2\beta_i r_i}(V_i) \cap B_{2r_i}(0) \cap \Nt(u_i)$ such that 
	$N^{u_i}_{X_i}(\rho r_i) < \tilde{\Lambda}_i - \delta_0$.
	
	
	\uline{Step $1$}. Since $0< \tilde{\Lambda}_i \leq C(\Lambda)$, modulo passing to a subsequence $\tilde{\Lambda}_i$ converges to some $\tilde{\Lambda} \in [0, C(\Lambda)]$. By Proposition \ref{prop:cpt} for each $j\in \{0, \cdots, d-2\}$ fixed, the sequence $T_{X_i^j, r_i} u_i$ converges uniformly and in $W^{1,2}$ to some harmonic function $u_\infty^j$ in $D_\infty^j \cap B_5(0)$, which vanishes on the boundary. With $j$ fixed, we consider two cases. \uline{Case (i): there are infinitely many boundary points in the sequence $\{X_i^j\}_i$}. Then we throw away all interior points in the sequence. By Proposition \ref{prop:cpt} $T_{r_i} v_{X_i^j}$ also converges to the same function $u_\infty^j$. Since $X^j_i \in B_{2r_i}(0)$, we have
	\[ X_i^j \to 0 \text{ as } i \to +\infty, \quad \text{ for each } j \text{ fixed. } \] 
	Hence the domain $D_\infty^j$ is in fact the upper half-space $\RR^d_+$.
	Let $\rho \in (0,1)$ be arbitrary. Then for any $i$ sufficiently large so that $\rho_i \leq \rho$, by the scale-invariance of the frequency function we have
	\begin{align*}
		& N(T_{r_i} v_{X_i^j}, 1) - N(T_{r_i} v_{X_i^j}, \rho) \\
		& = N(v_{X_i^j}, r_i) - N(v_{X_i^j}, \rho r_i) \\
		& = \widetilde{N}(v_{X_i^j}, r_i) \exp\left(-C\int_0^{r_i} \frac{\theta(4\tau)}{\tau} ~d\tau \right) - \widetilde{N}(v_{X_i^j},\rho r_i) \exp\left(-C\int_0^{\rho r_i} \frac{\theta(4\tau)}{\tau} ~d\tau \right) \\
		& =  \widetilde{N}(v_{X_i^j}, r_i) \left[ \exp\left(-C\int_0^{r_i} \frac{\theta(4\tau)}{\tau} ~d\tau \right) - \exp\left(-C\int_0^{\rho r_i} \frac{\theta(4\tau)}{\tau} ~d\tau \right) \right] \\
		& \qquad + \left[ \widetilde{N}(v_{X_i^j}, r_i) - \widetilde{N}(v_{X_i^j},\rho r_i) \right] \exp\left(-C\int_0^{\rho r_i} \frac{\theta(4\tau)}{\tau} ~d\tau \right) \\
		& \to 0, \qquad \text{ as } r_i \to 0. 
	\end{align*} 
	In the last line we used the assumption $X_i^j \in F_i$, the fact that $N_{X_i^j}(r_i)$'s are uniformly bounded by $C(\Lambda)$ and $\int_0^{r_i} \frac{\theta(4\tau)}{\tau}~d\tau \to 0$ (see Remark \ref{rmk:deptheta}).	
	Hence
	\[ N(u_\infty^j, 1) - N(u_\infty^j, \rho) = \lim_{r_i \to 0} N(T_{r_i} v_{X_i^j}, 1) - N(T_{r_i} v_{X_i^j}, \rho) = 0. \]
	Since $u_\infty^j$ is a harmonic function in the upper half-space, by the rigidity of the monotonicity formula in Corollary \ref{cor:rigidity} $u_\infty^j$ is a homogeneous harmonic polynomial in $B_1(0)$. Thus by the unique continuation theorem $u_\infty^j$ is a homogeneous harmonic polynomial in the upper half-space.
	
	Otherwise, we fall into \uline{Case (ii): there are infinitely many interior points in the sequence $\{X_i^j\}_i$}. Then we throw away the finitely many boundary points (if they exist) in the sequence. Since $X_i^j \in B_{2r_i}(0)$, we know that
	\[ \dist(X_i^j, \pD_i) \leq |X_i^j-0| < 2r_i. \]
	In particular, modulo passing to a subsequence, the sequence of real numbers $\dist(X_i^j, \pD_i)/r_i$ converges to some $d_j \in [0, 2]$. A simple computation then shows that $\frac{D_i- X_i^j}{r_i}$ converges graphically to the domain
	\[ D_\infty^j = \{(y, y_d) \in \RR^{d-1} \times \RR: y_d > -d_j \}, \]
	which is the upper half-space (when $d_j=0$), or the upper half-space shifted down by $d_j$ in the direction of $-e_d$. 
	
	We claim that 
	\begin{equation}\label{claim:homfreqin}
		\left| N(T_{X_i^j, r_i} u_i, 0, 1) - N(T_{X_i^j, r_i} u_i, 0, \rho) \right| = \left| N( u_i, X_i^j, r_i) - N(u_i, X_i^j, \rho r_i) \right| \to 0.
	\end{equation}
	This does not follow obviously from the assumption $X_i^j \in F_i$ and the observation \eqref{tmp:chgF}: Recalling the definition \eqref{def:Npr}, the frequency function $N^{u_i}_{X_i^j}(\cdot)$ that appears in $F_i$ is not always the same as $N(u_i, X_i^j, \cdot)$ straightforwardly; it is defined as follows
	\begin{equation}\label{tmp:defNpr}
		N^{u_i}_{X_i^j}(r) = \left\{\begin{array}{ll}
		N(u_i, X_i^j, r), & r \leq r_{cs}(X_i^j) \\
		N^{u_i}_{\tilde{X}_i^j}(r) = \widetilde{N}\left(v_{\tilde{X}_i^j} = \Psi_{\tilde{X}_i^j} \circ u_i, r \right) & r > r_{cs}(X_i^j)
	\end{array} \right. 
	\end{equation} 
	where $r_{cs}(X_i^j)$ denotes the critical scale for $u_i$ at the interior point $X_i^j$, and $\tilde{X}_i^j \in \pD$ is such that $\left| X_i^j - \tilde{X}_i^j \right| = \dist(X_i^j, \pD_i)$. 

	In fact, if $r_i \leq r_{cs}(X_i^j)$, then by the definition \eqref{tmp:defNpr}, $X_i^j \in F_i$ and the observation \eqref{tmp:chgF}, 
	we have
	\begin{align*}
		\left| N( u_i, X_i^j, r_i) - N(u_i, X_i^j, \rho r_i) \right| \leq \left|N^{u_i}_{X_i^j}(r_i) - N^{u_i}_{X_i^j}(\rho r_i) \right| 
		\leq \delta_i;
	\end{align*}
	if $\rho r_i \leq r_{cs}(X_i^j) < r_i$, then by the definition \eqref{tmp:defNpr}, $r_{cs}(X_i^j) < r_i \leq r_{in}(\delta_i)$, Lemma \ref{lm:spvarin_far}, $X_i^j \in F_i$ and the observation \eqref{tmp:chgF}, 
	we have
	\begin{align*}
		\left| N( u_i, X_i^j, r_i) - N(u_i, X_i^j, \rho r_i) \right| 
		& = \left| N( u_i, X_i^j, r_i) - N^{u_i}_{X_i^j}( \rho r_i) \right| \\
		&\leq \left|N(u_i, X_i^j, r_i) - N^{u_i}_{\tilde{X}_i^j}(r_i) \right| + \left| N^{u_i}_{\tilde{X}_i^j}(r_i) - N^{u_i}_{X_i^j} (\rho r_i) \right| \\
		& \leq 2\delta_i;
	\end{align*}
	if $\rho r_i > r_{cs}(X_i^j)$, then by $r_{cs}(X_i^j) < \rho r_i < r_i \leq r_{in}(\delta_i)$, Lemma \ref{lm:spvarin_far}, the definition \eqref{tmp:defNpr}, $X_i^j \in F_i$ and the observation \eqref{tmp:chgF} we have
	\begin{align*}
		& \left| N( u_i, X_i^j, r_i) - N(u_i, X_i^j, \rho r_i) \right| \\
		 & \leq \left|N( u_i, X_i^j, r_i) - N^{u_i}_{\tilde{X}_i^j}(r_i) \right| + \left|N^{u_i}_{\tilde{X}_i^j}(r_i) - N^{u_i}_{\tilde{X}_i^j}(\rho r_i) \right| + \left| N^{u_i}_{\tilde{X}_i^j}(\rho r_i) - N(u_i, X_i^j, \rho r_i) \right| \\
		 & \leq 2\delta_i + \left|N^{u_i}_{X_i^j}(r_i) - N^{u_i}_{X_i^j}(\rho r_i) \right| \\
		 & \leq 3\delta_i.
	\end{align*}
	So for each $i$ we always have
	\begin{equation}\label{tmp:frequiriin}
		\left| N( u_i, X_i^j, r_i) - N(u_i, X_i^j, \rho r_i) \right| \leq 3\delta_i \to 0, 
	\end{equation} 
	which finishes the proof of the claim \eqref{claim:homfreqin}.
	Therefore
	\begin{equation}\label{tmp:rigidityin}
		\left|N(u_\infty^j, 0, 1 ) - N(u_\infty^j, 0, \rho) \right| = \lim_{i\to \infty} \left| N(T_{X_i^j, r_i} u_i, 0, 1) - N(T_{X_i^j, r_i} u_i, 0, \rho) \right| = 0. 
	\end{equation} 
	If $d_j=0$ and thus $0\in \pD_\infty^j$, we appeal to the rigidity of the monotonicity formula for boundary points (Corollary \ref{cor:rigidity}) to conclude that $u_\infty^j$ is homogeneous with respect to the origin. Alternatively if $d_j>0$, then $0\in D_\infty^j$ and moreover $\dist(0, \pD_\infty^j)=d_j$.
	Since $u_\infty^j$ vanishes on the boundary of $D_\infty^j$, by
	 the rigidity of the monotonicity formula for interior points (see Proposition \ref{prop:intmonotonicity}) we conclude that $u_\infty^j$ is homogeneous with respect to $0$. 
	 Since $u_\infty^j(0) = 0$, by homogeneity and the unique continuation property $u_\infty^j \equiv 0$, which is impossible since $\iint_{B_1(0)} |u_\infty^j|^2 ~dZ = 1$. Therefore we can only have $d_j = 0$ and $D_\infty^j = \RR^d_+$.
	
	Moreover, in both cases, since $\tilde{\Lambda}_i \to \tilde{\Lambda}$ and $\delta_i \to 0$ we can show that
	\[ N(u_\infty^j, 0, 1) = \lim_{i\to \infty} N(u_i, X_i, r_i) \leq \tilde{\Lambda}, \]
	and
	\[ N(u_\infty^j, 0, \rho) = \lim_{i \to \infty} N(u_i, X_i, \rho r_i) \geq \tilde{\Lambda}. \]
	Therefore the degree of homogeneity of $u_\infty^j$ with respect to the origin satisfies
	\begin{equation}\label{eq:determinefreqlimit}
		N(u_\infty^j, 0, 0) = N(u_\infty^j, 0, 1) = N(u_\infty^j, 0, \rho) = \tilde{\Lambda}. 
	\end{equation}

	\uline{Step $2$}.
	Let 
	\[ Y_i^j = \frac{X_i^j-X_i^0}{r_i} \in \frac{\overline{D}_i -X_i^0}{r_i} .\] 
	By the assumption of effective spanning, we have
	\[ 2 \tau \leq |Y_i^j| \leq 4 \quad \text{ and } \quad \left| Y_i^j - Y_i^{j'} \right| = \frac{|X_i^j - X_i^{j'}| }{r_i} \geq 2\tau \text{ if } j\neq j'. \]
	Modulo passing to a subsequence we have that $\{Y_i^0, Y_i^1, \cdots, Y_i^{d-2}\}$ converges to $\{0, Y^1, \cdots, Y^{d-2}\}$, which $2 \tau$-effectively spans a $(d-2)$-dimensional linear subspace $V_\infty$. 
	Moreover, since
	\[ \dist\left(Y_i^j, \frac{\pD_i- X_i^0}{r_i} \right) = \frac{\dist(X_i^j, \pD_i)}{r_i} \to d_j, \]
	it follows that
	\begin{equation}\label{eq:distlimitdomain}
		\dist(Y^j, \pD_\infty^0) = \lim_{i\to \infty}\dist\left(Y_i^j, \frac{\pD_i- X_i^0}{r_i} \right) = d_j=0,
	\end{equation}
	i.e. $Y^j \in \pD_\infty^0$. 
	
	We claim that for any $j\in \{1, \cdots, d-2\}$, there is a constants $a_j \approx 1$ (depending on the values of $d, R$ and $\Lambda$) 
	such that
	\begin{equation}\label{eq:homdiffvt}
		u_\infty^0(Y^j + Y) = a_j ~u_\infty^j(Y), \quad \text{ for every } Y \in D_\infty^j. 
	\end{equation} 
	For simplicity of notation we just write down the proof of the claim when $j=1$. Recall that
\[ T_{X_i^1, r_i} u_i(Y) = \frac{u_i(X_i^1+ r_i Y)}{ \left( \frac{1}{r_i^d} \mathlarger{\iint}_{B_{r_i}(X_i^1)} u_i^2 ~dX \right)^{\frac12}} , \]
	and
\[ T_{X_i^0, r_i} u_i(Y_i^1 + Y) = \frac{u_i(X_i^0 + r_i Y_i^1 + r_i Y)}{ \left( \frac{1}{r_i^d} \mathlarger{\iint}_{B_{r_i}(X_i^0) } u_i^2 ~dX \right)^{\frac12}} = \frac{u_i\left(X_i^1 + r_i Y\right)}{ \left( \frac{1}{r_i^d} \mathlarger{\iint}_{B_{r_i}(X_i^0) } u_i^2 ~dX \right)^{\frac12}}. \]
	Since $|X_i^1 - X_i^0| < 4r_i$, using the assumption $X_i^0, X_i^1 \in \Nt(u_i)$
	 and the doubling property in \eqref{eq:Hrcmpu} (for boundary points) and \eqref{eq:doublingin} (for interior points), 
	we get
	\begin{equation}\label{tmp:chgnHtin}
		\iint_{B_{r_i}(X_i^1)} u_i^2 ~dX \approx \iint_{B_{r_i}(X_i^0) } u_i^2 ~dX, 
	\end{equation}
	where the constant depends on $d, \Lambda$. Hence modulo passing to a subsequence their ratio converges to a positive number.
	This implies the claim \eqref{eq:homdiffvt}. 
	Recall that each $u_\infty^j$ is a homogeneous polynomial with respect to the origin. The claim then implies that $u_\infty^0$ is a homogeneous polynomial with respect to $Y^1, \cdots, Y^{d-2}$. Therefore either $u_\infty^0$ is linear, or $u_\infty^0$ is invariant along the $(d-2)$-dimensional linear subspace $V_\infty = 0 + \spn\{ Y^1, \cdots, Y^{d-2}\} \subset \pD_\infty^0$. 
	
	\uline{Step $3$}.
	On the other hand, let 
	\[ Y_i = \frac{X_i- X_i^0}{r_i} \in \frac{\overline{D}_i - X_i^0}{r_i}. \] 
	Then modulo passing to a subsequence $Y_i$ converges to some point $Y_\infty \in \overline{\RR^d_+} \cap \overline{B_4(0)}$. By compactness we have $T_{X_i, r_i} u_i$ converges to some harmonic function (not necessarily homogeneous) $u_\infty$ in $\RR^d_+ - Y_\infty \cap B_5(0)$.
	Moreover, let
	\[ \widetilde{V}_i : = 0+\spn\{Y_i^1, \cdots, Y_i^{d-2}\} \]
	be a $(d-2)$-dimensional linear subspace. Since $X_i \in B_{2\beta_i r_i}(V_i)$, we have
	\[ \dist(Y_i, \widetilde{V}_i) = \frac{\dist(X_i, V_i)}{r_i} < 2\beta_i \to 0. \]
	Since $\widetilde{V}_i \to V_\infty$, this implies that $Y_\infty \in V_\infty \subset \partial \RR^d_+$. Hence the domain $\RR^d_+ - Y_\infty$ is exactly the upper half-space $\RR^d_+$. This in particular implies that $\dist(X_i, \pD_i)/r_i \to 0$, and similar to the proof of \eqref{eq:homdiffvt}, we can show there is a constant $a\approx 1$ such that
	\begin{equation}\label{tmp:tslinvplane}
		u_\infty^0(Y_\infty + Y ) = a ~u_\infty(Y). 
	\end{equation} 
	
	Since $u_\infty^0$ is invariant along the plane $V_\infty$ and $Y_\infty \in V_\infty$, it follows from \eqref{tmp:tslinvplane} that $u_\infty$ is homogeneous with respect to the origin, and moreover
	\begin{equation}\label{tmp:uinftyhom}
		N(u_\infty, 0, \rho) = N(u_\infty, 0, 0) = N(u_\infty^0, Y_\infty, 0) = N(u_\infty^0, 0, 0) = \tilde{\Lambda},
	\end{equation}  
	where we use \eqref{eq:determinefreqlimit} in the last equality.
	However, by the assumption on $X_i$ and Lemma \ref{lm:spvarin_far}
	\[ N(u_i, X_i, \rho r_i) \leq N^{u_i}_{X_i}(\rho r_i) + \delta_i < \left( \tilde{\Lambda}_i - \delta_0 \right) + \delta_i, \]
	and thus
	\[ N(u_\infty, 0, \rho) = \lim_{i\to \infty} N(T_{X_i, r_i} u_i,0, \rho) = \lim_{i\to \infty} N(u_i, X_i, \rho r_i) \leq \tilde{\Lambda} - \delta_0. \]
	This is in contradiction with \eqref{tmp:uinftyhom}.
	Therefore we have proven \eqref{eq:delta0pinch}.
	
	\bigskip
	
	\textit{Proof of \eqref{eq:tn}.} Notice that if the set on the left hand side is empty,
	there is nothing to prove. So we assume that is not the case. We argue by contradiction as before, but this time with $\beta$ fixed to be the value we just found. That is, we assume there are $(u_i, D_i) \in \mathfrak{H}(R, \Lambda)$ and $\delta_i,  r_i\to 0$ and $r_i \leq r_{in}(\delta_i)$ such that
	\[ \sup_{X\in B_{2r_i}(0) \cap \Nt(u_i)} N_X^{u_i}(r_i) \leq \tilde{\Lambda}_i \leq C(\Lambda), \] 
	the set 
	\[ F_i = \{X\in B_{2r_i}(0) \cap \Nt(u_i): N_X^{u_i}(\rho r_i) \geq \tilde{\Lambda}_i - \delta_i \} \]
	$2\tau r_i$-effectively spans a $(d-2)$-dimensional linear subspace $V_i$,
	and yet there exists some $X'_i \in \Ct_{r_i}(u_i) \cap B_{2r_i}(0) \cap \Nt(u_i) $ which satisfies $\dist(X'_i, V_i) \geq 2\beta r_i$.
	
	The first part of the proof is the same as in \uline{Step $1$} and \uline{Step $2$} above, and we obtain that $T_{X_i^0, r_i} u_i$ converges to $u_\infty^0$, which is a harmonic homogeneous polynomial and is invariant along a $(d-2)$-dimensional linear subspace $V_\infty$.

	Similarly to \uline{Step $3$} above, we let 
	\[ Y'_i = \frac{X'_i - X_i^0 }{r_i} \in \frac{ \overline{D}_i - X_i^0}{r_i}. \]
	Recall we have shown in \uline{Step $1$} in the proof of \eqref{eq:delta0pinch} that the sequence of domains $\frac{\overline{D}_i - X_i^0}{r_i} \to \overline{\RR^d_+}$.
	Then modulo passing to a subsequence $Y'_i$ converges to some point $Y'\in \overline{\RR^d_+} \cap \overline{B_4(0)}$; and $T_{X'_i, r_i} u_i$ converges to some harmonic function (not necessarily homogeneous) $u'_\infty$ in $\left( \RR^d_+ - Y' \right) \cap B_5(0)$.
	Additionally, since $\dist(X'_i, V_i) \geq 2\beta r_i$, we have
	\[ \dist(Y'_i, \widetilde{V}_i)  = \frac{\dist(X'_i, V_i)}{r_I} \geq 2\beta, \]
	and thus 
	\begin{equation}\label{eq:offspine}
		\dist(Y', V_\infty) \geq 2\beta>0.
	\end{equation}
%
	Moreover, similarly to the proof before
	we can show there is a constant $a'\approx 1$ such that $u_\infty^0(Y' + Y ) = a' ~u'_\infty(Y)$, and thus
	\begin{equation}\label{cl:cpderlimit}
		\nabla u_\infty^0(Y'+Y) = a' \nabla u'_\infty(Y).
	\end{equation}

	Since $u^0_\infty$ is invariant along the $(d-2)$-dimensional linear subspace $V_\infty\subset \partial \RR^d_+$, by identifying the coordinate system in $V_\infty^\perp \times V_\infty$ with $\RR^2 \times \RR^{d-2}$ we have
	\begin{itemize}
		\item either $u^0_\infty$ is a harmonic function in one variable, and then necessarily $u^0_\infty$ is linear in that variable and (modulo a change of coordinates in $V_\infty^\perp$) 
			\[ u^0_\infty (x_1, x_2, y) = c(x_1)_+, \quad \text{ where }x_1 \in V_\infty^\perp \cap \RR^d_+, x_2 \in V_\infty^\perp \cap \partial\RR^d_+ \text{ and } y\in V_\infty. \]
		\item or $u^0_\infty$ is a harmonic function in two variables, and then necessarily we have (modulo a change of coordinates in $V_\infty^\perp$) in polar coordinates
			\begin{equation}\label{eq:2varhhp}
				u_\infty^0(r,\omega, y) = \tilde{c} r^N \sin(N\omega), \quad \text{ where } r\geq 0, \omega \in [0, \pi], y\in V_\infty.   
			\end{equation} 
			Here the degree of homogeneity is determined as	$N = N(u^0_\infty, 1)= N(u^0_\infty, 0)$ and $N \geq 2$. 
	\end{itemize}
	Since $u_\infty^0$ satisfies $\iint_{B_1(0)} |u_\infty^0|^2 ~dY = 1$, the constant $c=\alpha_d^1 $ only depends on the dimension $d$, where $\alpha_d^1$ satisfies $\omega_{d-1} |\alpha_d^1|^2  \int_0^1 t^2 (1-t^2)^{\frac{d-1}{2}} ~dt = 1$ with $\omega_{d-1}$ being the volume of $(d-1)$-dimensional unit ball. And the constant $\tilde{c}$ satisfies
	\[ \frac{\pi}{2} |\tilde{c}|^2 \omega_{d-2} \int_0^1 t^{2N} (1-t^2)^{\frac{d-2}{2}} ~dt = 1, \] 
	and thus it depends on the dimension as well as the degree $N$. However $\tilde{c}$ has a uniform lower bound independent of $N$, which we denote by $|\tilde c|\geq \alpha_d^2$.
	
	The assumption $X'_i \in \Ct_{r_i}(u_i)$ implies that 
	\[ \inf_{B_\beta(0)} |\nabla T_{X'_i, r_i} u_i| \leq \alpha_0. \]
	Recall that $\nabla T_{X'_i, r_i }u_i \to \nabla u'_\infty $ in $L^2$, then modulo passing to a subsequence we have almost everywhere convergence. This, combined with the continuity of $\nabla T_{X'_i, r_i} u_i$ and $\nabla u'_\infty$, implies that
	\[ \inf_{B_\beta(0)} |\nabla u'_\infty| \leq \alpha_0. \]
	By \eqref{cl:cpderlimit}, it follows that
	\begin{equation}\label{tmp:smallgrad}
		\inf_{B_\beta(Y')} |\nabla u_\infty^0| = |a'| \inf_{B_\beta(0)} |\nabla u'_\infty| \leq  \tilde{\alpha}_0,
	\end{equation} 
	where $\tilde{\alpha}_0$ is $\alpha_0$ multiplied by the upper bound of $|a'|$, which is a dimensional constant.
	In the first case we have $|\nabla u_\infty^0| \equiv \alpha_d^1$. So by choosing $\tilde{\alpha}_0< \alpha_d^1$, we can guarantee the first case does not happen. In the second case, we have that for any $(r, \omega) \in \RR_+ \times [0,\pi]$
	\[ |\nabla u^0_\infty|^2 = \tilde{c}^2 \left[|Nr^{N-1} \sin(N\omega) |^2 + |N r^{N-1} \cos(N\omega)|^2 \right] = \tilde{c}^2 N^2 r^{2(N-1)}. \]	
	We identify $Y' \in \overline{\RR^d_+}$ with $(r_0, \omega_0, y_0)$. (The angle $\omega_0 = 0$ or $\pi$ if $Y' \in \partial \RR^d_+$, and $\omega_0 \in (0, \pi)$ if $Y' \in \RR^d_+$.) Then by \eqref{eq:offspine}, we have
	\[ r_0 = \dist(Y', V_\infty) \geq 2\beta>0. \]
	Therefore
	\[ \inf_{B_\beta(Y')} |\nabla u_\infty^0| = |\tilde{c}| N (r_0 - \beta)^{N-1} |\geq 2\alpha_d^2 \cdot \beta^{C(\Lambda)-1}. \]
	By choosing $\tilde{\alpha}_0 < 2\alpha_d^2 \cdot \beta^{C(\Lambda)-1}$, we get a contradiction with the bound in \eqref{tmp:smallgrad}.
\end{proof}

\section{Spatial variations of the frequency function}\label{sec:spvar}

The following lemma gives an estimate of the spatial variation of the frequency function (at the same scale) for harmonic functions on the upper half space. It can be viewed as a quantitative version of Lemma \ref{lm:spinebd}.

\begin{lemma}\label{lm:spvar}
	Let $u$ be a harmonic function in $\RR^d_+$ such that $u$ vanishes on $\partial \RR^d_+$. Let $N(X, r)$ denote the frequency function of $u$ centered at $X\in \pR$ and at scale $r>0$. Suppose $N(0,3) \leq \Lambda$. Then there exists a constant $C$ (depending on $d$ and $\Lambda$), such that for any $0<r<1$ and any $X_1, X_2\in B_1(0) \cap \pR$ with $|X_1 - X_2| \leq r/2$, we have
	\begin{equation}\label{eq:spvar}
		|N(X_1, r) - N(X_2, r)| \leq C \left( W^{\frac12}(X_1, r) + W^{\frac12}(X_2, r) \right),
	\end{equation}
	where $W(X_j, r) = N(X_j, 3r/2) - N(X_j, r/2)$ for $j=1,2$.
\end{lemma}
\begin{proof}
	Let $Y = X_2 - X_1$. Any point on the line segment $[X_1, X_2]$ can be written as $X=X_1 + t Y\in \pR $, where $t\in [0,1]$. Let 
	\[ \rho_X(Z) = \frac{Z-X}{|Z-X|} \]
	denote the radial direction with center $X$. Similarly to \eqref{eq:spvar_ls}, we have
	\begin{align}
		& \frac{d}{dt} N(X_1 + tY, r) \nonumber \\
		& = \frac{d}{dt} \log N(X_1 + tY, r) \cdot N(X_1 + tY, r) \nonumber \\
		& = \frac{2}{H(X, r)} \left[ r \int_{\partial B_r(X)} \langle \nabla u(Z), Y \rangle \partial_{\rho_X} u(Z)  ~d\mathcal{H}^{d-1} - N(X, r) \cdot \int_{\partial B_r(X)} \langle \nabla u(Z), Y \rangle u(Z) ~d\mathcal{H}^{d-1} \right].\label{tmp:spvarderNr}
	\end{align} 
	Recall that $u\in C^1 (\overline{\RR^d_+})$ and we extend $u$ by zero in its complement, so for simplicity we drop $\RR^d_+$ in the integration region.
	Let 
	\[ \E_j(Z) = \langle \nabla u(Z), Z-X_j \rangle - N(X_j, |Z-X_j|) \cdot u(Z), \quad \text{ for } j=1,2.  \]
	Then
	\begin{align*}
		\langle \nabla u(Z), Y\rangle = \E_1(Z) - \E_2(Z) + N(X_1, |Z-X_1|) \cdot u(Z) - N(X_2, |Z-X_2|) \cdot u(Z).
	\end{align*}
	Inserting this into the equality \eqref{tmp:spvarderNr} we get
	\begin{align}
		& \frac{d}{dt} N(X_1 + tY, r) \nonumber \\
		& = \frac{2}{H(X, r)} \int_{\partial B_r(X)} \left( \E_1(Z) - \E_2(Z) \right) \left[ r\partial_{\rho_X} u(Z) - N(X,r) \cdot u(Z) \right] d\mathcal{H}^{d-1} \nonumber \\
		& + \frac{2}{H(X,r)} \int_{\partial B_r(X)} \left( N(X_1, |Z-X_1|) - N(X_2, |Z-X_2|) \right)\left[ r\partial_{\rho_X} u(Z) - N(X,r) \cdot u(Z) \right] u(Z) d\mathcal{H}^{d-1} \nonumber \\
		& = : \I_1(X) + \I_2(X) \label{tmp:spvar}
	\end{align}
	We write
	\begin{align*}
		& N(X_1, |Z-X_1|) - N(X_2, |Z-X_2|) \\
		& = \left( N(X_1, |Z-X_1|) - N(X_1, r) \right) + \left( N(X_1, r) - N(X_2, r) \right) + \left( N(X_2, r) - N(X_2, |Z-X_2|) \right) \\
		& =: \mathcal{S}_1(Z) + \mathcal{S} + \mathcal{S}_2(Z),
	\end{align*}
	where the middle term $\mathcal{S} = N(X_1, r) - N(X_2, r)$ is independent of $Z$.
	For any $X$ on the line segment $[X_1, X_2]$ and any $Z\in \partial B_r(X)$, we have
	\[ \frac{r}{2} \leq |Z-X_j| \leq \frac32 r, \quad \text{ for } j=1,2. \]
	Hence
	\[ \left| \mathcal{S}_j(Z) \right| =  \left| N(X_j, |Z-X_j|) - N(X_j, r) \right| \leq W(X_j, r). \]
	On the other hand, plugging $\mathcal{S}$ into the second term of \eqref{tmp:spvar} we get
	\begin{align*}
		& \mathcal{S} \left[  \int_{\partial B_r(X)} r u \partial_{\rho_X} u ~d\mathcal{H}^{d-1}(Z) - N(X,r) \int_{\partial B_r(X)} u^2 ~d\mathcal{H}^{d-1}(Z) \right] \\
		& \qquad \qquad = \mathcal{S} \left[ r D(X, r) - N(X,r) H(X, r) \right] = 0.
	\end{align*}
	Notice that the above argument is exactly the same as the proof \eqref{eq:spvar_ls0}, with $\mathcal{S} = N_1 - N_2$.
	We can estimate the remaining terms of $\I_2(X)$ as follows:
	\begin{align}
		| \I_2(X) | & \leq \frac{2}{H(X,r)} \left(  W(X_1, r) + W(X_2, r) \right) \left[ \left( r + N(X, r)\right) \int_{\partial B_r(X)} u^2 ~d\mathcal{H}^{d-1} + r \int_{\partial B_r(X)} |\nabla u|^2 ~d\mathcal{H}^{d-1} \right] \nonumber \\
		& \leq  2\left(  W(X_1, r) + W(X_2, r) \right)\left[ C(\Lambda) + \frac{r\int_{\partial B_r(X)} |\nabla u|^2 ~\mathcal{H}^{d-1} }{H(X,r) } \right]. \label{tmp:I2}
	\end{align}
	We estimate $\I_1(X)$ by Cauchy-Schwartz inequality, and obtain
	\begin{equation}\label{tmp:I1}
		|\I_1(X)| \lesssim \left[ r \left( \frac{ \int_{\partial B_r(X)} |\nabla u|^2 ~d\mathcal{H}^{d-1}}{H(X,r)} \right)^{\frac12} + N(X,r) \right]\cdot  \left( \frac{1}{H(X,r)} \int_{\partial B_r(X)} |\E_1(Z)|^2 + |\E_2(Z)|^2 ~d\mathcal{H}^{d-1} \right)^{\frac12}.
	\end{equation}
	
	To further estimate $\I_1(X)$ and $\I_2(X)$, we make the following observation. Since $u$ is a harmonic function which vanishes on the boundary, we have $u^2$ is sub-harmonic in $\RR^d$. Hence
	\[ \fint_{\partial B_\rho(X)} u^2 \sH \text{ is increasing with respect to }\rho. \]
	It follows that
	\begin{align*}
		H(X_j, r) = \int_{\partial B_r(X_j)} u^2 \sH & \leq \frac2r \iint_{A_{r, \frac32 r}(X_j)} u^2 ~dZ \\
		& \leq \frac2r \iint_{A_{\frac{r}{2}, 2r}(X)} u^2 ~dZ \\
		& \leq 3\int_{\partial B_{2r}(X)} u^2 \sH = 3H(X, 2r).
	\end{align*}
	On the other hand by the monotonicity formula (or \eqref{eq:Hrcmpu} for the Laplacian) we have
	\[ \frac{H(X,2r)}{H(X,\frac32 r)} \leq \left( \frac43 \right)^{d-1+N(X,2r)} \leq C(d,\Lambda). \]
	In the second inequality we use Lemma \ref{lm:freqbd} (for the Laplacian operator in the upper half-space) to bound $N(X,2r)$ from above by $N(0,3)$.
	Therefore 
	\begin{equation}\label{eq:Hchangecenter}
		H(X,r) \geq c(d,\Lambda) ~H\left(X_j, \frac32 r \right) \quad \text{ for }j=1,2.
	\end{equation} 
	
	Now we allow $X=X_t$ to move in the line segment $[X_1, X_2]$, as  $t$ varies in $[0,1]$. Since $|X_1 - X_2| \leq r/2$, the integration region satisfies
	\[ \bigcup_{t\in [0,1]} \partial B_r(X_t) \subset A_{\frac{\sqrt{3}}{2} r, \frac32 r}(X_j). \] 
	By \eqref{tmp:I2} and \eqref{eq:Hchangecenter}, we have
	\begin{align}
		\int_0^1 |\I_2(X_t)| ~dt & \leq 2(W(X_1, r) + W(X_2, r)) \left[ C_1 + \frac{C_2 r}{H(X_1, r)} \int_0^1 \int_{\partial B_r(X_t)} |\nabla u|^2 \sH ~dt \right] \nonumber \\
		& \leq 2(W(X_1, r) + W(X_2, r)) \left[ C_1+  C_2 \frac{r D(X_1, \frac32 r)}{H(X_1, \frac32 r)}  \right] \nonumber \\
		& \leq C(\Lambda) \left( W(X_1, r) + W(X_2, r) \right).\label{tmp:I2f}
	\end{align}
	Next we use the Cauchy-Schwarz inequality to estimate
	\begin{align*}
		\int_0^1 |\I_1(X_t)| ~dt &\lesssim \left\{ \int_0^1 \left[ \frac{r^2}{H(X_t, r)} \int_{\partial B_r(X_t)} |\nabla u|^2 \sH + N^2(X_t,r) \right] ~dt \right\}^{\frac12}  \\
		& \qquad \times \left\{ \int_0^1 \frac{1}{H(X_t, r)} \int_{\partial B_r(X_t)} |\E_1(Z)|^2 + |\E_2(Z)|^2 \sH ~dt \right\}^{\frac12} \\
		& \lesssim \left[ \frac{r^2}{H(X_1, \frac32 r)} \cdot  D(X_1, \frac32 r) + C(\Lambda) \right]^{\frac12} \\
		& \qquad \times \left[ \sum_{j=1}^2 \frac{1}{H(X_j, \frac32 r)} \iint_{A_{\frac{\sqrt{3}}{2} r, \frac32 r}(X_j)} |\E_j(Z)|^2 ~dZ \right]^{\frac12}.
	\end{align*}
	Recall Proposition \ref{prop:monotonicity} (for the Laplace operator), in particular \eqref{eq:derNr} and \eqref{eq:Rhr}. We have
	\begin{align*}
		& \frac{1}{H(X_j, \frac32 r)} \iint_{A_{\frac{\sqrt{3}}{2} r, \frac32 r}(X_j)} |\E_j(Z)|^2 ~dZ \\
		& = \frac{1}{H(X_j, \frac32 r)} \iint_{A_{\frac{\sqrt{3}}{2} r, \frac32 r}(X_j)} \left|\langle \nabla u(Z), Z-X_j \rangle - N(X_j, |Z-X_j|) \cdot u(Z) \right|^2 ~dZ \\
		& \lesssim r \int_{\frac{\sqrt{3}}{2} r}^{\frac32 r} \frac{\rho}{H(X_j, \rho)} \int_{\partial B_\rho(X_j)} \left|\partial_{\rho_{X_j}} u(Z) - \frac{N(X_j, \rho)}{\rho} u(Z)  \right|^2 \sH ~d\rho \\
		& \lesssim r \int_{\frac{\sqrt{3}}{2} r}^{\frac32 r} R_h(X_j, \rho) ~d\rho \\
		& \lesssim W(X_j, r).
	\end{align*}
	Therefore
	\begin{equation}\label{tmp:I1f}
		\int_0^1 |\I_1(X_t)| ~dt \leq C(\Lambda) \left( W^{\frac12}(X_1, r) + W^{\frac12}(X_2, r) \right). 
	\end{equation} 
	
	Combining \eqref{tmp:spvar}, \eqref{tmp:I2f} and \eqref{tmp:I1f}, we get
	\begin{align*}
		|N(X_2, r) - N(X_1, r)| \leq \int_0^1 |\I_1(X_t)| + |\I_2(X_t)| ~dt \leq C_{sv} \left(W^{\frac12}(X_1, r) + W^{\frac12}(X_2, r) \right),
	\end{align*}
	where the constant $C_{sv}$ depends on $d$ and $\Lambda$.
\end{proof}

Next we estimate the spatial variation of the frequency function for harmonic functions in Dini domains. Firstly, recall that for boundary points we define the frequency function $N_X(r)$ using the transformation $\Psi_X$ in Section \ref{sec:reduce}, i.e.
\[ N_X(r):= \widetilde{N}(v_X, r) = N(v_X, r) \exp\left( C\int_0^r \frac{\theta(4s)}{s} ~ds \right); \]
 and the modified frequency function $\widetilde{N}(v_X, r)$
is monotone increasing with respect to $r$. We define the frequency drop accordingly as
\[ W_X(r) := N_X(3r/2) - N_X(r/2). \]

\begin{prop}\label{prop:spvar}
	Let $R, \Lambda>0$ be fixed. There exists a constant $C_{sv}>0$ such that the following holds. Suppose $(u,D) \in \mathfrak{H}(R, \Lambda)$ and $X_1, X_2$ are two points in $ B_R(0) \cap \pD$ with $|X_1 - X_2| \leq r/3$ and $0<r\leq 2R$. 
	Then we have
	\begin{equation}\label{eq:spvar_sharp}
		\left|N(v_{X_1}, r) - N(v_{X_2}, r) \right| \leq C_{sv} \left( W_{X_1}^{\frac12}(r) + W_{X_2}^{\frac12}(r) + \theta(4r) \right).
	\end{equation}
\end{prop}
\begin{remark}
	In fact, for the purpose of this paper, it suffices to use the rough estimate
	\begin{equation}\label{eq:spvar_rough}
		\left|N(v_{X_1}, r) - N(v_{X_2}, r) \right| \leq C \left( r^{\frac12} + \theta(4r) \right).
	\end{equation}
	The proof of \eqref{eq:spvar_rough} is similar to that of Lemma \ref{lm:spvarin_close}, where we simply bound each term in \eqref{tmp:tderNr_} by taking absolute values inside. 
	Here we include the proof of the sharper estimate \eqref{eq:spvar_sharp}.
\end{remark}
\begin{proof}
	Notice that unlike the case of the upper half space in Lemma \ref{lm:spvar}, here $\partial D$ is not flat. So we need to be more careful in defining the line in $\pD$ which connects $X_1$ to $X_2$. Assume $X_1 = (x_1, \varphi(x_1))$ and $X_2 = (x_2, \varphi(x_2))$. Let $y = x_2 - x_1 \in \RR^{d-1}$. We define for any $t\in [0,1]$
	\[ X_t = (x_t, \varphi(x_t)) := (x_1 + ty, \varphi(x_1 + ty)) \in \pD. \]
	We also remark that
	\begin{equation}\label{eq:tangderpD}
		\frac{d}{dt} X_t = \left(y, \langle \nabla\varphi(x_t), y \rangle \right) \text{ points in the tangential direction of $\pD$ at $X_t$}.
	\end{equation} 
	
	Recall that for any $X\in \pD$, we define
	\[ D(v_X, r) = \iint_{B_r \cap \Omega_X} \mu |\nabla_g v_X|^2_g ~dV_g = \iint_{\Psi_{X}(B_r) \cap D} |\nabla u|^2 ~dZ =: \widehat{D}(X, r); \] 
	and
	\[ H(v_X, r) = \int_{\partial B_r \cap \Omega_X} \mu v_X^2 \sH = \left( 1+ O(\theta(4r)) \right) \int_{\Psi_X(\partial B_r) \cap D} u^2 \sH = \left( 1+ O(\theta(4r)) \right) \widehat{H}(X, r), \]
	where we introduce the definition
	\[ \widehat{H}(X, r) := \int_{\Psi_X(\partial B_r) \cap D} u^2 \sH = \int_{\Psi_X(\partial B_r)} u^2 \sH. \]
	Let
	\[ \widehat{N}(X, r) := \frac{r \widehat{D}(X, r) }{\widehat{H}(X, r) }, \]
	then the frequency function of $v_X$ satisfies
	\begin{equation}\label{eq:NvX}
		N(v_X, r) = \frac{r D(v_X, r)}{H(v_X, r)} = \left( 1+O(\theta(4r)) \right) \frac{r\widehat{D}(X, r) }{\widehat{H}(X, r) } = \left( 1+O(\theta(4r)) \right) \widehat{N}(X,r),
	\end{equation}

	We claim that
	\begin{equation}\label{eq:PsiBr}
		\partial \Psi_X(B_r) = \partial B_r(X+ 3r\tilde{\theta}(r) e_d) = \Psi_X(\partial B_r).
	\end{equation} 
	In fact by the definition of the transformation map $\Psi$ (see \eqref{def:Psi}), it is clear that
	\[ \Psi(\partial B_r) = \partial B_r + 3r\tilde{\theta}(r) e_d = \partial B_r(3r\tilde{\theta}(r)e_d). \]
	To understand what the set $\partial \Psi(B_r)$ is, we first study the set $\Psi(B_r)$. Clearly
	\[ \Psi(B_r) = \bigcup_{\rho\in [0,r)} \Psi(\partial B_\rho) = \bigcup_{\rho\in [0,r)}  \partial B_\rho \left( 3\rho \tilde{\theta}(\rho) e_d \right). \]
	Consider the function 
	\[ f: \rho \in [0,r) \mapsto -\rho + 3\rho\tilde{\theta}(\rho), \]
	which corresponds to the height of the lower-most point of the (shifted) ball $\partial B_\rho + 3\rho \tilde{\theta}(\rho) e_d$. A simple computation shows that $f$ is a continuous function, and
	\begin{align*}
		f'(\rho) = -1+3\tilde{\theta}(\rho) + 3\rho \tilde{\theta}'(\rho) & = -1 + 3\tilde{\theta}(\rho) + \frac{3}{\log^2 2} \int_{\rho}^{2\rho} \frac{\theta(2s)-\theta(s)}{s} ~ds \\
		& \leq -1 + 3\theta(4\rho) + \frac{3}{\log 2} \theta(4\rho) \\
		& \leq -1 + 13 ~\theta(4r).
	\end{align*} 
	By choosing $r$ sufficiently small so that $\theta(4r)< 1/26$, we can guarantee that $f$ is a monotone decreasing function. In particular, this implies that the balls $\Psi(\partial B_\rho) = \partial B_{\rho}\left( 3\rho\tilde{\theta}(\rho) e_d \right)$ with $\rho \in [0,r)$ are nested, i.e.
	\[ B_{\rho}(3\rho \tilde{\theta}(\rho) e_d)  \subset B_{\rho'}(3\rho' \tilde{\theta}(\rho') e_d), \quad \text{ if } \rho \leq \rho'. \]
	In fact, let $X\in B_{\rho}(3\rho \tilde{\theta}(\rho) e_d)$ be arbitrary. Then
	\begin{align*}
		|X- 3\rho'\tilde{\theta}(\rho') e_d| & \leq |X- 3\rho\tilde{\theta}(\rho) e_d| + \left( 3\rho'\tilde{\theta}(\rho') - 3\rho \tilde{\theta}(\rho) \right) \\
		& < \rho + f(\rho') + \rho' - (f(\rho) + \rho)) \\
		& = \rho' + (f(\rho') - f(\rho)) \\
		& \leq \rho'.
	\end{align*}
	Hence $X\in B_{\rho'}(3\rho' \tilde{\theta}(\rho') e_d)$.
	Moreover by the intermediate value theorem $f(\rho)$ assumes all values between $\lim_{\rho \to r-} f(\rho) = -r+3r\tilde{\theta}(r)$ and $\lim_{\rho \to 0+} f(\rho) = 0$. 
	This finishes the proof of the claim \eqref{eq:PsiBr}.

	As in \eqref{tmp:spvarderNr} in the proof of Lemma \ref{lm:spvar}, we have
	\begin{align}
		& \frac{d}{dt} \widehat{ N}(X_t, r) \nonumber \\
		& = \frac{d}{dt} \log \widehat{ N}(X_t, r) \cdot \widehat{ N}(X_t, r) \nonumber \\
		& = \frac{1}{\widehat{H}(X_t, r)} \left[ r \frac{d}{dt} \widehat{D}(X_t, r) - \widehat{N}(X_t, r)\cdot \frac{d}{dt} \widehat{H}(X_t, r) \right] \nonumber \\
		& = \frac{2}{\widehat{H}(X_t, r)} \left[ r \int_{X_t + \partial\Psi (B_r)} \left\langle \nabla u(Z), \frac{d}{dt} X_t \right\rangle ~ \partial_{n} u(Z)  \sH \right. \nonumber \\
		& \qquad \qquad + r\int_{X_t + \Psi(B_r) \cap \pD } \left\langle \nabla u(Z), \frac{d}{dt} X_t \right\rangle ~ \partial_{n} u(Z) \sH  \nonumber \\
		& \qquad \qquad \left. - ~ \widehat{N}(X_t, r) \cdot \int_{X_t + \Psi(\partial B_r)} \left\langle \nabla u(Z), \frac{d}{dt} X_t \right\rangle u(Z) \sH \right] \nonumber \\
%
		& = \frac{2}{\widehat{H}(X_t, r)} \left[ r \int_{X_t + \Psi (\partial B_r)} \left\langle \nabla u(Z), \frac{d}{dt} X_t \right\rangle ~ \partial_{n} u(Z)  \sH \right. \nonumber \\
		& \qquad \qquad + r\int_{X_t + \Psi(B_r) \cap \pD } \left\langle \nabla u(Z), \frac{d}{dt} X_t \right\rangle ~ \partial_{n} u(Z) \sH  \nonumber \\
		& \qquad \qquad \left. - ~ \widehat{N}(X_t, r) \cdot \int_{X_t + \Psi(\partial B_r)} \left\langle \nabla u(Z), \frac{d}{dt} X_t \right\rangle u(Z) \sH \right] \nonumber \\
		& =: \I_1(X_t) + \I_2(X_t) + \I_3(X_t).\label{tmp:tderNr_}	
	\end{align}
	In the second to last equality we simply use \eqref{eq:PsiBr}.
	
	Since $u$ vanishes on the boundary and at any $Z=(x, \varphi(x)) \in \pD$, the vector $(y, \langle \nabla \varphi(x), y \rangle)$ points in the tangential direction of $\pD$, we have
	\[ \left\langle \nabla u(Z), (y, \langle \nabla \varphi(x), y \rangle) \right\rangle = 0. \]
	Hence by \eqref{eq:tangderpD}, we know
	\begin{align*}
		\left\langle \nabla u(Z), \frac{d}{dt} X_t \right\rangle = \left\langle \nabla u(Z), \frac{d}{dt} X_t - (y, \langle \nabla \varphi(x), y \rangle) \right\rangle = \left\langle \nabla u(Z), (0, \langle \nabla \varphi(x_t) - \nabla \varphi(x), y \rangle) \right\rangle.
	\end{align*}
	Thus
	\begin{align*}
		\left|\left\langle \nabla u(Z), \frac{d}{dt} X_t \right\rangle \right| \leq |\partial_d u(Z)| \cdot \theta(|x_t -x|) |y| \leq |\partial_d u(Z)| \cdot \theta(r)|X_2 - X_1|.
	\end{align*}
	Therefore we can estimate $\I_2(X_t)$ as follows
	\begin{align}
		\left| \I_2(X_t) \right| & \leq \frac{2r}{ \widehat{H}(X_t, r)} \int_{\Psi_{X_t}(B_r) \cap \pD} |\partial_d u|^2 \theta(r) |X_2-X_1| \sH \nonumber  \\
		& \lesssim \frac{r^2 \theta(r)}{ \widehat{H}(X_t, r)} \int_{\Psi_{X_t}(B_r) \cap \pD} |\nabla u|^2 \sH \nonumber \\
		& \lesssim \frac{\theta(r)}{\widehat{H}(X_t, r)} \mathcal{H}^{d-1} ( \Psi_{X_t}(B_r) \cap \pD) \cdot \fiint_{\Psi_{X_t}(B_{3r/2})} u^2 ~dZ \nonumber \\
		& \lesssim \theta(r).\label{tmp:I2Xt}
	\end{align}
	
	Since $u^2$ is subharmonic (after extending by zero) and 
	\[ |X_t - X_1| \leq (1+\theta(R))|X_2 -X_1| < \frac{5}{12}r, \] we have
	\begin{align}
		\widehat{H}(X_t,r) = \int_{\partial B_r(X_t+3r\tilde{\theta}(r) e_d)} u^2 \sH & \geq \frac{12}{11r} \iint_{A_{\frac{r}{12}, r}(X_t+3r\tilde{\theta}(r)e_d) } u^2 ~dZ \nonumber \\
		& \geq \frac{12}{11r} \iint_{A_{\frac{5}{48}r, \frac{47}{48} r}(X_t) } u^2 ~dZ \nonumber \\
		& \geq \frac{12}{11r} \iint_{A_{\frac{25}{48}r, \frac{27}{48}r}(X_1)} u^2~dZ \nonumber \\
		& \geq \frac{1}{22} \int_{\partial B_{\frac{r}{2}}(X_1)} u^2 \sH. \label{eq:whHrlb}
	\end{align}
	An elementary computation shows that
	\begin{align*}
		X_2 - X_1 - (y, \langle \nabla\varphi(x_t), y \rangle) & = (x_2-x_1, \varphi(x_2)-\varphi(x_1)) - (x_2-x_1, \langle \nabla\varphi(x_t), x_2-x_1 \rangle) \\
		& = (x_2 - x_1, \langle \nabla \varphi(x) - \nabla\varphi(x_t) , x_2-x_1 \rangle),
	\end{align*}
	where $x$ is some point on the line segment $[x_1, x_2]\subset \RR^{d-1}$. Therefore by the assumption of $\varphi$ in Definition \ref{def:Dini}, we have
	\[ \left|X_2 - X_1 - (y, \langle \nabla\varphi(x_t), y \rangle) \right|\leq \theta(|x_2-x_1|)\cdot |x_2-x_1|^2. \]
	Hence
	\begin{align*}
		 \langle \nabla u(Z), (y, \langle \nabla\varphi(x_t), y\rangle) \rangle & = \langle \nabla u(Z), X_2 -X_1 \rangle - \langle \nabla u(Z), X_2 - X_1 - (y, \langle \nabla\varphi(x_t), y\rangle) \rangle \\
		& = \langle \nabla u(Z), X_2 -X_1 \rangle + O(r^2 \theta(r))\cdot |\nabla u(Z)|
	\end{align*}
	Inserting $\langle \nabla u(Z), X_2 - X_1\rangle$ into the equality \eqref{tmp:tderNr_}, we can use the same argument as in the proof of Lemma \ref{lm:spvar} (as well as \eqref{eq:whHrlb}) for the estimate, so we will not repeat the argument here.
	
	The remaining term satisfies
	\begin{align*}
		\E(X_t) \leq & r^2 \theta(r) \left[ \frac{2r}{\widehat{H}(X_t, r) } \cdot \int_{X_t + \Psi(\partial B_r)}  |\nabla u|^2 \sH \right. \\
		& \qquad + \left. \frac{C(\Lambda) }{\widehat{H}(X_t, r)} \left( \int_{X_t + \Psi(\partial B_r)} |\nabla u|^2 \sH \right)^{\frac12} \left( \int_{X_t + \Psi(\partial B_r)} u^2 \sH \right)^{\frac12} \right] \\
		& \leq r^{\frac32} \theta(r) \left[  \frac{ 2r \mathlarger{\int}_{X_t + \Psi(\partial B_r)}  |\nabla u|^2 \sH  }{\mathlarger{\int}_{\partial B_{\frac{r}{2}}(X_1)} u^2 \sH }  + C(\Lambda) \left(\frac{r \mathlarger{\int}_{X_t + \Psi(\partial B_r)} |\nabla u|^2 \sH }{\mathlarger{\int}_{\partial B_{\frac{r}{2}}(X_1)} u^2 \sH  } \right)^{\frac12} \right].
	\end{align*}
	Hence by integrating on the interval $t\in [0,1]$, we get
	\begin{align}
		\int_0^1 \E(X_t)~dt & \leq r^{\frac32}\theta(r) \left[ \frac{ 2r \mathlarger{\iint}_{A_{\frac{r}{2}, \frac32 r}(X_1)}  |\nabla u|^2 \sH  }{\mathlarger{\int}_{\partial B_{\frac{r}{2}}(X_1)} u^2 \sH }  + C(\Lambda) \left( \frac{r \mathlarger{\iint}_{A_{\frac{r}{2}, \frac32 r}(X_1)} |\nabla u|^2 \sH }{\mathlarger{\int}_{\partial B_{\frac{r}{2}}(X_1)} u^2 \sH  } \right)^{\frac12} \right] \nonumber \\
		& \leq C'(\Lambda) r^{\frac32} \theta(r).\label{tmp:Et}
	\end{align}
	Finally, combining \eqref{tmp:tderNr_}, \eqref{tmp:I2Xt} and \eqref{tmp:Et} we conclude that
	\begin{align*}
		|\widehat{N}(X_2, r) - \widehat{N}(X_1, r)| & \leq C\left(W_{X_1}^{\frac12}(r) + W_{X_2}^{\frac12}(r) + \theta(r) \right) + \int_0^1 \E(X_t) ~dt \\
		& \leq C'\left(W_{X_1}^{\frac12}(r) + W_{X_2}^{\frac12}(r) + \theta(r) \right).
	\end{align*}
	Therefore by \eqref{eq:NvX}, we get
	\begin{align*}
		|N(v_{X_2},r) - N(v_{X_1}, r)| &\leq |\widehat{N}(X_2, r) - \widehat{N}(X_1, r)| + C\theta(4r) \left(N(v_{X_1}, r) + N(v_{X_2}, r)\right) \\
		& \leq C''(\Lambda) \left(W_{X_1}^{\frac12}(r) + W_{X_2}^{\frac12}(r) + \theta(4r) \right).
	\end{align*}
	

\end{proof}

Next we estimate the spatial variation for interior points. Compared to the spatial variation for boundary points in Proposition \ref{prop:spvar}, here we use a similar idea but it suffices to get a rough estimate.
\begin{lemma}\label{lm:spvarin_close}
	Let $R, \Lambda>0$ be fixed. Suppose $(u,D) \in \mathfrak{H}(R, \Lambda)$ and $p\in D\cap B_{\frac{R}{10}}(0) \cap \Nt(u)$ satisfies $\dist(p,\pD) \leq \frac35 R \tilde{\theta}(\frac35 R)$.
	Then for any $r\leq \frac{12}{25} R$ we have 
	\begin{equation}
		\left|N(v_{q}, r) - N(p, r) \right| \leq C\left( r^{\frac12} + \theta(4r) \right), \quad \text{ for every } q\in B_{\frac{r}{5}}(p) \cap \pD \cap \Nt(u);
	\end{equation}
	and 
	\begin{equation}
		\left|N(q, r) - N(p, r) \right| \leq C r^{\frac12}, \quad \text{ for every } q\in B_{\frac{r}{5}}(p) \cap D \cap \Nt(u).
	\end{equation}
	Here the constant $C>0$ depends on $d$ and $\Lambda$.
%
\end{lemma}
\begin{proof}
	The idea of the proof is similar to that of Proposition \ref{prop:spvar}, with $p \in D$ playing the role of $X_1$ and 
	\[ \tilde{q}:= \left\{\begin{array}{ll}
		q, & \text{ if } q\in D \\
		q+3r\tilde{\theta}(r) e_d, & \text{ if } q\in \pD
	\end{array}\right.\] playing the role of $X_2$.
	This way it suffices to estimate $|N(\tilde{q}, r) - N(p, r)|$ for both the interior and boundary cases.
	
	For $t\in [0,1]$ we first define $ X^0_t := p + t(\tilde{q}-p)$. However, the line segment $[\tilde{q}, p]$ may not be completely contained in $D$. Let $A = \{t\in [0,1]: X_t \notin D\}$. Since $\pD$ is a $C^1$ graph, $A$ is a finite union of closed intervals and can be written as 
	\[ A= \bigcup_{\ell=1}^{m} ~[t_{2k-1}, t_{2k}]. \] 
	To unify the notation at the endpoints, we also set $t_0 = 0$ and $t_{2m+1} = 1$. 
	We revise the definition of $X^0_t$ for all $t\in A$ by replacing $X_t^0$ by its projection onto $\pD$, i.e.
	\[ X^0_t: = \left(x_1 + t(x_2 - x_1), \varphi(x_1 + t(x_2 - x_1)) \right) \in \pD, \]
	where we denote $p=(x_1, z_1)\in D$ and $\tilde{q} = (x_2, z_2)\in D$. Notice that for every $t, t' \in [0,1]$, we have
	\begin{align*}
		|X_t^0 - X_{t'}^0| 	& \leq \left|\left((t-t')(x_2-x_1), \varphi(x_1 + t(x_2-x_1)) - \varphi(x_1 + t'(x_2-x_1)) \right) \right| \\
		& \leq |t-t'| \sqrt{|x_2 - x_1|^2 + (\theta(R)\cdot |x_2 -x_1|)^2 } \\
		& \leq \frac{10}{9} |t-t'| |x_2-x_1| \\
		& \leq \frac{10}{9} |t-t'| |q-p|.
	\end{align*} 
	Finally we let
	\[ X_t := \left\{\begin{array}{ll}
		X_t^0, & t=0 \text{ or } 1 \\
		X_t^0 + 3r\tilde{\theta}(r) e_d, & t\in A \\
		\text{the line segment connecting $X_{t_{2k-2}}$ and $X_{2k-1}$}, & t\in [t_{2k-2}, t_{2k-1}].
	\end{array} \right. \]
	After such modification we guarantee that $X_t \in D$.
	It is not hard to see that the total length $\ell$ of the curve $\{X_t\}_{t\in [0,1]}$ satisfies 
	\[ \ell \leq \frac{10}{9} |q-p| + 2 \cdot 3r\tilde{\theta}(r) < \frac{r}{4}.  \]
	We then reparametrize it (still denoted by $\{X_t\}_{t\in [0,1]}$) so that it has unit speed, and thus
	\begin{equation}\label{eq:dtXt}
		\left| \frac{d}{dt} X_t \right| \leq \ell < \frac{r}{4}.  
	\end{equation} 
	 Moreover, for every $t, t'\in [0,1]$ we have
	\begin{equation}
		|X_t - X_{t'}| \leq \ell < \frac{r}{4}.
	\end{equation}	
	
	As in \eqref{tmp:tderNr_}, we have
	\begin{align}
		& \frac{d}{dt} N(X_t, r) \nonumber \\
		& = \frac{1}{H(X_t, r)} \left[ r \frac{d}{dt} D(X_t, r) - N(X_t, r) \cdot \frac{d}{dt} H(X_t, r) \right] \nonumber \\
		& = \frac{2}{H(X_t, r)} \left[ r \int_{\partial B_r(X_t)} \left\langle \nabla u, \frac{d}{dt} X_t \right\rangle \partial_n u \sH + r \int_{B_r(X_t) \cap \pD} \left\langle \nabla u, \frac{d}{dt} X_t \right\rangle \partial_n u \sH \right. \nonumber \\
		& \qquad \qquad - \left. N(X_t, r) \cdot \int_{\partial B_r(X_t) } \left\langle \nabla u, \frac{d}{dt} X_t \right\rangle u \sH \right] \nonumber \\
		& =: \I_1(X_t) + \I_2(X_t) + \I_3(X_t).\label{tmp:tderNrin}
	\end{align}
	(There is a slight abuse of notation in the last equality, where $\partial_n u$ denotes the derivative pointing away from the sphere $\partial B_r(X_t)$ in the first integral, and $\partial_n u$ denotes the normal derivative pointing away from $\pD$ in the second integral.)
	We remark that the main difference with the boundary case \eqref{tmp:tderNr_} is that \eqref{tmp:tderNr_} does not have the second term $\I_2(X_t)$ since there $\frac{d}{dt} X_t$ is in the tangential direction on $\pD$ and thus $\langle \nabla u, \frac{d}{dt} X_t \rangle = 0$. Here, on the other hand, the vector $\frac{d}{dt} X_t$
	does not always point in the tangential direction of $\pD$. 
	We first estimate this term using an idea similar to the Rellich identity. 
	
	It is not hard to see that
	\begin{equation}\label{eq:Rellichnormal}
		\divg(|\nabla u|^2 e_d) = 2 \langle \nabla (\partial_d u), \nabla u \rangle = 2 \left(  \divg(\partial_d u \cdot \nabla u) -\partial_d u \cdot \Delta u \right) = 2\divg(\partial_d u \cdot \nabla u).
	\end{equation}
	Integrating both sides of \eqref{eq:Rellichnormal} on the region $B_r(X_t) \cap D$ and using the divergence theorem, we get
	\begin{align}
		& \int_{\partial B_r(X_t)} |\nabla u|^2 \left\langle e_d, \frac{ Z-X_t}{|Z-X_t|} \right\rangle \sH + \int_{B_r(X_t) \cap \pD} |\nabla u|^2 \langle e_d, n_D(Z) \rangle \sH \nonumber \\
		& = 2\int_{\partial B_r(X_t)} \partial_d u \left\langle \nabla u, \frac{ Z-X_t}{|Z-X_t|} \right\rangle \sH + 2 \int_{B_r(X_t) \cap \pD} \partial_d u \cdot \partial_n u \sH.\label{tmp:Rellichnormal}
	\end{align} 
	On every boundary point $Z=(x,\varphi(x)) \in \pD$, we have $\nabla u(Z) = (\partial_n u) n_D(Z)$, where 
	\[ n_D(Z) = \frac{(\nabla \varphi(x), -1) }{\sqrt{1+|\nabla \varphi(x)|^2} } \text{ is the unit outer normal vector.} \]
	 Hence
	\[ \partial_d u \cdot \partial_n u = \langle \nabla u, e_d \rangle \partial_n u =  |\partial_n u|^2 \langle n_D(Z), e_d \rangle = \frac{-(\partial_n u)^2}{\sqrt{1+|\nabla \varphi(x)|^2}},  \]
	and \eqref{tmp:Rellichnormal} is simplified to
	\begin{align}
		& \int_{B_r(X_t) \cap \pD} \frac{(\partial_n u)^2}{ \sqrt{1+|\nabla \varphi(x)|^2}} \nonumber \\
		& = 2\int_{\partial B_r(X_t)} \partial_d u \left\langle \nabla u, \frac{ Z-X_t}{|Z-X_t|} \right\rangle \sH - \int_{\partial B_r(X_t)} |\nabla u|^2 \left\langle e_d, \frac{ Z-X_t}{|Z-X_t|} \right\rangle \sH \nonumber \\
		& \leq 3\int_{\partial B_r(X_t)} |\nabla u|^2 \sH.
	\end{align}
	On the other hand by \eqref{eq:dtXt}, for every $Z\in \pD$ we have
	\begin{align*}
		\left| \left\langle \nabla u, \frac{d}{dt} X_t \right\rangle \right| = \left| \partial_n u\left\langle n_D(Z), \frac{d}{dt} X_t \right\rangle \right| & = \frac{\left| \partial_n u \right|}{\sqrt{1+|\nabla \varphi(x)|^2}} \left| \left\langle (\nabla \varphi(x), -1), \frac{d}{dt} X_t \right\rangle \right| \\
		& \leq \frac{r}{ 4} \frac{\left| \partial_n u \right|}{\sqrt{1+|\nabla \varphi(x)|^2}}.
	\end{align*}
	Therefore
	\begin{align}
		\left| \I_2(X_t) \right| & = \frac{2r}{H(X_t, r)} \int_{B_r(X_t) \cap \pD} \left| \left\langle \nabla u, \frac{d}{dt} X_t \right\rangle \partial_n u \right| \sH \nonumber \\
		& \leq \frac{r^2}{2H(X_t, r)} \int_{B_r(X_t) \cap \pD} \frac{(\partial_n u)^2}{\sqrt{1+|\nabla \varphi(x)|^2}} \sH \nonumber \\
		& \leq \frac{2r^2}{H(X_t, r)} \int_{\partial B_r(X_t) } |\nabla u|^2 \sH.\label{tmp:I2Xt}
	\end{align}

	Notice that
	\begin{align*}
		|X_t -p| = |X_t - X_1| \leq \ell < \frac{r}{4}.
	\end{align*}
	By the sub-harmonicity of $u^2$, we have
	\begin{align*}
		H(X_t, r) \gtrsim \frac{1}{r} \iint_{A_{\frac{r}{4}, r}(X_t)} u^2 ~dZ \geq \frac{1}{r} \iint_{A_{\frac{r}{2}, \frac34 r}(p)} u^2 ~dZ \gtrsim H\left(p, \frac{r}{2} \right).
	\end{align*}
	Therefore it follows from \eqref{eq:dtXt}, \eqref{tmp:I2Xt} and H\"older's inequality that
	\[ |\I_1(X_t)|,  |\I_2(X_t)| \lesssim \frac{r^2}{ H(p, \frac{r}{2} )} \int_{\partial B_r(X_t)} |\nabla u|^2 \sH, \]
	and
	\begin{align*}
		|\I_3(X_t)| & \leq \frac{C(\Lambda)r}{ H(X_t, r)} \left(\int_{\partial B_r(X_t)} |\nabla u|^2 \sH \right)^{\frac12} \left(\int_{\partial B_r(X_t)} u^2 \sH \right)^{\frac12} \\
		& \leq C(\Lambda) r^{\frac12} \left( \frac{r\mathlarger{\int}_{\partial B_r(X_t)} |\nabla u|^2 \sH }{H\left(p, \frac{r}{2} \right) } \right)^{\frac12}.
	\end{align*}
	Integrating on the interval $t\in [0,1]$ and using the fact that
	\[ \bigcup_{t\in [0,1]} \partial B_r(X_t) \subset B_{\frac54 r}(p), \] 
	we get
	\begin{align*}
		& \left| N(\tilde{q}, r) - N(p, r) \right| \\
		& \leq \int_0^1 |\I_1(X_t)| + |\I_2(X_t)| + |\I_3(X_t)| ~dt \\
		& \lesssim \frac{r^2}{H(p, \frac{r}{2})} \iint_{B_{\frac54 r}(p)} |\nabla u|^2 ~dZ + C(\Lambda) r^{\frac12} \left( \frac{r\mathlarger{\iint}_{B_{\frac54 r}(p)} |\nabla u|^2 \sH }{H\left(p, \frac{r}{2} \right) } \right)^{\frac12} \\
		& \lesssim C'(\Lambda) r^{\frac12},
	\end{align*}
	where we use the doubling property and the upper bound of $N(p, \frac54 r)$ in the last inequality.
	
	In the case of $q\in \pD$, it remains to estimate $|N(\tilde{q}, r) - N(v_q,r)|$. By \eqref{eq:NvX} and the properties of the transformation map $\Psi_X$, we know
	\[ N(v_q, r) = \left( 1+ O(\theta(4r)) \right) \widehat{N}(q,r) = \left( 1+ O(\theta(4r)) \right) N(\tilde{q}, r). \]
	Therefore
	\begin{align*}
		|N(\tilde{q},r) - N(v_q,r)| \lesssim_{\Lambda} \theta(4r),
	\end{align*}
	and it follows that
	\[ |N(v_q, r) - N(p,r)| \lesssim_\Lambda r^{\frac12} + \theta(4r). \]
\end{proof}

\section{$L^2$-best approximation theorem and the frequency function}\label{sec:L2approx}
Let $\mu$ be a Radon measure.
 We define the $k$-dimensional $L^2$ beta-number of $\mu$ in $B_r(p)$ as
\begin{equation}\label{def:beta}
	\beta_{\mu}^k(p,r) = \inf_{L^k} \left( \frac{1}{r^{2+k}} \int_{B_r(p)} \dist^2(Y, L^k) d\mu(Y) \right)^{\frac12}, 
\end{equation} 
where the infimum is taken over all $k$-dimensional affine planes. We remark that the scaling factor $r^{2+k}$ is to make sure the above definition is scale-invariant, when the measure $\mu$ is \textit{$k$-dimensional}.
In Section \ref{sec:covering} we will use the following Reifenberg theorem with dimension $k=d-2$ to give uniform volume bounds on $\Ct_{r}(u)\cap \Nt(u)$. 
\begin{theorem}[Discrete Reifenberg Theorem \cite{NVRR} ]\label{thm:Reifenberg}
	There exist constants $\eta_{dr}>0$ and $C_{dr}>0$ depending only on the dimension such that the following holds. Let $\{B_{r_X/5}(X)\}_{X\in \CC} \subset B_3(0) \subset \RR^d$ be a collection of pairwise disjoint balls with their centers $X\in B_1(0)$, and let $\mu=\sum_{X\in \CC} \omega_k r_X^k \delta_X$ be the associated packing measure. Assume that for each $B_s(p_0) \subset B_2(0)$
	\[ \iint_{B_s(p_0)} \int_0^s \left|\beta_\mu^k(p,r ) \right|^2 ~\frac{dr}{r} d\mu(p) \leq \eta_{dr} s^k.   \]
	Then we have the uniform estimate
	\begin{equation}
		\sum_{X\in \CC} r_X^k \leq C_{dr}.
	\end{equation}
\end{theorem} 
\begin{remark}
	For the discrete measure above associated to the covering, we notice that for any $X\in \CC \subset \spt\mu$, if $r\leq r_X/5$ then $B_r(X) \cap \spt\mu = \{X\}$, and thus by definition $\beta_{\mu}^k(X, r) = 0$.
\end{remark}

We will also need a variant of it, the Rectifiable Reifenberg Theorem.
\begin{theorem}[Rectifiable Reifenberg Theorem \cite{NVRR}]\label{thm:RReifenberg}
	There exist constants $\eta_{rr}>0$ and $C_{rr}>0$ depending only on the dimension such that the following holds. Let $E \subset B_2(0)$ be a $\mathcal{H}^k$-measurable set, and assume that for each $B_s(p_0) \subset B_2(0)$
	\[ \iint_{B_s(p_0)} \int_0^s \left|\beta_{\mathcal{H}^k \res_{E}}^k(p,r ) \right|^2 ~\frac{dr}{r} d\mathcal{H}^k\res_{E} \leq \eta_{rr} s^k.   \]
	Then $E$ is $k$-rectifiable, and for each $p_0\in E\cap B_1(0)$ and $0<s<1$ we have $\mathcal{H}^k(B_s(p_0)) \leq C_{rr} s^k $.
\end{theorem}

\begin{remark}
	There have been many generalizations of the classical Reifenberg theorem, see for instance \cite{Tol, DT}.
	Reifenberg theorems of the form in Theorems \ref{thm:Reifenberg} and \ref{thm:RReifenberg} first appear in \cite[Theorems 3.4 and 3.3]{NVRR} (see also \cite[Theorems 40 and 42]{NVAH}). We refer interested readers to \cite{ENV} for a review and generalization of Reifenberg-type theorems.
\end{remark}

The main goal of this section is to control the $(d-2)$-dimensional beta number $\beta_\mu^{d-2}(p,r)$ by the drop of the frequency function inside the ball $B_r(p)$.
\begin{theorem}\label{thm:L2appx}
	Let $R, \Lambda>0$ and $\delta_{in} \in (0,1)$ be fixed. Let $r_b := \min\{r_{in}, r_s, \frac{R}{10}\}>0$ where $r_{in}$ is given in Lemma \ref{lm:spvarin_far} with parameters $\delta_{in}$ and $\rho=\frac16$, and $r_s$ is given in Lemma \ref{lm:sym}. Then for any $(u,D) \in \HH(R, \Lambda)$, any $p\in \Ct_{r}(u) \cap \mathcal{N}(u) \cap B_{\frac{R}{10}}(0)$ with $\dist(p,\pD) \leq 6r_{in} \tilde{\theta}(6r_{in})$,
	 and any $0<r\leq r_b$, 
	the $(d-2)$-dimensional $L^2$ $\beta$-number in $B_r(p)$ satisfies
	\begin{align}
		\left|\beta_\mu^{d-2}(p, r) \right|^2 & \leq C_b\left[ \frac{1}{r^{d-2}} \int_{B_r(p)} \left( \widetilde{W}_X(r) + \delta_{in} \cdot \chi_{r_{cs}(X)/6 < r \leq r_{cs}(X) } \right) ~d\mu(X) \right. \nonumber \\
		& \qquad \qquad + \left. \frac{\mu(B_r(p))}{r^{d-2}} \left( r + \theta(24 r) \right) \right], \label{eq:L2appx}
	\end{align}
	where $\mu$ is an arbitrary Radon measure such that
	\[ \spt\mu \subset \mathcal{N}(u) \cap B_{6r_{in} \tilde{\theta}(6r_{in})} (\pD), \] 
	and $\widetilde{W}_X(r) := N_X(6r)-N_X(r)$. The constant $C_b$ depends on $d$, the uniform frequency bound $C(\Lambda)$ and $C_s$ in Lemma \ref{lm:sym}. (In particular it is independent of the choice of $\delta_{in}$.)
\end{theorem}
\begin{remark}
	Heuristically, the above estimate \eqref{eq:L2appx} says the following. Suppose the right hand side is small, that is to say, (on an average sense) points on the supports of the measure $\mu$ have small frequency drop at scale $r$ (i.e. if $\widetilde{W}_X(r) \ll 1$ when $X\in \spt\mu$). 
	 Then $\spt\mu$ must be distributed in a small neighborhood of a $(d-2)$-dimensional affine subspace. (In particular $\spt\mu$ can be distributed near a $k$-dimensional affine subspace with $k< d-2$.)
	 Intuitively, the reason is that if we look at the tangent function centered at any singular point $p\in \Ct_r(u) \cap \Nt(u)$, it can be almost invariant along an at most $(d-2)$-dimensional affine subspace.
\end{remark}

\uline{Step 1}.
Let $C_\mu$ denote the center of mass of the measure $\mu$ on $B_r(p)$, i.e.
\[ C_\mu := \fint_{B_r(p)} X ~d\mu(X). \]
To compute the beta-number $\beta_\mu^k(p,r)$ we consider the bilinear quadratic form $Q(v,w)$ associated to $\mu\res_{B_r(p)}$, which is defined as
\[ Q(v,w):= \frac{1}{r^{2+k}} \int_{B_r(p)} \langle v, Y-C_\mu \rangle \langle w, Y-C_\mu \rangle ~d\mu(Y). \]
Let $\lambda_1 \geq \cdots \geq \lambda_d \geq 0$ be the eigenvalues of $Q$ in decreasing order, and let $w_1, \cdots, w_d$ be the corresponding eigenvectors with unit norm. Then
\begin{equation}\label{eq:ev}
	\lambda_j = \sup \left\{\frac{1}{r^{2+k}} \int_{B_r(p)} \left| \langle w, Y-C_\mu \rangle \right|^2 ~d\mu(Y) \text{ s.t. } |w|^2=1 \text{ and } \langle w, w_i \rangle = 0 \text{ for all } i < j \right\}. 
\end{equation} 
Then it is not hard to see that the affine $k$-plane $L^k:= C_\mu + \spn\{w_1, \cdots, w_k\}$ achieves the infimum in the definition \eqref{def:beta} of the beta number, and moreover
\begin{equation}\label{eq:betaabsev}
	\left| \beta_\mu^k(p,r) \right|^2 = \frac{1}{r^{2+k}} \int_{B_r(p)} \dist^2(Y, L^k) d\mu(Y) = \lambda_{k+1} + \cdots + \lambda_d. 
\end{equation} 
\medskip

\uline{Step 2}. Let $(u,D) \in \HH(R,\Lambda)$ and $p\in \spt\mu$. Without loss of generality we assume $u$ is normalized (with respect to the ball $B_r(p)$) so that 
\begin{equation}\label{eq:nmlzpr}
	\frac{1}{r^d} \iint_{B_r(p)} u^2 ~dZ = 1. 
\end{equation} 
Since a constant multiple of a function does not change its frequency function, this normalization does not affect \eqref{eq:L2appx}.
%
\begin{itemize}
	\item For any $X\in B_{2r}(p) \cap \pD$ and any $\rho \in [r,6r]$, we have
\begin{align}
	H(v_X, \rho) & \approx \frac{1}{\rho} \iint_{A_{\frac{\rho}{2}, \frac32 \rho}(0) \cap \Omega_X} \mu v_X^2 ~dV_g \quad \text{ by } \eqref{eq:Hrcmpu} \nonumber \\
	& \lesssim \frac{1}{\rho} \iint_{A_{\frac{\rho}{4}, 2\rho}(X) \cap D} u^2 ~dZ \nonumber \\
	& \lesssim \frac{1}{r} \iint_{B_{14r}(p) \cap D} u^2 ~dZ \nonumber \\
	& \lesssim r^{d-1}, \quad \text{ by } \eqref{eq:nmlzpr} \text{ and } \eqref{eq:Hrcmpu}.\label{eq:Hrhobd_Xb}
\end{align}

	\item Suppose $X\in B_r(p) \cap D$. We consider two cases: either $10r\leq r_{cs}(p)$ so that we can use the monotonicity formula for $X$, or $10r>r_{cs}(p)$. In the first case, as above we use the monotonicity formula for $X$, and get
				\begin{equation}\label{eq:Hrhobd_Xinclose}
					H(X, \rho) \leq \frac{2}{\rho} \iint_{A_{\rho, \frac32\rho}(X) \cap D} u^2 ~dZ \lesssim \frac1r \iint_{B_{10r}(p) \cap D} u^2 ~dZ \lesssim r^{d-1}.
				\end{equation}
		In the second case, let $q_X \in \pD$ so that $|q_X-X| = \dist(X, \pD)$. By the definition of the critical scale we know $\dist(X, \pD) \ll r_0<6r$. Then by the monotonicity formula centered at $q_X$, we have
		\begin{align}
			H(X, \rho)\leq \frac{2}{\rho} \iint_{A_{\rho, \frac32\rho}(X) \cap D} u^2 ~dZ & \lesssim \frac1r \iint_{B_{10r}(q_X) \cap D} u^2 ~dZ \nonumber \\
			& \lesssim \frac1r \iint_{B_{r/2}(q_X) \cap D} u^2 ~dZ \nonumber \\
			& \leq \frac1r \iint_{B_{r}(p) \cap D} u^2 ~dZ = r^{d-1}.\label{eq:Hrhobd_Xinfar}
		\end{align}
\end{itemize}
The constants in the above inequalities depend on $\Lambda$ (in fact, the constant grows exponentially as $\Lambda$ grows).

For any $Z\in D$ and $j = 1, \cdots, d$ we have
\begin{equation}\label{eq:tmpQ}
	\lambda_j \langle w_j, \nabla u(Z) \rangle = Q(w_j, \nabla u(Z)) = \frac{1}{r^{2+k}} \int_{B_r(p)} \langle w_j, X-C_\mu \rangle \langle \nabla u(Z) , X-C_\mu \rangle ~d\mu(X).
\end{equation}
Since
\[ \int_{B_r(p)} \left( X - C_\mu \right) ~d\mu(X) = 0,  \]
we may add the term 
\[ \frac{1}{r^{2+k}} \left( \int_{B_r(p)} \langle w_j, X-C_\mu \rangle ~d\mu(X) \right) \cdot \left( \langle \nabla u(Z), C_\mu - Z \rangle + c u(Z) \right) \]
to the right hand side of \eqref{eq:tmpQ} without changing the equality:
\[ \lambda_j \langle w_j, \nabla u(Z) \rangle = \frac{1}{r^{2+k}} \int_{B_r(p)} \langle w_j, X-C_\mu \rangle \left[ \langle \nabla u(Z) , X-Z \rangle + c u(Z)  \right] ~d\mu(X). \]
Here $c$ is a constant whose value is to be determined later. Thus by H\"older's inequality and \eqref{eq:ev},
\begin{align*}
	& \left| \lambda_j \langle w_j, \nabla u(Z) \rangle  \right|^2  \\
	& \leq \left( \frac{1}{r^{2+k}} \int_{B_r(p)} \left| \langle w_j, X-C_\mu \rangle \right|^2 ~d\mu(X) \right)  \cdot \left( \frac{1}{r^{2+k}} \int_{B_r(p)} \left|  \langle \nabla u(Z) , Z-X \rangle - c u(Z) \right|^2 ~d\mu(X) \right) \\
	& = \lambda_j \left( \frac{1}{r^{2+k}} \int_{B_r(p)} \left|  \langle \nabla u(Z) , Z-X \rangle - c u(Z) \right|^2 ~d\mu(X) \right).
\end{align*} 
Assume $\lambda_j>0$, otherwise there is nothing to prove.
It then simplifies to
\begin{align*}
	\lambda_j \left| \langle \nabla u(Z), w_j \rangle \right|^2 \leq &  \frac{1}{r^{2+k}} \int_{B_r(p)} \left|  \langle \nabla u(Z) , Z-X \rangle - c u(Z) \right|^2 ~d\mu(X).
\end{align*} 



Next we integrate this inequality on the annulus region $A_{3r,4r}(p) \cap D = \{Z\in D: 3r<|Z-p|<4r \}$. 
We get
\begin{align}
	& \frac{\lambda_j}{r^{d-2}} \iint_{A_{3r, 4r}(p) \cap D} \left| \langle \nabla u(Z), w_j \rangle \right|^2 ~dZ \nonumber \\
	& \leq \frac{1}{r^{d-2}} \iint_{A_{3r, 4r}(p) \cap D} \frac{1}{r^{2+k}} \int_{B_r(p)} \left|  \langle \nabla u(Z) , Z-X \rangle - c u(Z) \right|^2 ~d\mu(X)  ~dZ \nonumber \\
	& \leq \frac{1}{r^{k}} \int_{B_r(p)} \frac{1}{r^{d}} \iint_{A_{2r,5r}(X) \cap D } \left|  \langle \nabla u(Z) , Z-X \rangle - c u(Z) \right|^2 ~dZ ~d\mu(X). \label{tmp:beta1}
\end{align}

Now we fix the constant $c$ as
\begin{equation}\label{def:c}
	c= \left\{\begin{array}{ll}
		N(v_p, 6r), & \text{ if } p\in \pD \\
		N(p, 6r), & \text{ if } p\in D \text{ and } 6r\leq r_{cs}(p) \\
		N(v_q, 6r), & \text{ if } p\in D \text{ and } 6r> r_{cs}(p)
	\end{array} \right.
\end{equation}
where in the last line $q\in \pD$ is such that $|p-q| = \dist(p, \pD)$.
(Notice that $c$ can not depend on $X \in \spt\mu \cap B_r(p)$, but it can depend on the fixed point $p$.)

We start with the more complicated \uline{case (i): when $X\in \pD$.} 
Recall \eqref{eq:uvhom}, we have
\begin{align*}
	\left|  \langle \nabla u(Z) , Z-X \rangle - cu(Z) \right| \leq \left| \langle \nabla_g v_X (Y), Y \rangle - cv_X(Y) \right| + C|Y|\theta(4|Y|) \cdot |\nabla v_X(Y)|
\end{align*}  
where $Y = \Psi_{X}^{-1}(Z)\in \Omega_X$ is the preimage of $Z\in D$ and $v_X(Y) = u(\Psi_X(Y)) = u(Z)$. (The constant $C$ is independent of the choice of $X \in \pD$.) Hence by a change of variable and \eqref{eq:detDPsi}, \eqref{eq:YPsiY}, the integrand becomes
\begin{align}
	& \frac{1}{r^{d}} \iint_{A_{2r,5r}(X) \cap D } \left|  \langle \nabla u(Z) , Z-X \rangle - cu(Z) \right|^2 ~dZ \nonumber \\
	& \leq \left(1+O\left(\theta(24r) \right) \right) \frac{1}{r^{d}} \iint_{A_{r,6r}(0) \cap \Omega_{X} } \left|  \langle \nabla_g v_X(Y), Y \rangle - cv_X(Y) \right|^2 ~dY \nonumber \\
	& \qquad \qquad + C \frac{\theta(24 r)}{r^{d-2}} \iint_{A_{r,6r}(0) \cap \Omega_X} |\nabla v_X(Y)|^2 ~dY.\label{eq:changeutov}
\end{align}
We denote the second term as $\E$, and we have
\begin{equation}\label{tmp:beta2}
	\E \lesssim \frac{\theta(24 r)}{r^{d-2}} D(v_X, 6r) \lesssim \frac{\theta(24 r)}{r^{d-1}} N_X(6r) \cdot H(v_X, 6r)  \lesssim \theta(24 r),
\end{equation}
with constants depending on $d, r_*$ and $\Lambda$. We have used \eqref{eq:Hrhobd_Xb} in the last inequality.

We rewrite the leading order term in the right hand side of \eqref{eq:changeutov} as follows
\begin{align}
	& \frac{1}{r^{d}} \iint_{A_{r,6r}(0) \cap \Omega_{X} } \left|  \langle \nabla_g v_X(Y), Y \rangle - cv_X(Y) \right|^2 ~dY \nonumber \\
	& = \frac{1}{r^{d}} \int_{r}^{6r} \rho^2 \int_{\partial B_\rho \cap \Omega_{X} } \left| \partial_\rho v_X - c \frac{v_X}{\rho} \right|^2 ~d\mathcal{H}^{d-1}(Y) ~d\rho \nonumber \\
	& \lesssim \left(1+O(\theta(24r)) \right)  \int_r^{6r} \frac{1}{\rho^{d-2}} \int_{\partial B_\rho \cap \Omega_{X}} \mu \left|\partial_\rho  v_X - c \frac{v_X}{\rho} \right|^2 ~dV_{\partial B_\rho} ~d\rho \label{eq:tmpRhr} 
\end{align}
Recall we have shown in \eqref{eq:Rhr} that
\begin{equation}\label{tmp:Rhrho}
	R_h(v_X, \rho) = \frac{2\rho}{H(v_X, \rho)} \int_{\partial B_\rho \cap \Omega_{X}} \mu \left| \partial_\rho v_X - N(v_X, \rho) ~\frac{v_X}{\rho} \right|^2 ~dV_{\partial B_\rho} \leq \left. \frac{d}{dr}\right|_{r=\rho}  N_X( r). 
\end{equation} 

\uline{If $p\in \pD$}, we fix the constant $c$ to be $N(v_p, 6r)$. We need to estimate the difference
\begin{equation}\label{tmp:triangle}
	|N(v_X,\rho) - c| \leq |N(v_X, \rho) - N(v_X, 6r)| + |N(v_X, 6r) - N(v_p, 6r)|.
\end{equation}
By \eqref{eq:spvar_rough} in Proposition \ref{prop:spvar}, we have
\begin{equation}\label{tmp:Ncmp1}
	|N(v_p, 6r) - N(v_X, 6r)| \leq C \left( r^{\frac12} + \theta(24r) \right),
\end{equation}
where the constant $C$ depends on $\Lambda$. On the other hand, by the monotonicity of the modified frequency function, for any $\rho \in [r, 6r]$ we have
\begin{align}
	& |N(v_X, \rho) - N(v_X, 6r)| \nonumber \\
	& = \left|N_X(\rho) \exp\left(-C\int_0^{\rho}\frac{\theta(4\tau)}{\tau} ~d\tau \right) - N_X(6r) \exp\left(-C\int_0^{6r}\frac{\theta(4\tau)}{\tau} ~d\tau \right) \right| \nonumber \\
	& \leq \left|N_X(\rho) - N_X(6r) \right| + N_X(6r) \left|\exp\left(-C\int_0^{\rho}\frac{\theta(4\tau)}{\tau} ~d\tau \right) - \exp\left(-C\int_0^{6r}\frac{\theta(4\tau)}{\tau} ~d\tau \right) \right| \nonumber \\
	& \leq N_X(6r) - N_X(r) + C(\Lambda) \int_{r}^{6r} \frac{\theta(4\tau)}{\tau} ~d\tau \nonumber \\
	& \leq N_X(6r) - N_X(r) + C'(\Lambda) \theta(24 r).\label{tmp:Ncmp2} 
\end{align}
Combining \eqref{tmp:triangle}, \eqref{tmp:Ncmp1}, \eqref{tmp:Ncmp2} and using the monotonicity of the modified frequency function, we get
\begin{equation}\label{eq:Xpb}
	|N(v_X, \rho) - c| \leq \widetilde{W}_X(r) + C \left( r^{\frac12} + \theta(24r) \right),
\end{equation}
where we use the notation $\widetilde{W}_X(r) := N_X(6r) -N_X(r)$ to denote the frequency drop. 

\underline{If $p\in D$ and $r\leq r_{cs}(p)/6$}, recall that we have fixed $c=N(p, 6r)$ in \eqref{def:c}. By Lemma \ref{lm:spvarin_close} we have
\[ 
| N(v_X, 6r) - N(p, 6r) | \leq C(r^{\frac12} + \theta(24 r)). \]
Therefore
\begin{align}
	|N(v_X, \rho) - c| & \leq |N(v_X, \rho) - N(v_X, 6r) | + |N(v_X, 6r) - N(p, 6r)| \nonumber \\
	 & \leq N(v_X, 6r) - N(v_X, r) + C(r^{\frac12} + \theta(24 r)) \nonumber \\
	& \leq \widetilde{W}_X(r) + C(r^{\frac12} + \theta(24 r)).\label{eq:Xbpin_close}
\end{align}

\underline{If $p\in D$ and $r> r_{cs}(p)/6$}, recall by \eqref{def:c} and \eqref{def:Npr} we have fixed $c=N(v_q, 6r)$,
where $q\in \pD$ is such that $|p-q| = \dist(p, \pD)$. Since every $X\in B_r(p)$ satisfies
\[ |X-q| \leq |X-p| + |p-q| < 2r, \]
by \eqref{eq:spvar_rough} in Proposition \ref{prop:spvar} we have
\[ |N(v_X, 6r) - N(v_q, 6r)| \leq C\left(r^{\frac12} + \theta(24r) \right). \]
Hence 
\begin{align*}
	|N(v_X, \rho) - c| \leq |N(v_X, \rho) - N(v_X, 6r)| + |N(v_X, 6r) - N(v_q, 6r)| \leq \widetilde{W}_X(r) + C(r^{\frac12} + \theta(24 r)).
\end{align*}

Combining any one of \eqref{eq:Xpb}, \eqref{eq:Xbpin_close}
with \eqref{eq:tmpRhr} and \eqref{tmp:Rhrho}, we conclude
\begin{align}
	& \frac{1}{r^{d}} \iint_{A_{r,6r}(0) \cap \Omega_{X} } \left|  \langle \nabla_g v_X(Y), Y \rangle - cv_X(Y) \right|^2 ~dY \nonumber \\
	& \lesssim   \int_r^{6r} \frac{1}{\rho^{d-2}} \int_{\partial B_\rho \cap \Omega_{X}} \mu \left|\partial_\rho  v_X - N(v_X, \rho) \frac{v_X}{\rho} \right|^2 ~dV_{\partial B_\rho} ~d\rho + \int_r^{6r} |N(v_X, \rho) - c|^2 \frac{H(v_X, \rho)}{\rho^d} ~d\rho \nonumber \\
	& \lesssim \int_r^{6r} R_h(v_X, \rho) \frac{H(v_X,\rho)}{\rho^{d-1}} ~d\rho + \left( \widetilde{W}_X(r) + r  + \theta(24 r) \right) \int_r^{6r} ~\frac{d\rho}{\rho} \nonumber \\
	& \lesssim \widetilde{W}_X(r) + r + \theta(24 r).\label{tmp:beta3}
\end{align}
We used any one of \eqref{eq:Hrhobd_Xb}, \eqref{eq:Hrhobd_Xinclose} or \eqref{eq:Hrhobd_Xinfar} in the last two inequalities.

Finally, combining \eqref{tmp:beta1}, \eqref{eq:changeutov}, \eqref{tmp:beta2} and \eqref{tmp:beta3}, we get
\begin{align}
	& \frac{\lambda_j}{r^{d-2}} \iint_{A_{3r, 4r}(p) \cap D} |\langle \nabla u(Z), w_j \rangle |^2 ~dZ \nonumber \\
	& \lesssim \frac{1}{r^k} \int_{B_r(p)} \left[ \frac{1}{r^d} \iint_{A_{r,6r}(0) \cap \Omega_X} \left| \langle \nabla_g v_X(Y), Y \rangle - N_p(2r) \cdot v_X(Y) \right|^2 ~dY + \theta(24 r) \right] ~d\mu(X) \nonumber \\
	& \lesssim \frac{1}{r^k} \int_{B_r(p)} \widetilde{W}_X(r) ~d\mu(X) + \frac{\mu(B_r(p))}{r^k} \left( r + \theta(24 r) \right).\label{eq:betaintm}
\end{align}

\underline{Case (ii): when $X\in D$ and $r\leq r_{cs}(X)/6$}. We bound the integrand in \eqref{tmp:beta1} by
\begin{equation}\label{eq:Xin_close}
	\frac{1}{r^{d}} \iint_{A_{2r,5r}(X) \cap D } \left|  \langle \nabla u(Z) , Z-X \rangle - cu(Z) \right|^2 ~dZ \lesssim \int_{2r}^{5r} \frac{1}{\rho^{d-2}} \int_{\partial B_\rho(X) \cap D} \left|\partial_{\rho_X} u - c ~\frac{u}{\rho} \right|^2 \sH ~d\rho
\end{equation}
Recall we have shown in Proposition \ref{prop:intmonotonicity} that for any $\rho\leq r_{cs}(X)$, we have
\begin{equation}\label{tmp:Rhrho}
	R_h(X, \rho) = \frac{2\rho}{H(X, \rho)} \int_{\partial B_\rho(X) \cap D} \left| \partial_{\rho_X} u - N(X, \rho) ~\frac{u}{\rho} \right|^2 \sH \leq \left. \frac{d}{dr}\right|_{r=\rho}  N(X, r). 
\end{equation} 
Hence to estimate the right hand side of \eqref{eq:Xin_close} it suffices to bound the difference $|N(X, \rho) - c|$.

\underline{If $p\in \pD$}, the constant $c=N(v_p, 6r)$ by \eqref{def:c}. Hence by Lemma \ref{lm:spvarin_close} and the assumption $6r \leq r_{cs}(X)$, we have
\begin{align}
	|N(X, \rho) - c| & \leq |N(X, \rho) - N(X, 6r)| + |N(X, 6r) - N(v_p, 6r)| \nonumber \\
	& \leq N(X, 6r)- N(X, r) + C(r^{\frac12} + \theta(24r)) \nonumber \\
	& = \widetilde{W}_X(r) + C(r^{\frac12} + \theta(24r)).\label{eq:Xinpb}
\end{align}
\underline{If $p\in D$ and $r\leq  r_{cs}(p)/6$}, the constant $c=N(p, 6r)$ by \eqref{def:c}. By Lemma \ref{lm:spvarin_close} we have
\begin{align}
	|N(X, \rho) - c| & \leq |N(X, \rho) - N(X, 6r)| + |N(X, 6r) - N(p, 6r)| \nonumber \\
	& \leq N(X, 6r)- N(X, r) + C(r^{\frac12} + \theta(24r)) \nonumber \\
	& = \widetilde{W}_X(r) + C(r^{\frac12} + \theta(24r)).\label{eq:Xpin_close}
\end{align}
On the other hand \underline{if $p\in D$ and $r>  r_{cs}(p)/6$}, by \eqref{def:c} and \eqref{def:Npr} the constant is fixed at $c = N(v_q, 6r)$, 
where $q\in \pD$ is such that $|p-q| = \dist(p, \pD)$. 
By Lemma \ref{lm:spvarin_close}, we have
\begin{align}
	|N(X, \rho) - c| & \leq |N(X, \rho) - N(X, 6r)| + |N(X, 6r) - N(v_q, 6r)| \nonumber \\
	& \leq \widetilde{W}_X(r) + C(r^{\frac12} + \theta(24r)).\label{eq:Xpin_far1}
\end{align}
Plugging back any one of \eqref{eq:Xpin_close}, \eqref{eq:Xinpb}, \eqref{eq:Xpin_far1} into \eqref{eq:Xin_close}, we obtain
\begin{align}
	& \frac{1}{r^{d}} \iint_{A_{2r,5r}(X) \cap D } \left|  \langle \nabla u(Z) , Z-X \rangle - cu(Z) \right|^2 ~dZ \nonumber \\
	& \lesssim \int_{2r}^{5r} \frac{1}{\rho^{d-2}} \int_{\partial B_\rho(X) \cap D} \left|\partial_{\rho_X} u - N(X, \rho) ~\frac{u}{\rho} \right|^2 \sH ~d\rho \nonumber \\
	& \qquad \qquad + \int_{2r}^{5r} \frac{|N(X, \rho) - c|^2}{\rho^d} \int_{\partial B_\rho(X) \cap D} u^2 \sH ~d\rho \nonumber \\
	& \lesssim \int_{2r}^{5r} R_h(X, \rho) \frac{H(X, \rho)}{\rho^{d-1}} ~d\rho + \int_{2r}^{5r} |N(X, \rho) - c|^2 \frac{H(X, \rho)}{\rho^{d}} ~d\rho \nonumber \\
	& \lesssim \widetilde{W}_X(r) + r + \theta(24 r).\label{eq:betacase2}
\end{align}

\underline{Case (iii): when $X\in D$ and $r>  r_{cs}(X)/6$}. Let $\tilde{X} \in \pD$ be such that $|X-\tilde{X}| = \dist(X, \pD)$. By the definition of the critical value, we know that
\begin{equation}\label{tmp:dXrcs}
	\dist(X, \pD) = r_{cs}(X) \tilde{\theta}(r_{cs}(X)) \ll r_{cs}(X) < \frac{15}{2} r. 
\end{equation} 
In particular since $\tilde{\theta}(6r) \leq \tilde{\theta}(\frac35 R) $ is chosen sufficiently small, we can guarantee
\[ |X-\tilde{X}| = \dist(X, \pD) < \frac{ r}{5}, \]
and thus 
\begin{equation}\label{eq:pXtilde}
	|p-\tilde{X}| \leq |p-X| + |X-\tilde{X}| < \frac65 r. 
\end{equation}  

This time we bound the integrand in \eqref{tmp:beta1} by
\begin{align}
	& \frac{1}{r^{d}} \iint_{A_{2r,5r}(X) \cap D } \left|  \langle \nabla u(Z) , Z-X \rangle - cu(Z) \right|^2 ~dZ \nonumber \\
	& \leq \frac{1}{r^{d}} \iint_{A_{r,6r}(\tilde{X}) \cap D } \left|  \langle \nabla u(Z) , Z-X \rangle - cu(Z) \right|^2 ~dZ \nonumber \\
	& \leq \frac{1}{r^{d}} \iint_{A_{r,6r}(\tilde{X}) \cap D } \left|  \langle \nabla u(Z) , Z- \tilde{X} \rangle - cu(Z) \right|^2 ~dZ + \frac{|X-\tilde{X}|^2}{r^d} \iint_{A_{r, 6r}(\tilde{X})} |\nabla u|^2 ~dZ \nonumber \\
	& =: \I_1 + \I_2.\label{eq:betaXin_far}
\end{align}
We can bound the error term $\I_2$ very easily as follows
\begin{align}
	\I_2 \leq  \frac{|X-\tilde{X}|^2 }{r^d} \iint_{B_{6r}(\tilde{X})} |\nabla u|^2 ~dZ & \leq C(\Lambda) \frac{|X-\tilde{X}|^2 }{r^{d+1}} \int_{\partial B_{6r}(\tilde{X})} u^2 \sH \nonumber \\
	& \lesssim C'(\Lambda) \frac{|X-\tilde{X}|^2 }{r^2} \nonumber \\
	& \lesssim \theta(24r),\label{eq:betaI2bd}
\end{align}
where we have used \eqref{eq:Hrhobd_Xb} in the second to last inequality, and we used \eqref{tmp:dXrcs}, the assumption $r_{cs}(X) < 6 r$ and $\tilde{\theta}(s) \leq \theta(4s)<1$ in the last inequality.
To bound the main term $\I_1$, we need to estimate $|N(v_{\tilde{X}}, \rho) - c|$ for $\rho \in [r, 6r]$. Again we divide into three cases.

\uline{If $p\in \pD$}, then the constant $c=N(v_p, 6r)$ by \eqref{def:c}. By \eqref{eq:spvar_rough} in Proposition \ref{prop:spvar}, we have
\begin{align}
	|N(v_{\tilde{X}}, \rho) - c| & \leq |N(v_{\tilde{X}}, \rho) - N(v_{\tilde{X}}, 6r)| + |N(v_{\tilde{X}}, 6r) - N(v_p, 6r)| \nonumber \\
	& \leq \widetilde{W}_{\tilde{X}}(r) + C(r^{\frac12} + \theta(24r)).\label{tmp:Ncg1}
\end{align}
\underline{If $p\in D$ and $r\leq  r_{cs}(p)/6$}, the constant $c=N(p, 6r)$ by \eqref{def:c}. By Lemma \ref{lm:spvarin_close}, we have
\begin{align}
	|N(v_{\tilde{X}}, \rho) - c| & \leq |N(v_{\tilde{X}}, \rho) - N(v_{\tilde{X}}, 6r) | + |N(v_{\tilde{X}}, 6r) - N(p, 6r)| \nonumber \\
	& \leq \widetilde{W}_{\tilde{X}}(r) + C(r^{\frac12} + \theta(24r)).\label{tmp:Ncg2}
\end{align}
\underline{If $p\in D$ and $r> r_{cs}(p)/6$}, we fix the constant $c$ to be $N(v_q, 6r)$,
where $q\in \pD$ is such that $|p-q| = \dist(p, \pD)$. Hence
\[ |q-p| = \dist(p, \pD) \ll r_{cs}(p) < 6r. \]
In particular by choosing sufficiently small radius we can guarantee that
\[ |q-\tilde{X}| \leq |q-p| + |p-\tilde{X}| < \frac45 r + \frac65 r= 2r. \]
Hence by \eqref{eq:spvar_rough} in Proposition \ref{prop:spvar}, we have
\begin{align}
	|N(v_{\tilde{X}}, \rho) - c| & \leq |N(v_{\tilde{X}}, \rho) - N(v_{\tilde{X}}, 6r) | + |N(v_{\tilde{X}}, 6r) - N(v_q, 6r)| \nonumber \\
	& \leq \widetilde{W}_{\tilde{X}}(r) + C(r^{\frac12} + \theta(24r)).\label{tmp:Ncg3}
\end{align}

Lastly, we claim that
\begin{equation}\label{eq:WXcg}
	\widetilde{W}_{\tilde{X}}(r) \leq \widetilde{W}_X(r) + \delta_{in},
\end{equation} 
with a constant only depending on $\Lambda$.
In fact, if $r> r_{cs}(X)$, by the definition \eqref{def:Npr} we have
\begin{equation}\label{tmp:c1}
	\widetilde{W}_{\tilde{X}}(r) = \widetilde{N}(v_{\tilde{X}}, 6r) - \widetilde{N}(v_{\tilde{X}}, r) = N_X(6r) - N_X(r) = W_X(r); 
\end{equation} 
if $r_{cs}(X)/6 < r \leq r_{cs}(X)$, by the definition \eqref{def:Npr} and Lemma \ref{lm:spvarin_far}
\begin{align}
	\widetilde{W}_{\tilde{X}}(r) & = N_{\tilde{X}}(6r) - N_{\tilde{X}}(r) \leq N_X(6r) - \left(N(X,r) - \delta_{in} \right) = \widetilde{W}_X(r) + \delta_{in}. \label{tmp:c2}
\end{align}

Similar to \eqref{eq:changeutov} in Case (i), we have
\begin{align}
	\I_1 & = \frac{1}{r^{d}} \iint_{A_{r,6r}(\tilde{X}) \cap D } \left|  \langle \nabla u(Z) , Z- \tilde{X} \rangle - cu(Z) \right|^2 ~dZ \nonumber \\
	& \leq \left(1+O\left(\theta(24r) \right) \right) \frac{1}{r^{d}} \iint_{A_{r,6r}(0) \cap \Omega_{\tilde{X}} } \left|  \langle \nabla_g v_{\tilde{X}}(Y), Y \rangle - cv_{\tilde{X}}(Y) \right|^2 ~dY \nonumber \\
	& \qquad \qquad + C \frac{\theta(24 r)}{r^{d-2}} \iint_{A_{r,6r}(0) \cap \Omega_{\tilde{X}}} |\nabla v_{\tilde{X}}(Y)|^2 ~dY \nonumber \\
	& \lesssim \frac{1}{r^{d}} \iint_{A_{r,6r}(0) \cap \Omega_{\tilde{X}} } \left|  \langle \nabla_g v_{\tilde{X}}(Y), Y \rangle - cv_{\tilde{X}}(Y) \right|^2 ~dY + \theta(24r).\label{eq:betaI1bd}
\end{align}
Moreover
\begin{align}
	& \frac{1}{r^{d}} \iint_{A_{r,6r}(0) \cap \Omega_{\tilde{X}} } \left|  \langle \nabla_g v_{\tilde{X}}(Y), Y \rangle - cv_{\tilde{X}}(Y) \right|^2 ~dY \nonumber \\
	& \lesssim \int_{r}^{6r} \frac{1}{\rho^{d-2}} \int_{\partial B_\rho \cap \Omega_{\tilde{X}} } \mu \left|\partial_{\rho} v_{\tilde{X}} - N(v_{\tilde{X}}, \rho) ~\frac{v_{\tilde{X}}}{\rho} \right|^2 ~dV_{\partial B_\rho} ~d\rho \nonumber \\
	& \qquad \qquad + \int_{r}^{6r} \frac{|N(v_{\tilde{X}}, \rho) - c|^2}{\rho^d} \int_{\partial B_\rho \cap \Omega_{\tilde{X}} } \mu v_{\tilde{X}}^2 ~dV_{\partial B_\rho} ~d\rho \nonumber \\
	& \lesssim \int_{r}^{6r} R_h(v_{\tilde{X}}, \rho) \frac{H(v_{\tilde{X}}, \rho)}{\rho^{d-1}} ~d\rho + \int_{r}^{6r} |N(v_{\tilde{X}}, \rho) - c|^2 \frac{H(v_{\tilde{X}}, \rho)}{\rho^{d}} ~d\rho \nonumber \\
	& \lesssim \widetilde{W}_{\tilde{X}}(r)  + r + \theta(24 r) \nonumber \\
	& \lesssim \widetilde{W}_X(r) + r + \theta(24r) + \delta_{in} \cdot \chi_{ r_{cs}(X)/6 < r \leq r_{cs}(X) }.\label{eq:betaI1bd_}
\end{align}
where we used \eqref{eq:Hrhobd_Xb} in the penultimate inequality and \eqref{eq:WXcg} in the last inequality.
Here we include the characteristic function to make sure when we sum them up in dyadic scales of $r$, the small constant $\delta_{in}$ appears at most finitely many times.
Finally, combining \eqref{eq:betaXin_far}, \eqref{eq:betaI2bd}, \eqref{eq:betaI1bd} and \eqref{eq:betaI1bd_}, we obtain
\begin{align}
	& \frac{1}{r^{d}} \iint_{A_{2r,5r}(X) \cap D } \left|  \langle \nabla u(Z) , Z-X \rangle - cu(Z) \right|^2 ~dZ \nonumber \\
	& \qquad \lesssim \widetilde{W}_X(r) + r + \theta(24r) + \delta_{in} \cdot \chi_{ r_{cs}(X)/6 < r \leq r_{cs}(X) }.\label{eq:betacase3}
\end{align}

Plugging \eqref{eq:betacase2} (in case (ii)) or \eqref{eq:betacase3} (in case (iii)) back into \eqref{tmp:beta1}, we get
\begin{align}
	& \frac{\lambda_j}{r^{d-2}} \iint_{A_{3r, 4r}(p) \cap D} \left| \langle \nabla u(Z), w_j \rangle \right|^2 ~dZ \nonumber \\
	& \leq \frac{1}{r^{k}} \int_{B_r(p)} \frac{1}{r^{d}} \iint_{A_{2r,5r}(X) \cap D } \left|  \langle \nabla u(Z) , Z-X \rangle - c u(Z) \right|^2 ~dZ ~d\mu(X) \nonumber \\
	& \lesssim \frac{1}{r^k} \int_{B_r(p)} \left( \widetilde{W}_X(r) + \delta_{in} \cdot \chi_{ r_{cs}(X)/6 < r \leq r_{cs}(X) } \right) ~d\mu(X) + \frac{\mu(B_r(p))}{r^k} \left( r + \theta(24 r) \right).\label{eq:betacase23}
\end{align}
\medskip

\uline{Step 3}.
The inequalities \eqref{eq:betaintm} and \eqref{eq:betacase23} are helpful in our estimate of the beta number, because of the expression \eqref{eq:betaabsev} and the following lemma.
\begin{lemma}\label{lm:sym}
	Let $R, \Lambda>0$ be fixed. There exist $C_s>0$ and $r_s>0$ such that the following holds for all $0< r\leq r_s$. For any $(u, D) \in \HH(R, \Lambda)$ and any $p\in \Ct_{r}(u)\cap \mathcal{N}(u) \cap B_{\frac{R}{10}}(0)$ satisfying $\dist(p, \pD) \leq \frac35 R \tilde{\theta}(\frac35 R)$, assume $u$ is normalized so that
	\begin{equation}\label{as:nmlz}
		\frac{1}{r^d} \iint_{B_r(p)} u^2 ~dZ = 1.
	\end{equation}
	Then we have
	\[ \frac{1}{r^{d-2}} \iint_{A_{3r, 4r}(p) \cap D} \sum_{j=1}^{d-1} \left| \langle \nabla u(Z), w_j \rangle \right|^2 ~dZ \geq C_s \]
	 for any $(d-1)$-orthonormal vectors $\{w_1, \cdots, w_{d-1}\}$.
\end{lemma}
\begin{proof}
	Assume the statement is false. Then there exist sequences $\delta_i, r_i \to 0$, $(u_i, D_i) \in \HH(R, \Lambda)$, $p_i \in \Ct_{r_i}(u_i) \cap \mathcal{N}(u_i) \cap \overline{D}_i$ and $(d-1)$-orthonormal vectors $\{\omega_i^1, \cdots, \omega_i^{d-1} \}$, such that
	\begin{equation}\label{cl:sym}
		\frac{1}{r_i^{d-2}} \iint_{A_{3r_i, 4r_i}(p_i) \cap D_i} \sum_{j=1}^{d-1} |\langle \nabla u_i(Z), \omega_i^j \rangle |^2 ~dZ < \delta_i. 
	\end{equation} 
	We consider the rescalings $T_{p_i, r_i} u_i$ of $u_i$'s, see Definition \ref{def:Tpr}. Then it follows from \eqref{cl:sym} and the normalization \eqref{as:nmlz} that
		\begin{equation}\label{cl:rescale} 
			\iint_{A_{3, 4}(0) \cap \frac{D_i-p_i}{r_i}} \sum_{j=1}^{d-1} |\langle \nabla T_{p_i, r_i} u_i(Z), w_i^j \rangle |^2 ~dZ < \delta_i. 
		\end{equation}
	We claim that modulo passing to a subsequence
   \[ T_{p_i, r_i} u_i \text{ converges uniformly and in } W^{1,2} \text{ to a harmonic function } u_\infty \text{ in } D_\infty \cap B_5(0). \] 
   In fact, if infinitely many $p_i$'s are boundary points (i.e. $p_i \in \pD_i$) we appeal to Proposition \ref{prop:cpt}, and $D_\infty$ is an (possibly tilted) upper half-space.
   If infinitely many $p_i$'s are interior points, then compactness follows from the estimate in Lemma \ref{lm:urunifgrad} and the assumption 
   \[ T_{p_i, r_i} u_i(0) = u(p_i)=0. \]
   Moreover, the limit domain $D_\infty$ is an upper half-space with $0\in \pD_\infty$, an upper half-space with $0\in D_\infty$, or the entire space $\RR^d$, depending on
   \[ \frac{\dist(p_i, \pD_i)}{r_i} \text{ converges to } 0, \text{ a positive number, or }\infty, \text{ respectively.} \]
	On the other hand, since for each $i$ the vectors $\{w_i^1, \cdots, w_i^{d-1}\}$ form a $(d-1)$-orthonormal basis, modulo passing to a subsequence, they converge to a $(d-1)$-tuple of orthonormal vectors, denoted by $\{w_1, \cdots, w_{d-1}\}$.
	
	 Therefore passing the inequality \eqref{cl:rescale} to the limit, we get
	\begin{equation}
		\iint_{A_{3,4}(0) \cap D_\infty} \sum_{j=1}^{d-1} |\langle \nabla u_\infty(Z), w_j \rangle|^2 ~dZ = 0.
	\end{equation}
	This implies that $\langle \nabla u(Z), w_j \rangle \equiv 0$ in the annulus $A_{3, 4}(0)$, for every $j=1, \cdots, d-1$. Since $\langle \nabla u(Z), w_j\rangle$ is also a harmonic function in $D_\infty \cap A_5(0)$, by unique continuation theorem it vanishes everywhere. In other words $u_\infty$ is invariant in the $w_j$-direction, for every $j=1, \cdots, d-1$. Since $u_\infty$ is a harmonic function in $D_\infty$ with $\iint_{B_1(0)} u_\infty^2 ~dZ = 1$, it follows that $u_\infty$ is a (non-trivial) linear function. In particular $u_\infty$ is homogeneous with respect to the origin of degree $1$. The second case where $D_\infty$ is an upper half-space with $0\in D_\infty$ is impossible. Indeed $u_\infty$ vanishes on $B_5(0) \cap \partial D_\infty$ and $u_\infty(0) = 0$ at an interior point $0$ imply that $u_\infty$ must be a trivial function. In the first case where $D_\infty$ is an upper half-space with $0\in \pD_\infty$, by the same computation as in the proof of \eqref{eq:tn} we have $|\nabla u_\infty| \equiv \alpha_d^1$. In the third case where $D_\infty = \RR^d$, by $u_\infty(0) = 0$ and symmetry we similarly get $|\nabla u_\infty| = \frac{ \alpha_d^1}{\sqrt{2}}$.
	
	On the other hand, recall that $p_i \in \Ct_{r_i}(u_i)$ means that
	\[ \inf_{B_\beta(0)} |\nabla T_{p_i, r_i} u_i| \leq \alpha_0. \]
	By choosing $\alpha_0$ sufficiently small so that $\alpha_0 < \frac{\alpha_d^1}{\sqrt{2}} < \alpha_d^1$, we get a contradiction.
	
	
\end{proof}

Finally, combining \eqref{tmp:beta1}, Lemma \ref{lm:sym} and (depending on whether $X\in \pD$ or $X\in D$) \eqref{eq:betaintm} or \eqref{eq:betacase23}, we obtain
\begin{align*}
	\left|\beta_\mu^{d-2}(p, r) \right|^2 & = \lambda_{d-1} + \lambda_{d} \\
	& \leq \frac{2}{d-1} \frac{1}{C_s} \sum_{j=1}^{d-1} \frac{\lambda_j }{r^{d-2}} \iint_{A_{3r, 4r}(p) \cap D} |\langle \nabla u(Z), w_j \rangle |^2 ~dZ \\
	& \lesssim \frac{1}{r^{d-2}} \int_{B_r(p)} \left( \widetilde{W}_X(r) + \delta_{in} \cdot \chi_{ r_{cs}(X)/6 < r \leq \frac{1}{6} r_{cs}(X) } \right) ~d\mu(X) + \frac{\mu(B_r(p))}{r^{d-2}} \left( r + \theta(24 r) \right)
\end{align*}
which finishes the proof of Theorem \ref{thm:L2appx}.

\section{Covering of the quantitative stratum}\label{sec:covering}
\begin{lemma}\label{lm:covering}
	Let $R, \Lambda >0 $ and $ \epsilon, \rho\in (0, 1/2)$ be fixed. There exist $\delta, r_{c}>0$ such that the following holds. 
	For any $(u, D) \in \mathfrak{H}(R, \Lambda) $, any radii $r_0, r_*$ satisfying $0<r_0<r_* \leq r_{c}$, and any subset $\St \subset \widetilde{\Ct}_{r_0}(u) \cap \mathcal{N}(u)\cap B_{r_{c}\tilde{\theta}(r_{c})}(\pD)$, let $\Lambda_* := \sup_{X\in \St \cap B_{2r_*}(0)} N_X(r_*)$.\footnote{By Lemmas \ref{lm:freqbd} and \ref{lm:freqbd_in} we know the bound $\Lambda_* \leq C(\Lambda)$. } There exists a finite covering of $\St \cap B_{r_*}(0)$ such that
	\[
		\St \cap B_{r_*}(0) \subset \bigcup_{X\in \mathscr{C}} B_{r_X}(X), \quad \text{ with } r_X \geq r_0; 
	\] 
	\begin{equation}\label{eq:disjoint}
		B_{r_X/5}(X) \cap B_{r_{X'}/5}(X') = \emptyset, \quad \text{ for every distinct pair } X, X' \in \CC;
	\end{equation}
	and 
	\begin{equation}\label{eq:pkest}
		\sum_{X\in \mathscr{C}} r_X^{d-2} \leq C_1 C_V \cdot r_*^{d-2}.
	\end{equation}
	Here $C_V>0 $ is a dimensional constant, and $C_1>0$ depends on the dimension as well as $\rho$.
	Moreover, for each center point $X\in \mathscr{C}$, one of the following alternatives holds:
	\begin{enumerate}[label=(\roman*)]
		\item \label{item:terminal} \underline{terminal ball:} $r_X= r_0$;
		\item \label{item:smalldim} \underline{small dimension}: $r_X>r_0$ and moreover, the set of points 
			\[ F_X = \{Y\in \St \cap B_{2r_X}(X): N_Y(\rho r_X/10) \geq \Lambda_* - \delta \} \]
			is contained in $B_{\rho r_X/5}(V_X)$, where $V_X$ is some $(d-3)$-dimensional affine subspace;
		\item \label{item:deffreqdrop} \underline{definite frequency drop}:
			\begin{equation}\label{eq:deffreqdrop}
				N_Y(r_X/10) < \Lambda_* -\delta, \quad \text{ for any } Y \in \St \cap B_{2r_X}(X).
			\end{equation}
	\end{enumerate}
\end{lemma}
\begin{remark}\label{rmk:steprho}
	For convenience and without loss of generality, we assume that both $r_0, r_*$ are some integer power of $\rho$, i.e.
	\[ r_0 = \rho^{j_0}, \quad r_* = \rho^{j_*}, \quad \text{ where } j_0 > j_* \text{ are positive integers}. \]
\end{remark}
\begin{proof}
	We choose the parameters in the following order:
	\begin{itemize}
		\item Let $\delta_0>0$ be fixed whose value is to be determined later in \eqref{def:delta0}.
		\item Let $\delta, r_{tn}, \beta$ be the values determined in Proposition \ref{prop:tn} with given values of $R, \Lambda$ and $\delta_0, \rho/10, \rho/10$ (that is, we take $\rho/10$ in place of $\rho$ and $\tau$ in the statement).
		\item Let $\delta_{in} \in (0, \delta_0/2]$ be fixed whose value is to be determined later in \eqref{def:deltain}, and $r_{in}$ as given in Lemma \ref{lm:spvarin_far} with parameters $\rho/10$ and $\delta_{in}$.
		\item Let $r_b= r_b(\epsilon, \delta_{in})$ be the radius determined in Theorem \ref{thm:L2appx}.
		\item Finally, let 
	\begin{equation}\label{def:rct}
		\tilde{r}_c  := \min\left\{r_{tn}, r_{in}, r_b, \frac{1}{80}R \right\},
	\end{equation}
	and we will determine the value of $r_c \leq \tilde{r}_c$ later in \eqref{def:rc}.
	\end{itemize}
	By Lemmas \ref{lm:freqbd} and \ref{lm:freqbd_in}, the assumption $2\tilde{r}_c \leq \frac{R}{40}$ guarantees the uniform upper bound
	\[ N_{X_0}(\tilde{r}_c) \leq C(\Lambda, R), \quad \text{ for any } X_0 \in B_{2\tilde{r}_c}(0) \cap \pD. \]
	
	\textbf{Construction of the covering}.
	To start let
	\[ F_0 := \{X\in \St \cap B_{2r_*}(0): N_X(\rho r_*/10) \geq \Lambda_* -\delta\}. \]
	The case when $F_0 = \emptyset$ is trivial: We just cover $\St \cap B_{r_*}(0)$ by a finite number of balls centered at $X\in \St \cap B_{r_*}(0)$ and with radius $r_X= \rho r_*$, and label the collection of centers as $\CC_{\varnothing}$. Clearly
	\begin{equation}\label{eq:countep1}
		\sum_{X\in \CC_{\varnothing}} r_X^{d-2} = (\rho r_*)^{d-2} \cdot \#\CC_{\varnothing} \leq 10^d\rho^{-2} r_*^{d-2} =: C_1 r_*^{d-2}; 
	\end{equation} 
	and for any $Y\in \St \cap B_{2\rho r_*}(X) \subset \St \cap B_{2r_*}(0)$,
	\[ N_Y(r_X/10) = N_Y(\rho r_*/10) < \Lambda_* - \delta.  \]
	
	Now we assume $F_0 \neq \emptyset$. If $F_0$ is contained in $B_{\rho r_*/5}(V_0) \cap B_{2r_*}(0)$, where $V_0$ is some $(d-3)$-dimensional affine subspace, then the trivial cover $B_{r_*}(0)$ satisfies alternative \ref{item:smalldim}. We say the center $0\in \CC_{sd}$, where the subindex is to indicate the \textit{small dimension} alternative. If not, then $F_0$ 
	$\rho r_*/5$-effectively span a (at least) $(d-2)$-dimensional affine subspace (denoted by $L_0$) in $B_{2r_*}(0)$. Then by Proposition \ref{prop:tn} as well as the inclusion $\widetilde{\Ct}_{r_0}(u) \subset \Ct_{r}(u)$ for any $r_0 \leq r\leq r_c$, we get
	\begin{equation}\label{eq:tnstep0}
		\St \cap B_{2r_*}(0) \subset B_{2\beta r_* }(L_0),  
	\end{equation} 
	and 
	\begin{equation}\label{eq:freqdropstep0}
		N_X(\rho r_*/10) \geq \Lambda_* - \delta_0
		\quad \text{ for every } X\in B_{2\beta r_*}(L_0) \cap B_{2r_*}(0) \cap \overline{D}. 
	\end{equation}
	We cover $\St \cap B_{r_*}(0)$ by a collection of balls centered at $X\in \St \cap B_{r_*}(0)$ with radius $r_X^{1} = \rho r_*$ and label the collection of centers $ \CC_g^{(1)}$:
	\begin{equation}\label{eq:gc0}
		\St \cap B_{r_*}(0) \subset \bigcup_{X\in \CC_g^{(1)}} B_{\rho r_*}(X), \quad \text{ with } X\in \St \cap B_{r_*}(0); 
	\end{equation} 
	and they satisfy
	\[ B_{\rho r_*/5}(X) \cap B_{\rho r_*/5}(X') = \emptyset, \quad \text{ for any distinct } X, X'\in \mathscr{C}_g^{(1)}.  \] 
	By \eqref{eq:tnstep0} and \eqref{eq:freqdropstep0} each $X\in \CC_g^{(1)}$ has small frequency drop, in the sense that
	\[ N_X(\rho r_*/10) \geq \Lambda_* - \delta_0. \]
	Because of this we call these balls the \textit{good balls} of level $1$.

	Now we state the construction inductively. Let $i$ be an integer with $0\leq i < j_0 - j_*$ (see Remark \ref{rmk:steprho}). Let $X\in \CC_g^{(i)}$ be fixed, and we consider the \textit{bad ball} $B_{r_X^i}(X)$ with $r_X^i = \rho^i r_*> r_0$ and the set
	\[ F_X^{(i)}:= \{Y\in \St \cap B_{2r_X^i}(X): N_Y(\rho r_X^i/10) \geq \Lambda_* - \delta \}. \]
	(We use the convention that $\CC_g^{(0)} = \{0\}$ and thus $B_{r_*}(0)$ is the ball in the base case.) Our goal is to cover $\St \cap B_{r_X^i}(X)$. 
	\begin{itemize}
		\item $\CC_{\varnothing}$: If the ball $B_{r_X^i}(X)$ satisfies that $F_X^{(i)} = \emptyset$, then we cover $\St \cap B_{r_X^i}(X)$ trivially by finitely many balls $\{B_{r_Y}(Y)\}$ centered at $Y \in \St \cap B_{r_X^i}(X)$ with radius $r_Y = \rho r_X^i = \rho^{i+1} r_*$, and label the collection of centers $\CC_{\varnothing}^{(i+1)}(X)$. These balls satisfy the \textit{definite frequency drop} alternative \ref{item:deffreqdrop}. Moreover, similarly to \eqref{eq:countep1} we have
			\begin{equation}
				\sum_{Y\in \CC_{\varnothing}^{(i+1)}(X)} (r_Y)^{d-2} \leq C_1 (r_X^i)^{d-2}.
			\end{equation}
		\item $\CC_{sd}$: If the ball $B_{r_X^i}(X)$ is such that $F_X^{(i)}$ satisfies the \textit{small dimension} alternative (i.e. $F_X^{(i)}$ is contained in a $\rho r_X^i/5$-tubular neighborhood of a $(d-3)$-dimensional affine subspace), then we simply cover $\St \cap B_{r_X^i}(X)$ by $B_{r_X^i}(X)$. In this case we say 
			\[ X\in \CC_{sd}^{(i+1)} \quad \text{ and } \quad r_X = r_X^i = \rho^i r_*. \]
		\item $\CC_g$: Alternatively, the non-empty set $F_X^{(i)}$ $\rho r_X^i/5$-effectively span a $(d-2)$-dimennsional affine subspace, denoted by $L_X$. Then as in \eqref{eq:gc0} we cover $\St \cap B_{r_X^i}(X)$ by a collection of balls centered at $Y\in \CC_g^{(i+1)}(X)$ with radius $r_Y^{i+1} = \rho r_X^i = \rho^{i+1} r_*$:
			\begin{equation}\label{eq:gci}
				\St \cap B_{r_X^i}(X) \subset \bigcup_{Y \in \CC_g^{(i+1)}(X)} B_{\rho r_X^i}(Y), \quad \text{ with } Y\in \St \cap B_{r_X^i}(X); 
			\end{equation} 
	and they satisfy
	\[ B_{r_Y^{i+1}}(Y) \cap B_{r_{Y'}^{i+1}}(Y)  = \emptyset, \quad \text{ for any distinct } Y, Y'\in \mathscr{C}_g^{(i+1)}(X).  \] 
	Moreover, each $Y\in \CC_g^{(i+1)}(X)$ has small frequency drop, in the sense that
	\[ N_X(\rho r_X^i/10) \geq \Lambda_* - \delta_0. \]	
%
		If $r_Y^{i+1} = r_0$, we call these balls the \textit{terminal balls} and denote $Y\in \CC_0 $ and $ r_Y = r_Y^{i+1} = r_0$. Otherwise, for any $Y\in \CC_g^{(i+1)}(X)$ we call $B_{r_Y^{i+1}}(Y)$ a \textit{good ball} of level $(i+1)$. We repeat this argument and keep subdividing the \textit{good balls}. 
	\end{itemize}
	
	By the above construction, after at most $(j_0-j_*)$ levels all the remaining balls $B_{r_X}(X)$ fall into one of the three categories:
	\begin{itemize}
		\item $X\in \CC_{\varnothing}^{(i)}$ for some $i \in \{1, \cdots, j_0 - j_*\} $, $r_X = \rho^i r_*$ and $B_{r_X}(X)$ satisfies \ref{item:deffreqdrop}.
		\item $X\in \CC_{sd}^{(i)}$ for some $i\in \{1, \cdots, j_0 - j_*\}$, $r_X =\rho^{i-1} r_*$ and $B_{r_X}(X)$ satisfies \ref{item:smalldim}. 
		\item $X\in \CC_0$ with $r_X = r_0$, and thus $B_{r_X}(X)$ satisfies \ref{item:terminal}.
	\end{itemize}
	At each level, by considering all the coverings as a whole, we can guarantee they are almost disjoint (as in \eqref{eq:disjoint}); by taking $\rho < 2/3$ and removing redundant new balls if necessary, we can guarantee they are also almost disjoint from all the previous balls with centers in $\CC_{\varnothing}$ or $\CC_{sd}$.
	It remains to prove the packing estimate \eqref{eq:pkest}.
	We denote
	\[ \CC_{\varnothing} = \bigcup_{i=1}^{j_0-j_*} \CC_{\varnothing}^{(i)}, \quad \CC_{sd} = \bigcup_{i=1}^{j_0-j_*} \CC_{sd}^{(i)}, \quad \CC = \CC_{\varnothing} \bigcup \CC_{sd} \bigcup \CC_0,  \]
	and we denote the intermediate balls
	\[ \CC_g = \bigcup_{i=1}^{j_0 - j_* -1} \CC_g^{(i)}. \]
	Moreover, for every point $Y \in \CC_{\varnothing}$ or $Y \in \CC_{sd}$, there must exist a precursor $X\in \CC_{g}^{(i-1)}$ with $i-1 = 0, \cdots, j_0-j_*$, such that $Y \in \CC_{\varnothing}^{(i)}(X)$ or $Y\in \CC_{sd}^{(i)}(X)$, respectively. In that case we denote $\tilde{r}_X = r_X^i>r_0$. We let $\CC_{pre}$ be the collection of all such $X$'s and let $\CC' := \CC_{pre} \cup \CC_0$. We denote $\tilde{r}_X = r_0$ if $X\in \CC_0$. In particular we have the sets $B_{\tilde{r}_X/5}(X)$ with $X\in \CC'$ are pairwise disjoint. Moreover, every $X\in \CC'$ satisfies
	\begin{equation}\label{eq:sfd}
		N_X(\tilde{r}_X/10 ) \geq \Lambda_* - \delta_0.
	\end{equation} 
	
	By the above construction, we have
	\begin{align*}
		\sum_{X\in \CC} (r_X)^{d-2} & \leq \sum_{X\in \CC_{\varnothing}} (r_X)^{d-2} + \sum_{X\in \CC_{sd}} (r_X)^{d-2} + \sum_{X\in \CC_0} (r_X)^{d-2} \\
		& \leq \max\{C_1, 1\} \sum_{X\in \CC_{pre}} (\tilde{r}_X)^{d-2} + \sum_{X\in \CC_0} (\tilde{r}_X)^{d-2} \\
		& \leq C_1 \sum_{X\in \CC'} (\tilde{r}_X)^{d-2}.
	\end{align*}
	Therefore it suffices to show 
	\begin{equation}\label{eq:pkestp}
		\sum_{X\in \CC'} (\tilde{r}_X)^{d-2} \leq C_V \cdot (r_*)^{d-2}
	\end{equation}
	for some universal constant $C_V$.
	This way, we reduce the packing estimates to center points with small frequency drop, which are good points for our purpose. 

	\textbf{Packing estimate}.
	 We first make the following remark: for any $X\in \CC_g^{(i)}$, we have the rough estimate
	\begin{equation}\label{est:crude}
		\sum_{Y\in \CC_g^{(i+1)}} (r_Y^{i+1})^{d-2} = (\rho r_X^i)^{d-2} \cdot \# \CC_g^{(i+1)} \leq (\rho r_X^i)^{d-2} \cdot \left(\frac{10}{\rho} \right)^{d-2} = C_0 (r_X^i)^{d-2}.
	\end{equation}
	Using this estimate we get
	\[ \sum_{X\in \CC} (r_X)^{d-2} = \sum_{X\in \CC_{\varnothing}} (r_X)^{d-2} + \sum_{X\in \CC_{sd}} (r_X)^{d-2} + \sum_{X\in \CC_0} (r_X)^{d-2} \lesssim C_0^{j_0-j_*} (r_*)^{d-2},  \]
	but the constant $C_0^{j_0-j_*} $ grows polynomially with the ratio $r_*/r_0$.
	
	Instead, we prove \eqref{eq:pkestp} inductively. Let $ \CC'_{t} =\{ X\in \CC': \tilde{r}_X \leq t\}$, where the relevant cases for us are $t=r_0, r_0/\rho, \cdots, r_*/6$.
	Here we denote $\tilde{r}_X = r_0$ if $X\in \CC_0$. We also let 
	\[ \mu = \omega_{d-2} \sum_{X\in \CC'} (\tilde{r}_X)^{d-2} \delta_{X}\quad  \text{ and } \quad \mu_t = \omega_{d-2} \sum_{X\in \CC'_t} (\tilde{r}_X)^{d-2} \delta_{X} \leq \mu_{r_*} = \mu . \]
	We will prove the following estimate inductively on the level of $t= r_0, r_0/\rho, \cdots, r_*/6$\footnote{If $r_*/6 \neq r_0/\rho^{i}$ for any integer $i\in \mathbb{N}$, we just replace it by $r_0/\rho^{i_*}$ where $i_*$ is the largest integer such $r_0/\rho^i \leq r_*/6$.}: there exists a universal constant 
	\[ C'_V = \max\{11^d \omega_{d-2}, 2^{d-2} C_{dr} \omega_{d-2}\}>0 \] such that
	\[ \mu_t(B_t(X)) \leq C'_V \cdot t^{d-2}, \quad \text{ for every } X\in B_{r_*}(0). \]
	Here $C_{dr}$ is the dimensional constant in the Discrete Reifenberg Theorem \ref{thm:Reifenberg}.
	It is easy to see the above statement follows from
	\begin{equation}\label{claim:pkest}
		\mu_t(B_{2t}(X)) \leq C'_V \cdot t^{d-2}, \quad \text{ for every } X\in B_{r_*}(0) \cap \spt \mu_t. 
	\end{equation} 
	
	We start from the base case $t=r_0$. In this case $\CC'_{r_0}$ consists entirely of points in $\CC_0$, and they satisfy
	\[ B_{r_0/5 }(Y) \cap B_{r_0/5}(Y') = \emptyset \quad \text{ for every distinct } Y, Y' \in \CC_0. \]
	Hence
	\begin{align*}
		\mu_{r_0}(B_{2r_0}(X)) & \leq \omega_{d-2}(r_0)^{d-2} \cdot \# \text{ of } \CC_0 \text{ inside }B_{2r_0}(X) \\
		& \leq 11^d \omega_{d-2} (r_0)^{d-2}.
	\end{align*}
	
	Now assuming we have proven the claim \eqref{claim:pkest} for $t=r_0, \cdots, r_0/\rho^i$, and we shall prove it for $t=r_0/\rho^{i+1}$. Let $\bar t = r_0/\rho^{i+1}$. We can split the measure as follows
	\[ \mu_{\bar t} = \mu_{\rho \bar t} + \tilde{\mu}_{\bar t} := \omega_{d-2} \sum_{X\in \CC'_{\rho \bar t}} (\tilde{r}_X)^{d-2} \delta_X + \omega_{d-2} \sum_{X\in \CC': \rho \bar t < \tilde{r}_X \leq \bar t} (\tilde{r}_X)^{d-2} \delta_X.  \]
	We first have a rough estimate
	\begin{align}
		\mu_{\bar t}(B_{2\bar t}(X)) & = \mu_{\rho \bar t}(B_{2\bar t}(X)) + \tilde{\mu}_{\bar t} (B_{2\bar t}(X)) \nonumber \\
		& \lesssim \rho^{-d} \cdot C'_V \cdot (\rho\bar t)^{d-2} + \rho^{-d} \cdot (\bar t)^{d-2} \nonumber \\
		& = C_2  C'_V \cdot (\bar t)^{d-2},\label{eq:roughest}
	\end{align}
	where the constant $C_2$ depends on the dimension and the value of $\rho$.
	The same rough estimate in fact holds for all $\rho \bar t < t \leq \bar t$.
	
	Next we use the Discrete Reifenberg Theorem \ref{thm:Reifenberg} to give the desired bound in \eqref{claim:pkest}.
	To simplify the notation we let 
	\[ \bar{\mu} : = \mu_{\bar t} \res B_{2\bar t}(X), \quad \text{ where } X\in \spt\mu_{\bar t} \text{ is fixed}. \]
	Now let $q\in \spt \bar\mu$ and $0<s\leq 2\bar t$ be arbitrary, and let $p\in B_s(q) \cap \spt \bar{\mu}$ and $0<r\leq s$. Since 
	\[ \spt \bar\mu \subset \CC' \subset \St \subset \mathcal{N}(u),\footnote{Except perhapst the top level, the origin.} \] by Theorem \ref{thm:L2appx} we can bound the $\beta$-number as follows
	\begin{align*}
		\left|\beta_{\bar\mu}^{d-2}(p,r) \right|^2 & \leq C_b \left[ \frac{1}{r^{d-2}} \int_{B_r(p)} \left( \widetilde{W}_Y(r) + \delta_{in} \cdot \chi_{r_{cs}(Y)/6 < r \leq \frac16 r_{cs}(Y) } \right) ~d\bar{\mu}(Y) \right. \\
		& \qquad \qquad + \left. \frac{\bar\mu(B_r(p)}{r^{d-2}} \left( r + \theta(24r) \right) \right] \\
		& \leq C_b \left[ \frac{1}{r^{d-2}} \int_{B_r(p)} \left( \widetilde{W}_Y(r) + \delta_{in} \cdot \chi_{ r_{cs}(Y)/6 < r \leq \frac16 r_{cs}(Y) } \right) ~d\bar{\mu}(Y) \right. \\
		& \qquad \qquad + \left. C_2 C'_V \cdot \left(  r + \theta(24r) \right) \right].
	\end{align*} 
	Here we also use the rough estimate \eqref{eq:roughest} in the second inequality. 
	Therefore
	\begin{align*}
		 & \int_{B_s(q)} \int_0^s \left|\beta_{\bar\mu}^{d-2}(p,r) \right|^2 \frac{dr}{r} ~d\bar\mu(p) \\
		 & = \int_{B_s(q)} \int_{\tilde{r}_p/5}^s \left|\beta_{\bar\mu}^{d-2}(p,r) \right|^2 \frac{dr}{r} ~d\bar\mu(p) \\
		& \lesssim \int_{B_s(q)} \int_{\tilde{r}_p/5}^s \frac{1}{r^{d-2}} \int_{B_r(p)} \widetilde{W}_Y(r) ~d\bar{\mu}(Y) \frac{dr}{r} ~d\bar\mu(p) 
		\\
		& \qquad + \int_{B_s(q)} \int_0^s  \theta(24r) \frac{dr}{r} ~d\bar\mu(p) + s~\bar\mu(B_s(q)) + \delta_{in}~ \bar\mu(B_{s}(q)) \\
		& =: \I_1 + \I_2 + \I_3 + \I_4.
	\end{align*}
	By Fubini theorem, we get
	\begin{align*}
		\I_1 \lesssim \int_{B_{2s}(q)} \int_0^s \left( \frac{1}{r^{d-2}} \int_{B_r(Y)}\chi_{\tilde{r}_p/5\leq r } ~d\bar\mu(p)\right) \widetilde{W}_Y(r) \frac{dr}{r}  ~d\bar\mu(Y).
	\end{align*}
	Let $Y\in B_{2s}(q) \cap \spt\bar\mu$ be fixed. If the radius $r< \tilde{r}_Y/5$, then by almost-disjointness (as in \eqref{eq:disjoint}) we have
	\[ B_r(Y) \cap \spt\bar\mu =\{Y\}, \]
	and thus
	\[ r\geq \tilde{r}_p/5 = \tilde{r}_Y/5. \]
	This contradicts the assumption $r< \tilde{r}_Y/5$. Therefore we always have $r \geq \tilde{r}_Y/5$. It follows that
	\begin{align*}
		\I_1 & \lesssim \int_{B_{2s}(q)} \int_{\tilde{r}_Y/5}^s \left( \frac{1}{r^{d-2}} \int_{B_r(Y)} ~d\bar\mu(p)\right) \widetilde{W}_Y(r) \frac{dr}{r}  ~d\bar\mu(Y) \\
		& \lesssim\int_{B_{2s}(q)} \int_{\tilde{r}_Y/5}^s \left( \frac{\bar\mu(B_r(Y))}{r^{d-2}} \right) \widetilde{W}_Y(r) \frac{dr}{r}  ~d\bar\mu(Y) \\
		& \lesssim \int_{B_{2s}(q)} \int_{\tilde{r}_Y/5}^s \widetilde{W}_Y(r) \frac{dr}{r}  ~d\bar\mu(Y) \quad \text{ by the rough estimate } \eqref{eq:roughest} \\
		& \lesssim \int_{B_{2s}(q)} N_Y(6s) - N_Y(\tilde{r}_Y/10) ~d\bar\mu(Y) \quad \text{ by telescoping on } \widetilde{W}_Y(r) \\
		& \lesssim \delta_0 (2s)^{d-2} \quad \text{by } 6s\leq r_*, \eqref{eq:sfd}, \eqref{eq:roughest} \text{ and a simple covering} \\
		& =: C_4 \delta_0 s^{d-2}.
	\end{align*}
	Besides we have
	\begin{align*}
		\I_2 \lesssim \left(\int_0^{24s} \frac{\theta(r)}{r} ~dr \right) \bar\mu(B_s(q)) =: C_5 \left( \int_0^{24s}\frac{\theta(r)}{r} ~dr \right) s^{d-2}.
	\end{align*}
	Clearly by the rough estimate \eqref{eq:roughest} and $s\leq r_*/6$, we have
	\[ \I_3 \lesssim r_* s^{d-2} =: C_6 r_* s^{d-2}, \quad \I_4 \lesssim \delta_{in} s^{d-2} =: C_7 \delta_{in} s^{d-2}.  \]
	
	Let $\eta_{dr}$ be the constant in Theorem \ref{thm:Reifenberg}. Let
	\begin{equation}\label{def:tauc}
		\tau_c := \sup\left\{\tau>0: \int_0^{24\tau } \frac{\theta(r)}{r} ~dr \leq \eta_{dr}/5C_5 \right\}, 
	\end{equation} 
	and we choose 
	\begin{equation}\label{def:rc}
		r_c = \min\left\{\tilde{r}_c, 4\tau_c, \frac{\eta_{dr}}{5C_6} \right\}. 
	\end{equation} We also choose $\delta_0$ (depending on $d, \Lambda, \epsilon, \rho$) so that 
	\begin{equation}\label{def:delta0}
		\delta_0 = \min\left\{ \frac{\eta_{dr}}{5C_3}, \frac{\eta_{dr}}{5C_4} \right\} 
	\end{equation}  and choose 
	\begin{equation}\label{def:deltain}
		\delta_{in} = \min \left\{ \frac{\delta_0}{2}, \frac{\eta_{dr}}{5C_7} \right\}. 
	\end{equation} 
	This way, we guarantee
	\[ \I_j \leq \frac{\eta_{dr}}{5} s^{d-2}, \quad \text{ for every } j=1, \cdots, 5. \] 
	Therefore
	\[		
		\int_{B_s(q)} \int_0^s \left|\beta_{\bar\mu}^{d-2}(p,r) \right|^2 \frac{dr}{r} ~d\bar\mu(p) \leq 
		\eta_{dr} s^{d-2}.
	\]
	Therefore by applying the Discrete Reifenberg Theorem \ref{thm:Reifenberg} to the ball $B_{2\bar t}(X)$ and the Radon measure $\bar\mu$, we obtain
	\[ \mu_{\bar t}(B_{2\bar t}(X)) = \omega_{d-2} \sum_{Y\in \spt\bar\mu} (\tilde{r}_Y)^{d-2} \leq \omega_{d-2} C_{dr} (2\bar t)^{d-2} \leq C'_V \cdot (\bar t)^{d-2}. \]
	This finishes the proof of the claim \eqref{claim:pkest}. This in turn finishes the proof of the packing estimate \eqref{eq:pkestp} with $C_V = c'_d C'_V$, where $c'_d$ is some dimensional constant to account for a covering of $B_{2r_*}(0)$ by almost-disjoint balls of radius $r_*/6$.
	
\end{proof}

By Lemma \ref{lm:covering} and a purely geometric argument, we get
\begin{lemma}\label{lm:covering2}
	Let $R, \Lambda, \epsilon >0 $ be fixed. There exist $\delta, r_{c}>0$ such that the following holds. 
	For any $(u, D) \in \mathfrak{H}(R, \Lambda) $, any radii $r_0, r_*$ satisfying $0<r_0<r_* \leq r_{c}$, and any subset $\St \subset \widetilde{\Ct}_{r_0}(u)\cap \mathcal{N}(u) \cap B_{r_c \tilde{\theta}(r_c) }(\pD)$, let $\Lambda_* := \sup_{X\in \St \cap B_{2r_*}(0)} N_X(r_*)$. There exists a finite covering of $\St \cap B_{r_*}(0)$ such that
	\[
		\St \cap B_{r_*}(0) \subset \bigcup_{X\in \mathscr{C}} B_{r_X}(X), \quad \text{ with } r_X \geq r_0; 
	\]
	\begin{equation}
		B_{r_X/5}(X) \cap B_{r_{X'}/5}(X') = \emptyset, \quad \text{ for every } X, X' \in \CC \text{ distinct};
	\end{equation}
	\begin{equation}\label{eq:pkest2}
		\sum_{X\in \mathscr{C}} r_X^{d-2} \leq C_1^2 C_V \cdot r_*^{d-2}.
	\end{equation}
	Here both $C_1, C_V>0$ are dimensional constants.
	Moreover, for each center point $X\in \mathscr{C}$, one of the following holds:
	\begin{enumerate}[label=(\roman*)]
		\item \underline{terminal ball:} $r_X= r_0$;
		\item \label{alt:freqdrop} \underline{definite frequency drop}: $r_X> r_0$, and $N_Y(r_X/10) < \Lambda_* -\delta$ for any $Y \in \St \cap B_{2r_X}(X)$.
	\end{enumerate}
\end{lemma}
\begin{proof}
	Since the proof is a standard argument, we only sketch the idea here and in particular, state how we choose the value $\rho$.
	It is clear that we only need to improve on balls $B_{r_X}(X)$ which fall the alternative \ref{item:smalldim} of Lemma \ref{lm:covering}. Firstly, we cover $\St \cap B_{r_X}(X) \setminus B_{\rho r_X}(F_X)$ by balls of radius $\rho r_X$, then each ball falls into alternative \ref{item:deffreqdrop} in Lemma \ref{lm:covering}, i.e. we have a $\delta$-frequency drop. Secondly we cover the set $\St \cap B_{r_X}(X) \cap B_{\rho r_X}(F_X)$ by balls of radius $\hat{r}_Y:= \rho r_X$:
	\[ \St \cap B_{r_X}(X) \cap B_{\rho r_X}(F_X) \subset \bigcup_{Y \in \CC_s(X) } B_{\hat{r}_Y}(Y), \quad \text{ with } Y\in \St \cap B_{r_X}(X) \cap B_{\rho r_X}(F_X) \subset \St \cap B_{\frac{6\rho r_X}{5}}(V_X) \]
	where $V_X$ is a $(d-3)$-dimensional affine subspace. Moreover
	\begin{equation}
		B_{\hat{r}_Y/5}(Y) \cap B_{\hat{r}_{Y'}/5}(Y') = \emptyset, \quad \text{ for every } Y, Y' \in \CC_s(X) \text{ distinct}.
	\end{equation} 
	Hence
	\[ \# \CC_s(X) \leq C_d \frac{\left(\frac{6\rho r_X}{5} \right)^3 \cdot (r_X)^{d-3} }{\left(\frac{\rho r_X}{5} \right)^d } \leq C'_d \cdot \rho^{3-d}. \]
	Therefore
	\begin{align*}
		\sum_{Y\in \CC_s(X)} (\hat{r}_Y)^{d-2} = (\rho r_X)^{d-2} \cdot \# \CC_s(X) \leq C'_d \cdot \rho (r_X)^{d-2}.
	\end{align*}
	By choosing $\rho$ so that 
	\begin{equation}\label{eq:choicerho}
		C'_d \cdot \rho <1,
	\end{equation}  this subdivision do not increase the total packing measure, and we repeat this argument until all the balls are either terminal (i.e. it has radius $r_0$) or have definite frequency drop.
	Eventually, we get the packing estimate \eqref{eq:pkest2} with constant $C_1 \cdot C_1 C_V = C_1^2 C_V$, where $C_V$ is the constant in Lemma \ref{lm:covering} and $C_1 = 10^d \rho^{-2}$ for the $\rho$ already chosen.
\end{proof}

We summarize the order in which we choose the constants. We always fix the dimension $d$, the constants $R, \Lambda>0$ and $\epsilon>0$.
\begin{itemize}
	\item We determine the value of the purely dimensional constant $C_V$ in Lemma \ref{lm:covering}, and the value $C_b$ in Theorem \ref{thm:L2appx} (recall that its value is independent of $\delta_{in}$);
	\item we fix the value of $\rho$ to satisfy \eqref{eq:choicerho} in Lemma \ref{lm:covering2};
	\item with $\rho$ and $C_b$ fixed, we determine the values of $C_1$ and $\tau_c, \delta_0, \delta_{in}$ in Lemma \ref{lm:covering}, by \eqref{eq:countep1}, \eqref{def:tauc}, \eqref{def:delta0} and \eqref{def:deltain}, respectively;
	\item with $\rho$ and $\delta_0$ fixed, we determine the values of $r_{tn}$, $\delta$ and $\beta$ in Proposition \ref{prop:tn} (using $\tau = \rho/10$);
	\item with $\rho$ and $\delta_{in}$ fixed, we determine the value of $r_{in}$ in Lemma \ref{lm:spvarin_far}, and determine the value of $r_b$ in Theorem \ref{thm:L2appx};
	\item with $\beta$ fixed, by the proof of \eqref{eq:tn} as well as Lemma \ref{lm:sym}, we determine the value of $\alpha_0$ in the definition of $\Ct_r(u)$ in \eqref{def:Cru};
	\item finally, with $r_{tn}, r_{in}, r_b, \tau_c$ chosen, we determine the value of $r_c$ by \eqref{def:rct} and \eqref{def:rc}.
\end{itemize}

\section{Proof of Theorem \ref{thm:main}}
The proof of the main theorem \ref{thm:main} follows easily from the covering lemma (Lemma \ref{lm:covering2}) in Section \ref{sec:covering}.
Let $\St:= \widetilde{\Ct}_{r_0}(u) \cap \mathcal{N}(u) \cap B_{r_c \tilde{\theta}(r_c)}(\pD)$, $X_0 \in \overline{D} \cap B_{\frac{R}{10}}(0) $, and 
\[ \Lambda_* := \sup_{X \in \St \cap B_{2r_*}(X_0) } N_X(r_*). \]
By Lemma \ref{lm:covering2}, we can cover $\St \cap B_{r_*}(X_0)$ by a collection of balls $\{ B_{r_X}(X)\}_{X\in \mathcal{C}^{(1)}}$, such that
\begin{itemize}
	\item $X\in \St \cap B_{r_*}(X_0)$ and $B_{r_X/5}(X) \cap B_{r_{X'}/5}(X') = \emptyset $ for every $X, X' \in \mathcal{C}^{(1)}$;
	\item either $r_X = r_0$, or $r_X \in ( r_0, r_*]$ and
		\begin{equation}\label{eq:maindrop}
			N_Y(r_X/10) < \Lambda_* - \delta, \quad \text{ for every } Y \in \St \cap B_{2r_X}(X); 
		\end{equation} 
	\item and
		\begin{equation}\label{eq:mainpacking}
			\sum_{X\in \mathcal{C}^{(1)}} r_X^{d-2} \leq C_1^2 C_V \cdot r_*^{d-2}. 
		\end{equation} 
\end{itemize} 
For each ball $B_{r_X}(X)$ in the covering such that $r_X>r_0$, we first cover $\St \cap B_{r_X}(X)$ by finitely many balls centered at $Y\in \St \cap B_{r_X}(X)$ with radius $r_Y = r_X/10$, and label the collection of centers $\mathcal{C}_0^{(1)}(X)$. Clearly
\begin{equation}\label{tmp:packing2}
	\sum_{Y\in \mathcal{C}_0^{(1)}(X)} r_Y^{d-2} \leq \#(\mathcal{C}(X)) \cdot \left( \frac{r_X}{10} \right)^{d-2} \leq C_d r_X^{d-2}. 
\end{equation} 
For each $Z\in \St \cap B_{2r_Y}(Y) \subset \St \cap B_{2r_X}(X)$, by \eqref{eq:maindrop} we have
\[ N_Z(r_Y) = N_Z(r_X/10) < \Lambda_* - \delta. \]
Hence
\[ \sup_{Z\in \St \cap B_{2r_Y}(Y)} N_Z(r_Y) \leq \Lambda_* - \delta.  \]
We apply Lemma \ref{lm:covering2} again, with $\Lambda_*-\delta$ in place of $\Lambda_*$, to find a covering of the ball $B_{r_Y}(Y)$
\[ \St \cap B_{r_Y}(Y) \subset \bigcup_{Z \in \mathcal{C}^{(2)}(Y) } B_{r_Z }(Z), \]
The covering satisfies either $r_Z = r_0$, or $r_Z> r_0$ and
\begin{equation}
	N_W(r_Z/10) < (\Lambda_*- \delta) - \delta, \quad \text{ for every } W\in \St \cap B_{2r_Z}(Z);
\end{equation}
and we have the packing estimate
\begin{equation}\label{tmp:packing3}
	\sum_{Z\in \mathcal{C}^{(2)}(Y)} r_Z^{d-2} \leq C_1^2 C_V \cdot r_Y^{d-2}. 
\end{equation} 
Denote 
\[ \mathcal{C}^{(2)} := \left\{ Z\in \mathcal{C}^{(2)}(Y): Y\in \mathcal{C}_0^{(1)}(X) \text{ and } X\in \mathcal{C}^{(1)} \right\}. \]
Combining \eqref{tmp:packing3}, \eqref{tmp:packing2}, and \eqref{eq:mainpacking}, we get
\begin{align*}
	\sum_{Z\in \mathcal{C}^{(2)}} r_Z^{d-2} \leq \sum_{X\in \mathcal{C}^{(1)}} \sum_{Y\in \mathcal{C}_0^{(1)}(X)} \sum_{Z\in \mathcal{C}^{(2)}(Y)} r_Z^{d-2} & \leq C_1^2 C_V \cdot \sum_{X\in \mathcal{C}^{(1)}} \sum_{Y\in \mathcal{C}_0^{(1)}(X)} r_Y^{d-2} \\
	& \leq C_d C_1^2 C_V \cdot \sum_{X\in \mathcal{C}^{(1)}} r_X^{d-2} \\
	& \leq C_d \left( C_1^2 C_V \right)^2 r_*^{d-2}.
\end{align*}

For each $j=1, \cdots, j_* = \left\lceil \frac{\Lambda_*}{\delta} \right\rceil-1$, we apply the above argument inductively to balls $B_{r_Z}(Z)$ with $r_Z>r_0$, with uniform frequency bound $\Lambda_*-j\delta$. This induction stops after $j_*$
 many steps, since for $j=j_*+1$ the alternative \ref{alt:freqdrop} of Lemma \ref{lm:covering2} can not happen because 
\[ \Lambda_* - \left\lceil \frac{\Lambda_*}{\delta} \right\rceil \delta \leq 0. \]
Moreover, for each $j$ we have the packing estimate
\[ \sum_{Z\in \mathcal{C}^{(j+1)}} r_Z^{d-2} \leq C_d^j \left( C_1^2 C_V \right)^{j+1} r_*^{d-2}. \]
Finally let
\[ \mathcal{C} := \bigcup_{j=1}^{j_*} \mathcal{C}^{(j)}, \]
and we have the packing estimate
\begin{align*}
	\sum_{Z\in \mathcal{C}} r_Z^{d-2} \leq \sum_{j=0}^{j_*} \sum_{Z\in \mathcal{C}^{(j+1)}} r_Z^{d-2} & \leq \sum_{j=0}^{j_*} C_d^j \left( C_1^2 C_V \right)^{j+1} r_*^{d-2} \\
	& \leq C_d^{\left\lceil \frac{\Lambda_*}{\delta} \right\rceil} \left( C_1^2 C_V \right)^{\left\lceil \frac{\Lambda_*}{\delta} \right\rceil + 1} r_*^{d-2} \\
	& \leq: C_p r_*^{d-2},
\end{align*} 
where the constant $C_p$ only depends on $d, \Lambda$ (since $\Lambda_* \leq C(\Lambda)$) and $\delta$ (which in turn depends on $R, \Lambda$ and $\epsilon$). We remark that since we do not know how $\delta$ depends on $\Lambda$, we can not say that $C_p$ depends on $\Lambda$ exponentially.

Since all balls in the covering have the same radius $r_0$, the above packing estimate implies that
\[ \#(\mathcal{C}) r_0^{d-2} \leq \sum_{Z\in \mathcal{C}} r_Z^{d-2} \leq C_p r_*^{d-2}, \]
and thus
\[ \#(\mathcal{C}) \leq C_p \left( \frac{r_*}{r_0} \right)^{d-2}. \]
Therefore for any $X_0 \in \overline{D} \cap B_{\frac{R}{10}}(0)$ and any $r_0 < r_* \leq r_c$, we have the volume estimate
\[ \left|B_{r_0} \left(\widetilde{\Ct}_{r_0}(u) \cap \mathcal{N}(u) \cap B_{r_c \tilde{\theta}(r_c)}(\pD) \cap B_{r_*}(X_0) \right) \right| \leq \#(\mathcal{C}) \cdot r_0^d \leq C_p r_*^{d-2} r_0^{2}. \]
On the other hand, for any interior point $X$ such that $\dist(X, \pD) \geq r_c \tilde{\theta}(r_c)$, its critical scale $r_{cs}(X) \geq r_c$. In other words, the frequency function $N(X, \cdot)$ is monotone increasing on the interval $[0, r_c]$. Therefore the same proof applies to this much simpler case. (This case is similar to that of \cite{NVCS}: Even though the balls in the covering may still intersect the boundary, they are nonetheless well in the monotonic interval and should be considered \textit{interior balls} in the application of the argument.) Finally we conclude that
\[ \left|B_{r_0} \left(\widetilde{\Ct}_{r_0}(u) \cap \mathcal{N}(u) \cap B_{r_*}(X_0) \right) \right| \leq C'_p ~r_*^{d-2} r_0^{2}. \]
In particular, by the definition of $(d-2)$-dimensional Minkowski content, we have
\[ \mathcal{M}^{d-2, *}\left(\widetilde{\Ct}_{r_0}(u) \cap \mathcal{N}(u) \cap B_{r_*}(X_0) \right) \leq C r_*^{d-2}, \]
and
\[ \mathcal{M}^{d-2, *}\left(\widetilde{\Ct}_{r_0}(u) \cap \mathcal{N}(u) \cap B_{\frac{R}{10}}(0) \right) \leq C R^d r_c^{-2} = C(d, R, \Lambda, \epsilon). \]
Since the above bound is independent of $r_0$, we conclude that the singular set satisfies the Minkowski bound
\[ \mathcal{M}^{d-2, *}\left(\mathcal{S}(u) \cap B_{\frac{R}{10}}(0) \right) \leq \mathcal{M}^{d-2, *}\left(\widetilde{\Ct}_{r_0}(u) \cap \mathcal{N}(u) \cap B_{\frac{R}{10}}(0) \right) \leq C(d, R, \Lambda, \epsilon). \]

To prove the rectifiability, we apply a similar argument as in the proof of the packing estimate in Lemma \ref{lm:covering}. But this time we appeal to Theorem \ref{thm:RReifenberg}, in place of Theorem \ref{thm:Reifenberg}, to the $(d-2)$-measurable set $\widetilde{\Ct}_{r_0}(u) \cap \mathcal{N}(u) \cap B_{r_c \tilde{\theta}(r_c)}(\pD) \cap B_{r_*}(X_0) $. The rectifiability of $\mathcal{S}(u)$ follows from \eqref{incl:Ct} and the countable additivity of rectifiable sets.

\end{document}